\documentclass[envcountsame,envcountsect]{LMsvmult}

\usepackage{mathptmx}       % selects Times Roman as basic font
\usepackage{helvet}         % selects Helvetica as sans-serif font
\usepackage{courier}        % selects Courier as typewriter font
\usepackage{makeidx}         % allows index generation
\usepackage{graphicx}        % standard LaTeX graphics tool
\usepackage{multicol}        % used for the two-column index
\usepackage[bottom]{footmisc}% places footnotes at page bottom
\usepackage{url}
\usepackage{latexsym}
\usepackage{amssymb,amsmath}
\usepackage{ascmac}
\usepackage{amsfonts}
\makeindex             % used for the subject index
\usepackage{color}

\newtheorem{thm}{Theorem}[section]

\newtheorem{lem}[thm]{Lemma}
\newtheorem{prop}[thm]{Proposition}

\newtheorem{ex}[thm]{Example}
\newtheorem{rem}[thm]{Remark}
\newtheorem{ass}[thm]{Assumption}
\newtheorem{algo}[thm]{Algorithm}

\newcommand{\mca}{\mathcal{A}}\newcommand{\mcb}{\mathcal{B}}
\newcommand{\mcc}{\mathcal{C}}
\newcommand{\mcf}{\mathcal{F}}
\newcommand{\mcg}{\mathcal{G}}\newcommand{\mch}{\mathcal{H}}
\newcommand{\mci}{\mathcal{I}}\newcommand{\mcj}{\mathcal{J}}
\newcommand{\mcl}{\mathcal{L}}
\newcommand{\mcm}{\mathcal{M}}\newcommand{\mcn}{\mathcal{N}}
\newcommand{\mcp}{\mathcal{P}}

\newcommand{\mcs}{\mathcal{S}}
\newcommand{\mcv}{\mathcal{V}}

\newcommand{\mfp}{\mathfrak{p}}

\newcommand{\mfC}{\mathfrak{C}}

\newcommand{\mbbd}{\mathbb{D}}

\newcommand{\mbbg}{\mathbb{G}}\newcommand{\mbbh}{\mathbb{H}}

\newcommand{\mbbm}{\mathbb{M}}\newcommand{\mbbn}{\mathbb{N}}

\newcommand{\mbbr}{\mathbb{R}}
\newcommand{\mbbs}{\mathbb{S}}

\newcommand{\mbby}{\mathbb{Y}}\newcommand{\mbbz}{\mathbb{Z}}
\newcommand{\mbbrp}{\mathbb{R}_{+}}
\newcommand{\mbbzp}{\mathbb{Z}_{+}}

\newcommand{\mbT}{\mathbf{T}}

\newcommand{\al}{\alpha} \newcommand{\lam}{\lambda} \newcommand{\ep}{\epsilon} 
\newcommand{\vp}{\varphi}  \newcommand{\del}{\delta}
  \newcommand{\sig}{\sigma}
\newcommand{\Lam}{\Lambda} \newcommand{\gam}{\gamma} \newcommand{\Gam}{\Gamma}
\newcommand{\dd}{\Delta^{n}}  \newcommand{\Del}{\Delta}
\newcommand{\Sig}{\Sigma}

\newcommand{\p}{\partial}  
\newcommand{\cil}{\xrightarrow{\mcl}} % <- Convergence in law
\newcommand{\cip}{\xrightarrow{p}} % <- Convergence in probability
 \newcommand{\argmax}{\mathop{\rm argmax}}

\newcommand{\diag}{\mathop{\rm diag}} 
 \newcommand{\sgn}{\mathop{\rm sgn}}

\def\nn{\nonumber}

\def\sumj{\sum_{j=1}^{n}}
\def\var{{\rm var}}

\def\tcm#1{\textcolor{magenta}{#1}}

\def\ds#1{\displaystyle{#1}}

\def\idl{infinitely divisible distribution}
\def\lp{L\'evy process}
\def\lm{L\'evy measure}
\def\ld{L\'evy density}

\def\cadlag{c\`adl\`ag}

\allowdisplaybreaks

\begin{document}

\begingroup

\title*{
Parametric estimation of {\lp es}
%\thanks{This version: \today}
}
\titlerunning{
Parametric estimation of {\lp es}
}
\author{Hiroki Masuda}
\institute{
Hiroki Masuda \at 
Institute of Mathematics for Industry, Kyushu University,
744 Motooka, Nishi-ku, Fukuoka 819-0395, Japan\\
\email{hiroki@imi.kyushu-u.ac.jp}
}

\numberwithin{equation}{section} \smartqed

\maketitle

\abstract{
The main purpose of this chapter is to present some theoretical aspects of parametric estimation of {\lp es} 
based on high-frequency sampling, with a focus on infinite activity pure-jump models. 
Asymptotics for several classes of explicit estimating functions are discussed. 
In addition to the asymptotic normality at several rates of convergence, 
a uniform tail-probability estimate for statistical random fields is given.
As specific cases, we discuss method of moments for the stable {\lp es} in much greater detail, 
with briefly mentioning locally stable {\lp es} too. 
Also discussed is, due to its theoretical importance, 
a brief review of how the classical likelihood approach works or does not, 
beyond the fact that the likelihood function is not explicit.
\\[5mm]
\textbf{AMS Subject Classification 2000:}\\
\begin{tabular}{ll}
Primary: & \quad   60F05, 62F12, 60G51, 60G52.\\
Secondary: & \quad 60G18.
\end{tabular}
\keywords{Asymptotic normality, Convergence of moments, Estimating function, Fisher information matrix, 
L\'evy process, Likelihood inference, (Locally) Stable L\'evy process.}
}

%62Fxx Parametric inference
%62F12 Asymptotic properties of estimators
%60Fxx Limit theorems [See also 28Dxx, 60B12]

%60F05 Central limit and other weak theorems
%60G18 Self-similar processes
%60G51 Processes with independent increments; Levy processes 
%60G52 Stable processes

%%%%%
%%%%%

%\tableofcontents% to be in comment-out for the headings in force.

%%%%%
%%%%%

\section{Introduction}

{\lp es} form the basic class of continuous-time stochastic processes, 
serving as building blocks to make up more general models, 
such as a solution to a L\'evy driven stochastic differential equation. 
An estimation paradigm with a universal implementable manner is, however, hard or impossible to be available 
because of the diversity of the driving {\lm}, and this has been attracting much interest from statisticians. 
The main objective of this chapter is to present several asymptotic results concerning 
parametric estimation of {\lp es} observed at high frequency. 
Explicit case studies will be presented along topics.

\medskip

Throughout, we are given an underlying probability space $(\Omega,\mcf,P)$ 
endowed with a real-valued L\'evy process $X=(X_{t})_{t\in\mbbrp}$. 
The expectation operator is denoted by $E$. 
Let $\vp_{\xi}$ and $\mcl(\xi)$ stand for the characteristic function 
and the distribution of a random variable $\xi$, respectively. 
We recall that there is one-to-one distributional correspondence $\mcl(X_{1})=F$ 
between a {\lp} $X$ and an {\idl} $F$ on $\mbbr$. 
The celebrated L\'evy-Khintchine formula for a {\lp} says that 
for each {\lp} there uniquely corresponds a generating triplet $(b,c,\nu)$ 
associated with the truncation function being the identity on $[-1,1]$ and $0$ otherwise:
\begin{equation}
\frac{1}{t}\log\vp_{X_{t}}(u)=ibu-\frac{1}{2}cu^{2}+\int(e^{iuz}-1-iuz\mathbf{1}_{U}(z))\nu(dz),\quad
t\in\mbbrp,
\label{hm:lk}
\end{equation}
where $b\in\mbbr$ is the constant trend, $c\ge 0$ is the variance of the Gaussian part, 
and $\nu$ is the L\'evy measure, namely, a $\sigma$-finite measure on $(\mbbr,\mcb(\mbbr))$ such that 
$\nu(\{0\})=0$ and $\int(|z|^{2}\wedge 1)\nu(dz)<\infty$, and finally 
$\mathbf{1}_{U}$ stands for the indicator function of the set $U:=\{z; |z|\le 1\}$. 
We may always deal with c\`adl\`ag (right-continuous and having left-hand side limits) 
modifications of $X$, implying that $X$ a.s. takes values in the space 
$\mbbd(\mbT):=\{x:\mbT\to\mbbr;~\text{$t\mapsto x_{t}$ is c\`adl\`ag.}\}$, $\mbT\subset\mbbrp$, 
equipped with the Skorokhod topology; hence we may always deal with separable {\lp}, 
so that, e.g., probabilities of union or intersection of seemingly uncountable corrections of events can be defined. 
The generating triplet uniquely determines the law of $X$ on the space $\mbbd(\mbT)$. 
As usual, we will denote by $\Del X_{t}:=X_{t}-\lim_{s\uparrow t}X_{s}$ the (directed) jump size of $X$ at time $t$. 
If $\nu(\mbbr)=\infty$ (resp. $\nu(\mbbr)<\infty$), 
then $X$ is said to be infinite-activity (resp. finite-activity), meaning that 
sample paths of $X$ a.s. have infinitely (resp. finitely) many jumps over each finite time interval. 
We refer to \cite{Ber96}, \cite{JacShi03}, and \cite{Sat99} 
for systematic and comprehensive accounts of {\lp es}.

We are concerned with parametric estimation of $\mcl(X)$. 
We denote by $\theta\in\Theta$ the finite-dimensional parameter of interest, 
and by $(P_{\theta}; \theta\in\Theta)$ the family of the induced image measures of $X$; 
in general, there may be nuisance elements of $\mcl(X)$, so that $\theta$ may not completely determine $(b,c,\nu)$. 
Throughout this chapter, we assume that $\Theta$ is a bounded convex domain in $\mbbr^{p}$ unless otherwise mentioned; 
the boundedness may or may not necessary according to each situation, 
but we put it as a standing assumption just for convenience. 
The closure of $\Theta$ will be denoted by $\overline{\Theta}$.

There do exist many {\idl s} admitting a closed-form density, 
and we indeed have a wide variety of $X$ with explicit density of $\mcl(X_{1})$. 
Even so, however, the likelihood function of $\mcl(X_{t})$ for $t\ne 1$ may not be necessarily explicit 
due to the lack of the reproducing property of $\mcl(X_{1})$; 
one such an example is the Student-$t$ {\lp} \cite{Cuf07}, where $\mcl(X_{1})$ is Student-$t$ hence fully explicit, 
but $\mcl(X_{t})$ for $t\ne 1$ is not. Also, as in the case of the stable {\lp es} (see Section \ref{hm:sec_slp}), 
it can happen that the {\lm} is given in a simple closed form 
while the transition density of $\mcl(X_{t})$ is intractable for any $t$.

Although it has been a long time since the rigorous probabilistic structure of L\'evy processes has been clarified, 
we do not have any absolute way to perform statistical estimation for its general class. 
The problem exhibits rather different features and solutions 
according as what the true data-generating triplet is: for example, 
the concrete structure of $\nu$ may essentially affect estimation of the drift $b$; 
and also, coexistence of both diffusion and jump parts can make estimation much more difficult 
than in continuous or purely-discontinuous cases.

More important from a statistical viewpoint is the structure of available data, that is to say, 
how much one can observe $X$'s sample path. 
We can single out the following two cases as basic situations in developing a large-sample theory.

\begin{itemize}
\item Having {\it continuous-time data} $(X_{t})_{t\in[0,T]}$ with $T\to\infty$ should be ideal, 
in which case we may estimate some parameters without error, rendering the statistical theory void.

\item Suppose that we observe $X$ at discrete-time points $(t^{n}_{j})_{j=0}^{n}\subset[0,\infty)$ such that
\begin{equation}
0\equiv t^{n}_{0}<t^{n}_{1}<\dots<t^{n}_{n}=:T_{n}
\label{hm:discsamp}
\end{equation}
for each $n\in\mbbn$. 
Then, we will refer to the sampling scheme as {\it low-frequency sampling} if 
the sampling intervals $\dd_{j}t:=t^{n}_{j}-t^{n}_{j-1}$ satisfy that
\begin{equation}
\liminf_{n\to\infty}\min_{1\le j\le n}\dd_{j}t>0,
\label{hm:LF_sampling}
\end{equation}
which entails that $T_{n}\to\infty$. 
In contrast, {\it high-frequency sampling} means that we have
\begin{equation}
h_{n}:=\max_{1\le j\le n}\dd_{j}t\to 0,
\label{hm:HF_sampling}
\end{equation}
as $n\to\infty$, and in this case the terminal sampling time $T_{n}$ may or may not tend to infinity. 
In either case, we are led to consider estimation based on the infinitesimal array of independent random variables 
$(\dd_{j}X)_{j=1}^{n}$, where
\begin{equation}
\dd_{j}X=X_{t^{n}_{j}}-X_{t^{n}_{j-1}}
\nonumber
\end{equation}
denotes the $j$th increments of $X$. 
Asymptotic results can become drastically different from the case of low-frequency sampling; 
in particular, best possible convergence rate of an estimator can differ for each component. 
For brevity, we assume that 
\begin{equation}
\text{$\dd_{j}t\equiv h_{n}$ for $j\le n$}\quad\text{and}\quad\text{$\liminf_{n\to\infty}~T_{n}>0$}
\label{hm:addeqnum-1}
\end{equation}
whenever discrete-time sampling is concerned. 
The equidistance of sampling could be weakened if we render 
$\dd_{1}t,\dots,\dd_{n}t$ asymptotically not so deviating from one another 
with a suitable control of balance between behaviors of $T_{n}$, $h_{n}$, and $\min_{1\le j\le n}\dd_{j}t$.

\end{itemize}

Our main interest is in parametric estimation of pure-jump {\lp es} having some nice explicit features, 
based on high-frequency sampling; the cases of continuous-time data and low-frequency sampling 
will be mentioned only briefly. 
Needless to say, ``high-frequency'' of data in statistical model is a relative matter, 
for there is no universal way to associate model time with actual time; 
one may put one day, one minute, one second, and so on to $t=1$, and more concretely, 
daily data over three years can be as high-frequency as one thousand intraday data over one day. 

\medskip

Here is the outline of this chapter. 
Section \ref{hm:sec_la} overviews some basic aspects of the maximum-likelihood approach 
for both continuous-time and discrete-time data. 
When attempting parametric inference for the unknown parameter $\theta\in\Theta$ 
based on a realization of $(X_{t})_{t\in\mbT}$, $\mbT\subset\mbbrp$, 
the maximum-likelihood estimator (MLE) is theoretically the first to be looked at, 
although it requires full specification of $P_{\theta}$ 
and may be fragile against model misspecification. 
Since the likelihood function directly depends on $\mcl(X_{t};t\in\mbT)$, 
it takes different forms according as structure of available data. 
Specific case studies given in Section \ref{hm:sec_la} are based on \cite{KawMas11}, \cite{KawMas13}, and \cite{Mas09}.

In Sections \ref{hm:sec_slp}, we will look at the non-Gaussian stable {\lp es} in much greater details. 
Although the stable {\lp es} has the intrinsic scaling property allowing us to 
make several estimates of probabilities and expectations tractable, 
the transition density does not have a closed form except for a few special cases. 
More severely, as long as the joint estimation of the stable-index and the scale parameters are concerned, 
the asymptotic Fisher information matrix will turn out to be singular at {\it any} admissible parameter value. 
Nevertheless, we can provide some practical moment estimators, 
which are asymptotically normally distributed with non-singular asymptotic covariance matrices. 
The contents of this section are based on \cite{Mas09_slp} and \cite{Mas10_proc}. 
In Section \ref{hm:sec_lslp}, we will briefly mention the locally stable {\lp es}, 
a far-reaching extension of the stable {\lp es}.

Section \ref{hm:sec_pldi} presents a somewhat general framework for deducing 
the convergence of moments of scaled $M$-estimators, which plays a crucial role in 
asymptotic analysis concerning the expectation of an estimator-dependent random sequence, 
such as the mean squared prediction and the AIC-like bias correction in model assessment. 
Thanks to the polynomial type large deviation inequality developed in \cite{Yos11}, 
we will give a set of easy sufficient conditions for the uniform tail-probability estimate 
for a class of statistical random fields associated with possibly multi-scaling $M$-estimation.

Finally, we conclude in Section \ref{hm:sec_cr} with a few remarks on related issues.

%%%%%%

\medskip

Throughout this chapter, we will use the following basic notation. 
We denote by $\cil$ and $\cip$ the convergences in law and probability, respectively. 
For a multilinear form $M=\{M^{(i_{1}i_{2}\dots i_{K})}: i_{k}=1,\dots,d_{k}; k=1,\dots,K\}
\in\mbbr^{d_{1}}\otimes\dots\otimes\mbbr^{d_{K}}$ 
and variables $u_{k}=(u_{k}^{(i)})_{i\le d_{k}}\in\mbbr^{d_{k}}$, we write
\begin{equation}
M[u_{1},\dots,u_{K}]=\sum_{i_{1}=1}^{d_{1}}\dots
\sum_{i_{K}=1}^{d_{K}}M^{(i_{1}i_{2}\dots i_{K})}u_{1}^{(i_{1})}\dots u_{K}^{(i_{K})}.
\nonumber
\end{equation}
Sometimes we will write $a^{(i)}$ (resp. $A^{(ij)}$) for 
the $i$th (resp. $(i,j)$th) entry of a vector $a$ (resp. matrix $A$). 
We denote by $\p_{a}^{m}$ the $m$th partial differential operator with respect to a multidimensional variable $a$, 
and by $I_{r}$ the $r\times r$-identity matrix. For a matrix $A$, $A^{\top}$ denotes its transpose. 
We write $x_{n}\lesssim y_{n}$ if there exists a generic positive constant $C$, possibly varying from line to line, 
such that $x_{n}\le Cy_{n}$ for every $n$ large enough. 
We also write $f(\cdot)\sim g(\cdot)$ for two deterministic functions $f$ and $g$ 
if the ratio $f/g$ tends to $1$. 
The map $x\mapsto\mathrm{sgn}(x)$ takes values $1,0,-1$ according as $x>0$, $=0$, $<0$, respectively. 
Given two measures $\mu_{1}$ and $\mu_{2}$ on some measurable space, we write $\mu_{1}\sim\mu_{2}$ if they are equivalent, 
i.e., if they have the same null sets. 
Finally, we will write $N_{p}(\mu,\Sig)$ and $\phi(\cdot;\mu,\Sig)$ for 
the $p$-variate normal distribution and the probability density 
with mean vector $\mu$ and covariance matrix $\Sig$, respectively.

%%%%%
%%%%%

\section{Classical maximum-likelihood approach}\label{hm:sec_la}

\subsection{Local asymptotics for continuous-time and low-frequency sampling}\label{hm:sec_conti-lf}

\subsubsection{Continuous-time data}\label{hm:sec_conti.obs}

We denote by $P_{\theta}^{T}$ the restriction of $P_{\theta}$ to $\mcf^{X}_{T}$, 
where $(\mcf^{X}_{t})_{t\in\mbbrp}$ is the natural filtration of $X$, 
namely, the smallest $\sig$-field to which $X$ is adapted. 
The asymptotics here is taken for $T\to\infty$. 
No one doubts this situation is ``ideal''; 
in particular, we can completely distinguish the continuous part and possibly infinite-activity jump part. 
Although practically irrelevant and far from being realistic, 
the statistical theory based on continuous-time data is fruitful in its own right 
and enables us to get some insight into what we can do best for estimating $\mcl(X)$. 
In particular, we will see that continuous-time data $(X_{t})_{t\in[0,T]}$ 
may allow us to pinpoint some (not necessarily all!) parameter components ``path-wise'', 
so that no statistics is required. 
We refer to the review article \cite[Sections 2.4 and 4.1]{Jac09} for related discussions.

%%%

\medskip

We need a criterion for the equivalence between $P^{T}_{\theta}$ and $P^{T}_{\theta'}$, 
so as to make the likelihood ratio (the Radon-Nikodym derivative) $dP^{T}_{\theta'}/dP^{T}_{\theta}$ well-defined. 
The equivalence can be effectively characterized 
in terms of the generating triplet, say $(b_{\theta},c_{\theta},\nu_{\theta})$ (recall \eqref{hm:lk}):

\begin{thm}
Given any $T>0$ and $\theta,\theta'\in\Theta$, we have $P_{\theta}^{T}\sim P_{\theta'}^{T}$ 
if and only if both of the following conditions hold.
\begin{description}
\item[(a)] $c_{\theta}=c_{\theta'}$.
\item[(b)] $\nu_{\theta}\sim\nu_{\theta'}$ with the function
\begin{equation}
\zeta(z;\theta,\theta'):=\frac{d\nu_{\theta'}}{d\nu_{\theta}}(z)
\label{hm:zeta_def}
\end{equation}
satisfying that 
\begin{itemize}
\item $\displaystyle{b_{\theta'}=b_{\theta}+\int_{|z|\le 1}z\{\zeta(z;\theta,\theta')-1\}
\nu_{\theta}(dz)+\gam\sqrt{c_{\theta}}}$ for some $\gam\in\mbbr$,
\item $\displaystyle{\int\left(1-\sqrt{\zeta(z;\theta,\theta')}\right)^{2}
\nu_{\theta}(dz)<\infty}$.
\end{itemize}
\end{description}
\label{hm:thm_coac}
\end{thm}

If $c_{\theta}=c_{\theta'}>0$, then we may have much wider possible choices for 
$b_{\theta}$, $b_{\theta'}$, $\nu_{\theta}$ and $\nu_{\theta'}$. 
See \cite[Theorem 33.1]{Sat99} for the proof of Theorem \ref{hm:thm_coac}; 
see also \cite[Theorem 4.1]{AkrJoh81}, \cite[Theorem IV.4.32]{JacShi03}, \cite[Theorem 15]{KabLipShi80}, and \cite{Sat00}.

\medskip

When the absolute continuity fails, we may identify some parameters without statistical error. 
If $P_{\theta}^{T}$ and $P_{\theta'}^{T}$ for each $T$ are mutually singular whenever $\theta\ne\theta'$, 
then, given a specific value $\theta_{0}$ of $\theta$ we may pathwise determine 
whether or not the true value equals $\theta_{0}$. 

\medskip

\begin{ex}{\rm 
Consider the model $X_{t}=bt+\sig J_{t}$ with $\theta=(b,\sig,\beta)\in\mbbr\times(0,\infty)\times(0,2]$, 
where $J$ is the $\beta$-stable {\lp} with $\log\vp_{J_{t}}(u)=-|u|^{\beta}$. 
Especially if $\beta\in(0,2)$, it follows from Theorem \ref{hm:thm_coac} that 
$P^{T}_{\theta}\sim P^{T}_{\theta'}$ if and only if $\theta=\theta'$ 
(the integrals in the conditions of the theorem should be zero), 
rendering that continuous-time data leads to no sensible result for {\it all} the parameters involved. 
Thus likelihood based arguments lose their meaning, 
while the statistical problem still {\it a priori} makes sense. 
In Section \ref{hm:sec_slp} we will look at the stable {\lp} in more detail, 
but let us here illustrate a possible error-free identification 
in the simple setting where $b=0$ and $\beta\in(0,2)$ are known, so that $\theta=\sig>0$. 
Fix any constant $p\in(-1/2,\beta/2)$ so that $E(|J_{1}|^{2p})<\infty$, and write $\mu_{\beta}(p)=E(|J_{1}|^{p})$. 
Having observed a sample path $(X_{t})_{t\in[0,T]}$, we can compute
\begin{equation}
\hat{\theta}_{T,n}:=\left\{\frac{\mu_{\beta}(p)^{-1}}{T^{p/\beta}2^{n(1-p/\beta)}}
\sum_{j=1}^{2^{n}}|X_{jT2^{-n}}-X_{(j-1)T2^{-n}}|^{p}\right\}^{1/p}
\nonumber
\end{equation}
for {\it any} $n$, hence $\lim_{n}\hat{\theta}_{T,n}$ too as soon as it exists. 
Thanks to the scaling property of the stable {\lp}, 
we can see that $E_{\theta}^{T}(\hat{\theta}_{T,n}^{p})=\theta^{p}$ and 
$\var_{\theta}^{T}(\hat{\theta}_{T,n}^{p})\lesssim 2^{-n}$. 
It follows from the Borel-Cantelli argument that 
$\hat{\theta}_{T,n}$ is a strongly consistent estimator: $P^{T}_{\theta}(\lim_{n}\hat{\theta}_{T,n}=\theta)=1$.
}\qed\label{hm:ex_coac_slp}
\end{ex}

\medskip

\begin{ex}{\rm 
The generalized hyperbolic distribution is a very popular {\idl} 
in the fields of turbulence and mathematical finance; a nice systematic review can be found in 
\cite{BibSor03} and \cite{Ham10}. 
The distribution of the generalized hyperbolic {\lp} $X$ is characterized by the five parameters 
$\theta:=(\lam,\al,\beta,\del,\mu)$; in particular, 
$\del>0$ and $\mu\in\mbbr$ represent scale and location, respectively, and 
the corresponding {\ld}, say $g$, admits the asymptotic expansion \cite[Proposition 2.18]{Rai00}
\begin{equation}
g(z)=\frac{\del}{\pi}z^{-2}+\frac{1}{2}\left(\lam+\frac{1}{2}\right)|z|^{-1}+\frac{\beta\del}{2}z^{-1}
+o(|z|^{-1}),\qquad |z|\to 0.
\nonumber
\end{equation}
By means of Theorem \ref{hm:thm_coac}, \cite[Sections 2.5 and 2.6]{Rai00} proved that 
$P^{T}_{\theta}\sim P^{T}_{\theta'}$ if and only if $\del=\del'$ and $\mu=\mu'$. 
As mentioned before, continuous-time sample allows us to distinguish all jump times and jump sizes, 
hence for each $n$ we can identify all $t$ such that $|\Del X_{t}|\ge 1/n$. 
Also proved in \cite{Rai00} is that the statistics
\begin{align}
\hat{\del}_{T,n}&:=\frac{\pi}{nT}\sharp\left\{s\le t;~\Del X_{s}\ge 1/n\right\},
\nn\\
\hat{\mu}_{T,n}&:=\frac{1}{T}\left(X_{T}-\sum_{0<s\le T}\Del X_{s}\mathbf{1}_{[1/n,\infty)}(|\Del X_{s}|)\right)
\nonumber
\end{align}
are strongly consistent estimators of $\del$ and $\mu$, respectively, as $n\to\infty$. 
A continuous-time data allows us to compute $\del=\lim_{n}\hat{\del}_{T,n}$ and $\mu=\lim_{n}\hat{\mu}_{T,n}$.
It is possible to see that $X$ is an example of the locally Cauchy {\lp} in the sense that 
$\mcl(h^{-1}X_{h})$ weakly tends to the Cauchy distribution as $h\to 0$ 
(see Section \ref{hm:sec_lslp} for brief remarks on locally stable {\lp es}); 
in Section \ref{hm:sec_niglp}, we will look at this point in more detail for the normal-inverse Gaussian {\lp}, 
the special case where $\lam=-1/2$.
}\qed\label{hm:ex_coac_gh}
\end{ex}

\medskip

Fix a $\theta\in\Theta$, assume that $P^{T}_{\theta}\sim P^{T}_{\theta'}$ for each $\theta'\ne\theta$, and let 
\begin{equation}
A_{T}(\theta)=\diag\{a_{1T}(\theta),\dots,a_{pT}(\theta)\}
\nonumber
\end{equation}
be a non-random positive definite diagonal matrices such that $a_{kT}(\theta)\to 0$ as $T\to\infty$ 
for $k=1,\dots, p$, 
and $\mci(\theta)\in\mbbr^{p}\otimes\mbbr^{p}$ a non-random nonnegative definite symmetric matrix. 
Let $u\in\mbbr^{p}$ and
\begin{equation}
\theta_{T}=\theta_{T}(u):=\theta+A_{T}(\theta)u.
\nonumber
\end{equation}
When $T\to\infty$ we may assume that $\theta_{T}\in\Theta$. 
The family of probability measures $(P_{\theta}^{T};\theta\in\Theta, T>0)$ is said to satisfy 
{\it the local asymptotic normality (LAN)} at $\theta$ with rate $A_{n}(\theta)$ and 
(constant) Fisher information matrix $\mci(\theta)$, if for each $u$ the stochastic expansion
\begin{equation}
\log\frac{dP^{T}_{\theta_{T}}}{dP^{T}_{\theta}}(u)
=\Del_{T}(\theta)[u]-\frac{1}{2}\mci(\theta)[u,u]+o_{p}(1)
\label{hm:LAN_def}
\end{equation}
holds under $P_{\theta}$, where $\Del_{T}(\theta)\cil N_{p}(0,\mci(\theta))$; 
a nice concise exposition of interpretation of the LAN as the weak convergence to a Gaussian experiment 
can be found in \cite[Chapter 7]{vdV98}.

From a decision theoretic point of view, the LAN is of dominant importance in asymptotic statistics. 
If we have the LAN, the notion of asymptotic optimality in regular statistical estimation and testing hypotheses come into effect, 
and the asymptotic optimality is described in term of the sequence $\Del_{T}(\theta)$ up to deterministic factors. 
Especially, the matrix $A_{T}(\theta)$ corresponds to the maximal (multiscale) rate 
at which we can infer the true value of $\theta$. 
We here recall that an estimator $\hat{\theta}_{T}$ of $\theta$ is called regular if for each $u$ the distribution 
$A_{T}(\theta)^{-1}(\hat{\theta}_{T}-\theta)$ weakly converges along $(P^{T}_{\theta_{T}})$ 
to some distribution $\Pi(\theta)$ free of $u$. 
The celebrated Haj\'ek-Inagaki convolution theorem (\cite{Haj70} %\cite{Haj72} 
and \cite{Ina70}) tells us that $\Pi(\theta)=N_{p}(0,\mci(\theta)^{-1})\ast\mu$ for some distribution $\mu$, 
based on which we can deduce the asymptotically maximal concentration property: 
for any convex set $C\subset\mbbr^{p}$ symmetric around the origin 
and any regular estimator $\hat{t}_{n}$ of $\theta$ we have
\begin{equation}
\limsup_{T\to\infty}P_{\theta}\left\{A_{T}(\theta)^{-1}(\hat{t}_{n}-\theta)\in C\right\}
\le\int_{C}\phi\left(z;0,\mci(\theta)^{-1}\right)dz.
\label{hm:mc_ineq}
\end{equation}
Moreover, we have
\begin{equation}
\liminf_{T\to\infty}\var_{\theta}\left\{A_{T}(\theta)^{-1}(\hat{t}_{n}-\theta)\right\}\ge \mci(\theta)^{-1}
\label{hm:mincovmat}
\end{equation}
in the matrix sense; hence, $\mci(\theta)^{-1}$ is the minimal possible asymptotic covariance matrix, 
and if $A_{T}(\theta)^{-1}(\hat{t}_{n}-\theta)\cil N_{p}(0,\mcv(\theta))$ under $P_{\theta}$, 
then $\mcv(\theta)-\mci(\theta)^{-1}$ should be non-negative definite.

It is also possible for several kinds of tests to construct a locally asymptotically optimal test function. 
We refer to \cite{LeCYan00}, \cite{Str85} and \cite{vdV98} 
for more details of what we can benefit from the LAN theory in testing hypothesis. 

We should note that, in order to apply the general asymptotic optimality theory based on the LAN, 
the matrix $\mci(\theta)$ should be positive definite over $\Theta$; if not, the LAN may not be of much help. 
It will turn out that in our framework 
the singularity of$\mci(\theta)$ will very naturally occur for every $\theta\in\Theta$ 
(see Sections \ref{hm:sec_mxlplan} and \ref{hm:sec_sslp_LAN}). 

The LAN for continuously observed pure-jump {\lp es} was proved in \cite[Theorem 5.1]{AkrJoh81} 
(see also \cite{Lus92} and \cite{Sor91} for related general results 
concerning continuously observed multidimensional models containing non-null diffusion part). 
To state the result, let $\mu(dt,dz):=\sum_{s>0}\mathbf{1}_{\{0\}^{c}}(\Delta X_{s})\ep_{(s,\Delta X_{s})}(dt,dz)$ 
denote the random measure of jumps associated with $X$ (cf. \cite[II \S 1b]{JacShi03}), 
and $\tilde{\mu}_{\theta}(dt,dz):=\mu(dt,dz)-\nu_{\theta}(dz)dt$ its compensated version under $P_{\theta}$.

\begin{thm}
Let $X$ be a {\lp} having the generating triplet $(b_{\theta},0,\nu_{\theta})$. 
Assume that $P^{T}_{\theta}\sim P^{T}_{\theta'}$ for $\theta\ne\theta'$ 
and that there exists an $\mbbr^{p}$-valued measurable function $k$ on $\mbbr\times\Theta$ for which,
\begin{description}
\item[(a)] The Fisher information matrix
\begin{equation}
\mci(\theta):=4\int k^{\otimes 2}(z;\theta)\nu_{\theta}(dz)
\nonumber
\end{equation} 
is finite and positive definite for each $\theta\in\Theta$.

\item[(b)] The following convergences hold (recall \eqref{hm:zeta_def}):
\begin{itemize}
\item $\displaystyle{\int |u|^{-2}\left\{\sqrt{\zeta(z;\theta,\theta+u)}-1-k(z;\theta)[u]
\right\}^{2}\nu_{\theta}(dz)\to 0}$ as $|u|\to 0$,

\item $\displaystyle{\int |u|^{-2}\left|\left\{\sqrt{\zeta(z;\theta,\theta+u)}-1\right\}^{2}
-\left\{k(z;\theta)[u]\right\}^{2}\right|\nu_{\theta}(dz)\to 0}$ as $|u|\to 0$,

\item For each $|u|\ne 0$, we have as $T\to\infty$
\begin{equation}
T\int\left\{\sqrt{\zeta(z;\theta,\theta+T^{-1/2}u)}-1\right\}^{2}\nu_{\theta}(dz)
\to\int\left\{k(z;\theta)[u]\right\}^{2}\nu_{\theta}(dz),
\nonumber
\end{equation}

\end{itemize}

\end{description}
Then, the stochastic expansion \eqref{hm:LAN_def} holds at $\theta$ with $A_{T}(\theta)=T^{-1/2}$ and
\begin{equation}
\Del_{T}(\theta)=\frac{2}{\sqrt{T}}
\int_{0}^{T}\!\!\int k(z;\theta)
\tilde{\mu}_{\theta}(dt,dz).
\nonumber
\end{equation}
\label{hm:thm_AkrJoh81}
\end{thm}

Having the LAN in hand, we then look for an estimator $\hat{\theta}_{T}$ such that 
$A_{T}(\theta)^{-1}(\hat{\theta}_{n}-\theta)\cil N_{p}(0,\mci(\theta)^{-1})$. 
The consistency and asymptotic normality for general pure-jump {\lp es} were studied by \cite{Akr82}.

\medskip

\begin{rem}{\rm 
Given a {\lm} $\nu$ and $\Theta^{\natural}:=\{\theta\in\Theta;~\int_{|z|\ge 1}e^{\theta z}\nu(dz)<\infty\}$ 
with $\Theta^{\natural}\ne\{0\}$, 
we can form a natural exponential family $(P_{\theta};\theta\in\Theta^{\natural})$ generated by $X$ 
based on the newly defined {\lm} $\nu_{\theta}(dz):=e^{\theta z}\nu(dz)$. 
This simple transform makes it possible to do several explicit computations. 
See \cite[Chapter 2 and Section 11.5]{KucSor97} 
for details of the exponential family generated by a general semimartingale models. 
The special form $\nu_{\theta}(dz)$ leads to a handy asymptotically optimal estimator of $\theta$ 
even for discrete-time data \cite[Section 3.1]{Woe01}; 
needless to say, the estimator may not be of direct use when $\nu$ also depends on unknown parameters.
}\qed\end{rem}

%%%%%

\medskip

In the rest of this chapter, we will concentrate on discrete-time data. 
The filtration of the underlying statistical experiments are then much smaller than 
in the case of continuous-time data, and estimation without error is seldom possible.

We will suppose that $\dd_{j}t=h_{n}$ for $j=1,\dots,n$. According to the distributional identity 
\begin{equation}
\mcl(\dd_{j}X)=\mcl(X_{h_{n}}),
\nonumber
\end{equation}
the problem amounts to estimation based on the rowwise independent triangular-array data. 
Even if we know that $\mcl(X_{h_{n}})$ is absolutely continuous with respect to the Lebesgue measure, 
the likelihood function $\theta\mapsto\prod_{j=1}^{n}p_{h_{n}}(\dd_{j}X;\theta)$ 
can be in general described only in terms of the seemingly intractable Fourier-inversion formula:
\begin{equation}
p_{h_{n}}(\dd_{j}X;\theta)=\frac{1}{2\pi}\int\exp(-iu\dd_{j}X)\{\vp_{X_{1}}(u)\}^{h_{n}}du.
\label{hm:fi_formula}
\end{equation}
This annoying fact prevents us from developing a (more or less) universally feasible procedure for 
studying asymptotic behavior of likelihoods; with a positive thought, we could get a nice opportunity of research. 
As will be seen later, there do exist some specific examples where we may get rid of the integral in \eqref{hm:fi_formula} 
to obtain a tractable form, from which we can derive valuable information about asymptotically optimal inference.

\begin{rem}{\rm 
The above situation is somewhat similar to estimation of the discretely observed nonlinear diffusion model
\begin{equation}
dX_{t}=b(X_{t},\mu)dt+\sqrt{c(X_{t},\sig)}dw_{t}.
\nonumber
\end{equation}
Under some regularity conditions on the coefficients 
we can prove the existence of the likelihood (transition density). 
However, its closed-form is seldom known. We then face the statistical problem: 
how can we estimate the parameter $\theta=(\mu,\sig)$ based on $(X_{t^{n}_{j}})_{j=1}^{n}$? 
There is a large literature on this subject, and still lively ongoing in several directions. 
See \cite[Chapters 3 and 5]{Pra99} and \cite{Sor12} for an extensive review of recent developments.
}\qed\end{rem}

%%%

\subsubsection{Low-frequency sampling}

Woerner, in her thesis \cite{Woe01}, systematically studied the LAN with several case studies, 
largely under low-frequency sampling. Suppose that $h_{n}\equiv h>0$, 
so that the situation is a special case of the classical i.i.d. sample
\begin{equation}
Y_{h,j}:=X_{jh}-X_{(j-1)h}
\nonumber
\end{equation}
having the infinitely divisible population $\mcl(X_{h})$. 
The model is then to be estimated at usual rate $\sqrt{n}$. 
In order to deduce the LAN, it is therefore possible to resort to 
the classical criterion based on the differentiability in quadratic mean of 
the family $(P_{\theta};\theta\in\Theta)$:

\begin{thm}
Assume that $\mcl(X_{h})$ admits a parametric density $p_{h}(x;\theta)$ with respect to some measure $\mu$: 
$P_{\theta}(X_{h}\in dx)=p_{h}(x;\theta)\mu(dx)$, and that for $\theta\in\Theta$ 
there exists a measurable $\mbbr^{p}$-valued function $\zeta_{h}(x;\theta)$ such that
\begin{equation}
\int\left\{
\sqrt{p_{h}(x;\theta+u)}-\sqrt{p_{h}(x;\theta)}-\frac{1}{2}\zeta_{h}(x;\theta)[u]\sqrt{p_{h}(x;\theta)}
\right\}^{2}\mu(dx)=o(|u|^{2})
\label{hm:def_dqm}
\end{equation}
as $|u|\to 0$. Then:
\begin{enumerate}
\item $\displaystyle{\int\zeta_{h}(x;\theta)p_{h}(x;\theta)\mu(dx)=0}$;
\item the Fisher information matrix
\begin{equation}
\mci_{h}(\theta):=\int\left\{\zeta_{h}(x;\theta)\right\}^{\otimes 2}p_{h}(x;\theta)\mu(dx)
\nonumber
\end{equation}
exists and is finite;
\item for each $u$, we have the stochastic expansion
\begin{equation}
\sumj\log\frac{p_{h}(Y_{h,j};\theta+u/\sqrt{n})}{p_{h}(Y_{h,j};\theta)}
=\frac{1}{\sqrt{n}}\sumj \zeta_{h}(Y_{h,j};\theta)[u]-\frac{1}{2}\mci_{h}(\theta)[u,u]+o_{p}(1)
\label{hm:LAN_def+1}
\end{equation}
under $P_{\theta}$, where $n^{-1/2}\sumj \zeta_{h}(Y_{h,j};\theta)\cil N_{p}(0,\mci_{h}(\theta))$.
\end{enumerate}
\label{hm:thm_vdV_LAN}
\end{thm}

See \cite[Theorem 7.2]{vdV98} for the proof of Theorem \ref{hm:thm_vdV_LAN}. 
For \eqref{hm:def_dqm} to hold, it is sufficient that:
\begin{itemize}
\item for each $x$, the nonnegative map $\theta\mapsto s_{h}(x;\theta):=\sqrt{p_{h}(x;\theta)}$ 
is of class $\mcc^{1}(\Theta)$;
\item and the matrix 
$\displaystyle{\int\left\{(\p_{\theta}p_{h}/p_{h})(x;\theta)\right\}^{\otimes 2}p_{h}(x;\theta)\mu(dx)}$ 
is well-defined and continuous as a function of $\theta$.
\end{itemize}
Indeed, we can then apply Scheff\'e type argument to 
deduce \eqref{hm:def_dqm} with $\zeta_{h}=\p_{\theta}p_{h}/p_{h}$, 
which may be defined to be $0$ when $p_{h}=0$. See \cite[Lemma 7.6]{vdV98} for details.

\medskip

What should be noted here is the dependence of the Fisher information matrix on the sampling step size $h$, 
which may clarify how estimation of each component of $\theta$ is affected by $h$. 
If we can let $h=h_{n}$ vary as $n$ increases in such a way that 
$r_{n}^{-1}\mci_{h_{n}}(\theta)^{(jj)}\to\mci_{0}(\theta)^{(jj)}>0$ for some positive sequence $(r_{n})$, 
then the $\sqrt{r_{n}n}$-consistency for the MLE of $\theta^{(j)}$ can be expected. 
The high-frequency sampling scheme corresponds to such a situation, where we will see later that 
both $r_{n}\to 0$ and $r_{n}\to\infty$ may occur, depending on the concrete structure of the underlying {\lp}. 
In Theorem \ref{hm:thm_lan_ds} below, 
we will present a unified treatment of low- and high-frequency sampling schemes for proving LAN 
under somewhat more restrictive conditions involving the second derivative of 
$\theta\mapsto\log p_{h_{n}}(x;\theta)$.

%%%%%

\subsection{Local asymptotics for high-frequency sampling}

From now on we will concentrate on the equidistant high-frequency sampling scheme; 
recall \eqref{hm:discsamp}, \eqref{hm:HF_sampling}, and \eqref{hm:addeqnum-1}.

\subsubsection{On small-time behavior of increments}

When $h_{n}\to 0$, things become entirely different from the low-frequency sampling. 
The high-frequency sampling is theoretically fruitful, for it allows us to take into account 
approximation of the underlying model structure in small time, providing a somewhat unified picture for asymptotics. 
As was already mentioned, this brings about special phenomena in estimating an underlying continuous-time model. 
In particular, various optimal rates of convergence of regular estimators become available through the LAN. 
A criterion for deducing the LAN in case of high-frequency sampling and univariate $\theta$ 
was proved in \cite[Theorem 1.6]{Woe01}. 
Theorem \ref{hm:thm_lan_ds} below will put similar conditions, 
but importantly, it can deal with cases where optimal rate may be different componentwise.

\medskip

Since each $\dd_{j}X$ vanishes as $h_{n}\to 0$, it is meaningful to clarify 
a transform giving rise to a nontrivial weak limit. 
The simplest yet important one is the location-scale linear transform
\begin{equation}
\dd_{j}X\to\sig_{n}^{-1}(\dd_{j}X-\mu h_{n})
\label{hm:lem_ltst-1-1}
\end{equation}
for some $\mu\in\mbbr$ and $\sig_{n}>0$ with $\sig_{n}\to 0$ as $n\to\infty$. 
In this case the limit is necessarily strictly stable (cf. Section \ref{hm:sec_slp}), 
and moreover, due to \cite[Proposition 1]{BerDon97} much more is true:

\begin{lem}
Assume that $Y$ is a {\lp} in $\mbbr$ 
and that there exist a non-random positive function $\sig:(0,\infty)\to(0,\infty)$ 
and a non-degenerate distribution $F$ (i.e. $F$ is not a Dirac measure) such that 
\begin{equation}
\sig(h)^{-1}Y_{h}\cil F,\qquad h\to 0.
\label{hm:lem_ltst-1}
\end{equation}
Then we have the following.
\begin{enumerate}
\item $\sig$ is regularly varying with index $1/\beta$ 
(i.e. $\sig(uh)/\sig(h)\to u^{1/\beta}$ as $h\to 0$ for each $u>0$) where $\beta\in(0,2]$, 
and $F$ is strictly $\beta$-stable.
\item $\vp_{Y_{h}}\in L^{1}(du)$ for each $h>0$, 
hence in particular $\mcl(Y_{h})$ admits a continuous Lebesgue density, say $p_{h}(\cdot)$.
\item the uniform convergence
\begin{equation}
\sup_{y\in\mbbr}\left|\sig(h)p_{h}(y\sig(h))-\phi(y;\beta)\right|\to 0,\qquad h\to 0,
\nonumber
\end{equation}
is valid, where $\phi(\cdot;\beta)$ denotes the $\beta$-stable density of $F$ given in \eqref{hm:lem_ltst-1}.
\end{enumerate}
\label{hm:lem_ltst}
\end{lem}

Apart from the $\beta$-stable {\lp es}, for which the stable approximation is trivial due to the scaling property, 
several familiar {\lp es} are known to fulfill \eqref{hm:lem_ltst-1} with $\sig(h)=h^{1/\beta}$, 
hence $\sig_{n}=h_{n}^{1/\beta}$ in \eqref{hm:lem_ltst-1-1}. 
Such a {\lp} may be called locally $\beta$-stable, which we will briefly discuss in Section \ref{hm:sec_lslp}.

Some information about small-time asymptotic behaviors of an increment 
both in probability and a.s. can be found in \cite[Chapter 10]{Don07}.

\medskip

\begin{rem}{\rm 
One may wonder what will occur when $\mcl\{\sig(h)^{-1}(X_{h}-\mu h)\}$ is not weakly convergent 
for any $\sig(\cdot)>0$ and $\mu\in\mbbr$. In such cases, the {\lm} $\nu_{\theta}$ 
does not behave as any $\beta$-stable {\lp}, 
and some non-linear transform of $X_{h}$, say $f_{h}(X_{h})$, might be relevant. 
The right $f_{h}$ should be strongly model-dependent, so that 
it may be hard to formulate a general way to find it. 
Nevertheless, there exists a concrete example concerning subordinators; 
recall that a subordinator $X$ is a univariate L\'evy process whose sample path is a.s. nondecreasing, 
and whose general form of the L\'evy-Khintchine formula is given by
\begin{equation}
\frac{1}{t}\log\vp_{X_{t}}(u)=iub+\int_{0}^{\infty}(e^{iuz}-1)\nu(dz)
\nn%\label{hm:subo-cf}
\end{equation}
for some $b\ge 0$ and $\nu$ supported by $\mbbrp$. 
Recently, \cite{BarLopWol12} characterized the class of drift-free ($b=0$) subordinators $X$ for which
\begin{equation}
X_{h}^{-h}\cil\mcp_{\gam},\quad\text{as $h\to 0$},
\label{hm:ex_nlt_subo_lim}
\end{equation}
where $\mcp_{\gam}$ ($\gam>0$) denotes the Pareto distribution corresponding to the density 
$x\mapsto\gam x^{-\gam-1}\mathbf{1}_{[1,\infty)}(x)$. 
For example, \cite{BarLopWol12} proved that the above weak convergence holds if 
$\mcl(X_{1})$ admits a Lebesgue density $p_{1}(x)$ such that
\begin{equation}
\frac{\log p_{1}(x)}{\log x}\to\gam-1,\quad\text{as $x\to 0$.}
\label{hm:ex_nlt_subo}
\end{equation}
Building on \eqref{hm:ex_nlt_subo_lim}, 
one may think of making semiparametric inference based on the array $\{(\dd_{j}X)^{-h_{n}}\}_{j=1}^{n}$, 
leaving the parameters other than $\gam$ unknown; 
a simple and fully explicit example satisfying \eqref{hm:ex_nlt_subo} is the gamma subordinator $X$ 
with the density of $\mcl(X_{h})$ being $p_{h}(x)=\beta^{\gam h}x^{\gam h-1}\exp(-\beta x)/\Gam(\gam h)$, $x>0$. 
We do not pursue this subject further in this chapter, but only make a small remark about simulations: 
it may happen that a $\mcl(X_{h})$-random number is too small to be regarded as non-zero by computer, 
causing a trouble in taking its reciprocal. 
\label{hm:rem_nltrans}
}\qed\end{rem}

%%%%%

\subsubsection{LAN with multi-scaling}

We will assume some regularity conditions. 

\begin{ass}
The support of $\mcl(X_{t})$ does not depend on $t>0$ and $\theta\in\Theta$. 
For each $t>0$ and $\theta\in\Theta$ 
the distribution $\mcl(X_{t})$ under $P_{\theta}$ admits a Lebesgue density $p_{t}(x;\theta)$, 
which is in turn of the class $\mcc^{2}(\Theta)$ for each $x\in\mbbr$ as a function of $\theta$.
\label{hm:lan_assump1_dd}
\end{ass}

The log-likelihood function a.s. exists as the sum of the rowwise independent triangular arrays:
\begin{equation}
\ell_{n}(\theta)=\sum_{j=1}^{n}\log p_{h_{n}}(\dd_{j}X;\theta).
\nonumber
\end{equation}
We will present a criterion for the LAN under discrete-time sampling, 
which is applicable to both low- and high-frequency sampling schemes.

Under Assumption \ref{hm:lan_assump1_dd}, we let
\begin{align}
g_{nj}(\theta)&:=\p_{\theta}\log p_{h_{n}}(\dd_{j}X_{j};\theta),
\nn\\
A_{n}(\theta)&:=\diag\{a_{1n}(\theta),\dots,a_{pn}(\theta)\},
\label{hm:def_An}
\end{align}
where each positive entry $a_{jn}(\theta)\to 0$. 
We further assume the following.

\begin{ass}
The following convergences hold true as $n\to\infty$:
\begin{description}
\item[(a)] $\displaystyle{n\left|A_{n}(\theta)E_{\theta}\{g_{n1}(\theta)\}\right|^{2}\to 0}$;
\item[(b)] $\displaystyle{nE_{\theta}[\{A_{n}(\theta)g_{n1}(\theta)\}^{\otimes 2}]\to\mci(\theta)}$;
\item[(c)] $\displaystyle{n\left\{\sup_{\rho\in D_{n}(a;\theta)}
E_{\rho}\left(|A_{n}(\theta)\p_{\theta}[g_{n1}(\rho)^{\top}]A_{n}(\theta)|^{2}
+|A_{n}(\theta)g_{n1}(\rho)|^{4}\right)\right\}\to 0}$ for every $a>0$, 
where $D_{n}(a;\theta):=\{\rho\in\Theta; |A_{n}(\theta)^{-1}(\rho-\theta)|\le a\}$.
\end{description}
\label{hm:lan_assump2_dd}
\end{ass}

Of course, Assumption \ref{hm:lan_assump2_dd} are partly related to the Lindeberg-Feller central limit theorem. 
Once $A_{n}(\theta)$ is specified and $g_{nj}(\theta)$ is explicit, 
verification of Assumption \ref{hm:lan_assump2_dd} may not be so difficult. 
We note that (c) ensures the Lindeberg condition:
\begin{equation}
\sum_{j=1}^{n}E_{\theta}\left\{|A_{n}(\theta)g_{nj}(\theta)|^{2};
~|A_{n}(\theta)g_{nj}(\theta)|\ge\ep\right\}\to 0
\nonumber
\end{equation}
for every $\ep>0$. 

Let $P^{n}_{\theta}$ denotes the restriction of $P_{\theta}$ to $\sig(X_{t^{n}_{j}};j\le n)$.

\begin{thm}
Under Assumptions \ref{hm:lan_assump1_dd} and \ref{hm:lan_assump2_dd}, 
the family of probability measures $(P^{n}_{\theta};\theta\in\Theta, n\in\mbbn)$ 
satisfies the LAN at $\theta\in\Theta$ with rate $A_{n}(\theta)$ and Fisher information matrix $\mci(\theta)$: 
for each $u$, we have under $P_{\theta}$
\begin{equation}
\ell_{n}(\theta+A_{n}(\theta)u)-\ell_{n}(\theta)
=A_{n}(\theta)\p_{\theta}\ell_{n}(\theta)[u]-\frac{1}{2}\mci(\theta)[u,u]+o_{p}(1),
\label{hm:LAN_def+2}
\end{equation}
with $A_{n}(\theta)\p_{\theta}\ell_{n}(\theta)\cil N_{p}(0,\mci(\theta))$.
\label{hm:thm_lan_ds}
\end{thm}

Theorem \ref{hm:thm_lan_ds} can be proved all the same as in \cite[Section 4.1]{KawMas13}; 
we should note that it is somewhat straightforward to extend Theorem \ref{hm:thm_lan_ds} to deal with ergodic models, 
with the help of limit theorems for mixing random variables and/or martingale limit theorems. 
So far, several explicit examples have been known for which we can apply Theorem \ref{hm:thm_lan_ds}. 
Real difficulty arises when $g_{nj}$ is not explicit, even if its existence can be verified; 
obviously, without restricting the target class of {\lp es} 
it is impossible to deduce any LAN with specific $A_{n}(\theta)$ and $\mci(\theta)$. 
Research in this direction is currently under investigation.

\medskip

The asymptotic orthogonality of parameters (diagonal Fisher information matrix) 
is known to be very useful in statistics; e.g. \cite{CoxRei87} and \cite{JorKnu04}. 
In the high-frequency sampling scheme, we quite naturally encounter the opposite-side phenomenon, namely, 
the determinant of the normalized observed information matrix 
$-A_{n}(\theta)\p_{\theta}^{2}\ell_{n}(\theta)A_{n}(\theta)$ tends in probability to zero 
(so the Fisher information matrix $\mci(\theta)=0$) for {\it every $\theta\in\Theta$}. 
This problem does not seem to be sidestepped 
simply by using an off-diagonal norming $A_{n}(\theta)=\{A_{n}^{kl}(\theta)\}_{k,l}$. 
We will look at some such examples in Sections \ref{hm:sec_mxlplan} and \ref{hm:sec_sslp_LAN}. 
As mentioned before, the LAN itself is not quite meaningful if $\mci(\cdot)\equiv 0$, 
although it reveals which parameters cause the unpleasant asymptotic singularity, 
giving us a caution for adopting the likelihood approach. 
In this case, there would exist no unbiased estimator with finite variance, 
and the possible asymptotic distributions of the maximum likelihood estimators would be no longer normal 
and have infinite-variance (see \cite{LiuBro93} and \cite{StoMar01}). 
%$\det\{A_{n}(\theta)\p_{\theta}^{2}\ell_{n}(\theta)A_{n}(\theta)\}
%=|A_{n}^{11}(\theta)A_{n}^{22}(\theta)-A_{n}^{12}(\theta)A_{n}^{21}(\theta)|^{2}
%[\p^{2}_{\theta_{1}}\ell_{n}(\theta)\p^{2}_{\theta_{2}}\ell_{n}(\theta)-\{\p_{\theta_{1}}\p_{\theta_{2}}\ell_{n}(\theta)\}^{2}]
%=|1-A_{n}^{12}A_{n}^{21}/(A_{n}^{11}A_{n}^{22})|
%[(A_{n}^{11})^{2}\p^{2}_{\theta_{1}}\ell_{n}(\theta)(A_{n}^{22})^{2}\p^{2}_{\theta_{2}}\ell_{n}(\theta)
%-\{(A_{n}^{11}A_{n}^{22})\p_{\theta_{1}}\p_{\theta_{2}}\ell_{n}(\theta)\}^{2}]$...
Nevertheless, it is worth mentioning that we may bypass the non-invertibility of the asymptotic covariance matrix 
at the expense of the optimal rate of convergence, retaining asymptotic normality 
(\cite{Kaw13}, \cite{Mas09_slp} and \cite{Tod13}); 
some examples will be given in Sections \ref{hm:sec_sym.slp_me} and \ref{hm:sec_skewslp_me} for the stable models.

%%%

\subsubsection{Example: Meixner {\lp}}\label{hm:sec_mxlplan}

The Meixner distribution, denoted by ${\rm Meixner}(\alpha,\beta,\delta,\mu)$, 
is infinitely divisible and admits a density 
\begin{equation}\label{hm:basic parametrization}
 x\mapsto \frac{\left(2 \cos (\beta/2)\right)^{2\delta}}{2\pi\alpha
\Gamma(2\delta)}\exp\left[\frac{\beta}{\alpha}(x-\mu)\right]
\left|\Gamma\left(\delta+i\frac{x-\mu}{\alpha}\right)\right|^2,
\quad x\in\mathbb{R}.
\end{equation}
We write $\theta=\left(\alpha,\,\beta,\,\delta,\,\mu\right)\in \Theta$, 
a bounded convex domain whose closure satisfies that
\[
 \Theta^-\subset \left\{\left(\alpha,\,\beta,\,\delta,\,\mu\right)\in
 \mathbb{R}^4;\,\alpha>0,\,|\beta|< \pi,\,\delta>0,\,\mu\in\mathbb{R}\right\}.
\]
The L\'evy measure of ${\rm Meixner}(\alpha,\beta,\delta,\mu)$ admits the explicit Lebesgue density 
\[
 g(z;\theta):=\delta\frac{\exp(\beta z/\alpha)}{z{\rm sinh}(\pi z/\alpha)},\quad z\ne 0.
\]
We refer the reader to \cite{Gri99} and \cite{SchTeu98} for more details of the Meixner distribution.

Let $X$ be a L\'evy process such that $\mcl(X_1)={\rm Meixner}(\alpha,\beta,\delta,\mu)$. 
The characteristic function of $\mathcal{L}(X_t)$ is given by
\[
 \vp_{X_{t}}(u)=e^{iu\mu t}\left(\frac{\cos (\beta/2)}{{\rm cosh}\left((\alpha
 u-i\beta)/2\right)}\right)^{2\delta t},
\]
implying that for each $c>0$ and $t>0$,
\begin{equation}
\mathcal{L}\left(c (X_t-t \mu)\right)={\rm Meixner}(c\alpha,\,\beta,\,t\delta,\,0).
\label{hm:rep}
\end{equation}
For each $n\in\mbbn$, we define the i.i.d. random variables $\ep_{n1},\ep_{n2},\dots$ by
\begin{equation}\label{hm:def of epsilon}
\epsilon_{nj}=\epsilon_{nj}(\alpha,\delta,\mu,h_{n}):=
\frac{\dd_{j}X-h_{n}\mu}{h_{n}\alpha\delta},
\end{equation}
with common distribution $\mcl(\ep_{n1})={\rm Meixner}((h_{n}\delta)^{-1},\beta,h_{n}\delta,0)$. 
We can also see that $\mcl(\epsilon_{n1})$ has mean, variance, skewness and kurtosis, respectively,
\[
 \tan\frac{\beta}{2},\quad \frac{1}{2h_{n}
 \delta \cos^2(\beta/2)},\quad \sin
 \frac{\beta}{2}\sqrt{\frac{2}{h_{n}\delta}},\quad 3+\frac{2-\cos (\beta)}{h_{n}\delta}.
\]
Further, $\mcl(\ep_{n1})$ converges to the standard Cauchy distribution, as $n\to\infty$: 
indeed, for each $u\in\mbbr$
\begin{align*}
\vp_{\ep_{n1}}(u)
&=\left\{\frac{\cos (\beta/2)}{{\rm cosh}\left((u/(h_{n}\delta)-i\beta)/2
\right)}\right\}^{2h_{n}\delta}\\
&=\left\{\frac{2}{e^{u/(2h_{n}\delta)}(1-i\tan (\beta/2))+
e^{-u/(2h_{n}\delta)}(1+i\tan (\beta/2))}\right\}^{2h_{n}\delta}\\
&\sim\left(\frac{2}{1-i\sgn(u)\tan(\beta/2)}\right)^{2h_{n}\delta}e^{-|u|}
\nn\\
&\to e^{-|u|},\quad n\to\infty.
\end{align*}
The Meixner {\lp} possesses the small-time Cauchy property as well as the long-time Gaussianity 
in the functional sense; see \cite{KawMas11} and the references therein.

The log-likelihood function is
\begin{align}
\ell_{n}(\theta)&=\sumj
\bigg\{2h_{n}\delta \log \left(2
\cos\frac{\beta}{2}\right) -\log (2\pi\alpha) -\log\Gamma(2h_{n}\delta)
\nn\\
&{}\qquad+h_{n}\beta\del\ep_{nj}+
\log\left|\Gamma\left(h_{n}\delta(1+i\ep_{nj})\right)\right|^2\bigg\},
\nn%\label{hm:LF}
\end{align}
and we have the following LAN result:

\begin{thm}
If $T_{n}\to\infty$, then we have the LAN for each $\theta\in \Theta$ at rate
\begin{equation}
A_{n}={\rm diag}\left(\frac{1}{\sqrt{n}},\frac{1}{\sqrt{T_{n}}},\frac{1}{\sqrt{n}},\frac{1}{\sqrt{n}}\right)
\nonumber
\end{equation}
and Fisher information matrix
\begin{equation}
\mci(\theta):=\left(
\begin{array}{cccc}
1/(2\alpha^2) & 0 & 1/(2\alpha\delta) & 0 \\
 & \del/\{2\cos^2(\beta/2)\} & 0 & 0 \\
 &  & 1/(2\delta^2) & 0 \\
\text{{\rm sym.}} &  &  & 1/(2\alpha^2\delta^2)
\end{array}\right).
\label{hm:F_info}
\end{equation}
\label{hm:mx_main_thm}
\end{thm}

We omit the proof of Theorem \ref{hm:mx_main_thm}, referring the interested reader to \cite{KawMas11}.

\medskip

The Fisher information matrix (\ref{hm:F_info}) is singular for every $\theta\in\Theta$, 
which is obviously caused solely by the joint maximum-likelihood estimation of $\al$ and $\del$; 
as soon as $\alpha$ or $\delta$ is fixed, 
the resulting $3\times 3$ Fisher information matrix becomes purely diagonal, 
ensuring that the maximum likelihood estimators are asymptotically independent. 
The asymptotic singularity also acts as a practical warning in the maximum likelihood 
estimation for the Meixner {\lp} with a very small $\delta$ 
under low-frequency sampling scheme, since, as seen by \eqref{hm:rep}, 
the parameter $\delta$ and the time $t$ play the same role. 
The form of $\mci(\theta)$ of (\ref{hm:F_info}) is much simpler compared with 
that of the Fisher information matrix in the low-frequency sampling; see \cite[Appendix A]{GriPro09}. 

As we will see in Section \ref{hm:sec_sslp_LAN}, 
the joint maximum-likelihood estimation of the stability index and the scale parameter of the stable {\lp es}
also leads to a constantly singular Fisher information matrix. It can be expected that 
the asymptotic singularity occurs for every L\'evy process satisfying the small-time stable approximation 
and having unknown index $\beta$ (Lemma \ref{hm:lem_ltst}) and scale; 
a discussion on this issue can be found in \cite{Kaw13}. 
In this direction, the case of the Meixner {\lp es} is not directly relevant 
since we beforehand know that the small-time stability index equals one.

\medskip

We may expect from the definition of $\ep_{nj}$ of \eqref{hm:def of epsilon} 
that the asymptotic singularity stems from the non-identifiability between the parameters 
$\al$ and $\del$ in small time; they may be identifiable only in the product form $\al\del$. 
The case of continuous-time data captures this point more directly:

\begin{prop}\label{hm:prop continuous time}
Let $T>0$ and let $\theta_k:=(\alpha_k,\beta_k,\delta_k,\mu_k)\in \Theta$, $k=1,2$.
The probability measures 
$P^{T}_{\theta_{1}}$ and $P^{T}_{\theta_{2}}$ 
are equivalent if and only if $\alpha_1\delta_1=\alpha_2\delta_2$ and $\mu_1=\mu_2$.
\end{prop}

\begin{proof}
Since $g(z;\theta)>0$ for every $z\ne 0$, the function
\begin{equation}
\zeta(\cdot;\theta_{1},\theta_{2}):=\frac{g(\cdot;\theta_{2})}{g(\cdot;\theta_{1})}:~\mbbr\backslash\{0\}\to(0,\infty)
\nonumber
\end{equation}
is well-defined. 
The mean of $\mcl(X_{1})$ is given by $\mu_{0}(\theta):=\mu +\alpha\delta\tan(\beta/2)$, hence
\begin{equation}
\vp_{X_{1}}(u)=\exp\left\{
iu\mu_{0}(\theta)+\int \left(e^{iuz}-1-iuz\right)g(z;\theta)dz
\right\},\quad u\in\mathbb{R}.
\nonumber
\end{equation}
Now, according to Theorem \ref{hm:thm_coac} it suffices to show that the following two conditions 
hold if and only if $\alpha_1\delta_1=\alpha_2\delta_2$ and $\mu_1=\mu_2$:
\begin{description}
\item[{\rm (a)}] $\ds{\int\{1-\sqrt{\zeta(z;\theta_{1},\theta_{2})}\}^{2}g(z;\theta_{1})dz<\infty}$;
\item[{\rm (b)}] $\ds{\mu_{0}(\theta_{2})=\mu_{0}(\theta_{1})
+\int z\{\zeta(z;\theta_{1},\theta_{2})-1\}g(z;\theta_{1})dz}$.
\end{description}
Let us look at the behaviors of the L\'evy density $g(z;\theta)$ near the origin and at infinity. 
By means of the approximation
\begin{equation}
\frac{z}{{\rm sinh}(z)}=1-\frac{z^{2}}{6}+O(z^{4}),\quad|z|\to 0,
\nonumber
\end{equation}
we see that the L\'evy density $g(z;\theta)$ satisfies that
\begin{equation}
g(z;\theta)=\frac{\alpha\delta}{\pi z^{2}}
\left(1+\frac{\beta}{\alpha}z+O(z^{2})\right),\quad |z|\to 0.
\label{hm:g_origin}
\end{equation}
Since $x\mapsto{\rm sinh}(x)$ behaves like $e^{x}/2$ (resp. $-e^{-x}/2$) 
as $x\to\infty$ (resp. $x\to-\infty$), we have
\begin{equation}
g(z;\theta)\sim\left\{
\begin{array}{ll}
2\delta z^{-1}\exp\{-(\pi-\beta)z/\alpha\}, &\quad z\to\infty, \\
2\delta |z|^{-1}\exp\{-(\pi+\beta)|z|/\alpha\}, &\quad z\to-\infty.
\end{array}
\right.
\label{hm:g_infi}
\end{equation}
By (\ref{hm:g_origin}) and (\ref{hm:g_infi}),
\begin{align}
& \left(1-\sqrt{\zeta(z;\theta_1,\theta_{2})}\right)^2g(z;\theta_1)
\nn\\
&=\left(\sqrt{g(z;\theta_{1})}-\sqrt{g(z;\theta_{2})}\right)^2
\nn\\
&\sim\left\{
\begin{array}{ll}
(\pi|z|)^{-2}\Big\{\left(\sqrt{\alpha_1\delta_1}-\sqrt{\alpha_2\delta_2}\right)^{2} & \\
\qquad+\left(\sqrt{\alpha_1\delta_1}-\sqrt{\alpha_2\delta_2}\right)
\left(\beta_1\sqrt{\frac{\delta_1}{\alpha_1}}-\beta_2\sqrt{\frac{\delta_2}{\alpha_2}}\right)z+O(z^{2})
\Big\}, &\quad |z|\to 0, \\
C_+z^{-1}\exp(-q_+z), &\quad z\to\infty, \\
C_-|z|^{-1}\exp(-q_-|z|), &\quad z\to-\infty,
\end{array}
\right.
\nn
\end{align}
for some positive constants $C_{\pm}$ and $q_{\pm}$, depending on $(\theta_1,\theta_2)$. 
Hence (a) holds if and only if $\alpha_{1}\delta_{1}=\alpha_{2}\delta_{2}$, 
which is to be imposed in the rest of this proof.

The condition (b) is equivalent to
\[
\mu_1+\alpha_1\delta_1\tan\frac{\beta_1}{2}-\mu_2-\alpha_2\delta_2\tan\frac{\beta_2}{2}
=\int \left(\delta_1 
\frac{\exp(\beta_1 z/\alpha_1)}{\sinh (\pi z/\alpha_1)}
-\delta_2 \frac{\exp(\beta_2 z/\alpha_2)}{\sinh (\pi z/\alpha_2)}
\right)dz.
\]
In the case $\alpha_{1}\delta_{1}=\alpha_{2}\delta_{2}=:C>0$, the last display can be rewritten as
\[
\mu_1-\mu_2+C\left\{\tan\frac{\beta_1}{2}-\tan\frac{\beta_2}{2}
-\int \left(\frac{\exp(\beta_1 z/\alpha_1)}{\alpha_1\sinh (\pi z/\alpha_1)}
-\frac{\exp(\beta_2 z/\alpha_2)}{\alpha_2\sinh (\pi z/\alpha_2)}
\right)dz\right\}=0.
\]
Denote the $\{\dots\}$ part on the left-hand side by $f(\beta_1,\beta_2)$. 
We show that $f$ is identically zero (given any positive $\al_{1}$ and $\al_{2}$), 
entailing that (b) holds if and only if $\mu_{1}=\mu_{2}$ hence completing the proof. We have
\[
 f(0,0)=\int \left(\frac{1}{\alpha_2 \sinh (\pi z/\alpha_2)}
-\frac{1}{\alpha_1 \sinh (\pi z/\alpha_1)}\right)dz\equiv 0,
\]
since the integrand is odd, continuous in $\mathbb{R}$, and exponentially decreasing as $|z|\to\infty$. 
Using the fact that the variance $\alpha_k^2\delta_k/(2\cos^2(\beta_k/2))$ of
${\rm Meixner}(\alpha_k,\beta_k,\delta_k,\mu_k)$ equals $\int z^{2}g(z;\theta_k)dz$, we get
\[
\frac{1}{\alpha_k^2}\int z~\frac{\exp(\beta_k z/\alpha_k)}{{\rm sinh}(\pi z/\alpha_k)}dz
=\frac{1}{2\cos^{2}(\beta_k/2)}.
\]
Hence
\begin{align*}
\p_{\beta_{1}}f(\beta_1,\beta_2)&=\frac{1}{2(\cos (\beta_1/2))^2}-
\frac{1}{\alpha_1^2}\int_{\mathbb{R}\backslash\{0\}}z~\frac{\exp(\beta_1 z/\alpha_1)}{\sinh (\pi z/\alpha_1)}dz\equiv 0.
\end{align*}
We can deduce that $\p_{\beta_2}f(\beta_1,\beta_2)\equiv 0$ in a similar manner. 
It follows that $f$ is identically zero.
\qed\end{proof}

%%%%%

\subsection{Uniform asymptotic normality of MLE with non-degenerate Fisher information}

\subsubsection{Basic result}

When the Fisher information matrix is non-degenerate, we can go further in an elegant way. 
The contents of this section is essentially a special case of 
Sweeting's general result \cite{Swe80} (also relevant is \cite[Chapter 1, Section 4]{BasSco83}), 
based on which we can provide a simple set of sufficient conditions for 
the asymptotic normality and the asymptotic optimality of the MLE, as well as for the LAN. 
A nice feature of the results is that it is almost enough to look at the uniform asymptotic behavior of 
the normalized observed information matrix having a {\it positive definite} limit (the Fisher information matrix) 
continuous in the parameter, and need not take care of the central limit theorem for the score-function part. 

\medskip

We will assume that the log-likelihood function $\theta\mapsto\ell_{n}(\theta)$ a.s. 
belongs to the class $C^{2}(\Theta)$, 
and we write the score function and the observed information matrix by
\begin{equation}
\mcs_{n}(\theta)=\p_{\theta}\ell_{n}(\theta)\quad\mathrm{and}
\quad\mci_{n}(\theta)=-\p_{\theta}^{2}\ell_{n}(\theta),
\nn%\label{hm:g.ig-si}
\end{equation}
respectively. To state the result we need to introduce some more definitions. 
Let us recall that the convergences in distribution and in probability of random vectors are metrizable. 
We will need their uniform versions. 
Let the symbol $\to_{u}$ stand for the ordinary uniform convergence over each compact subset of $\Theta$. 
For vector-valued random functions $\zeta_{n}(\cdot)$ and $\zeta(\cdot)$ on $\Theta$ 
with each $\zeta_{n}(\theta)$ being $\sig(X_{t^{n}_{j}};j\le n)$-measurable, we write 
$\zeta_{n}(\theta)\cil_{u}\zeta(\theta)$ and $\zeta_{n}(\theta)\cip_{u}\zeta(\theta)$ 
if $d_{\mcl}(\zeta_{n}(\theta),\zeta(\theta);\theta)\to_{u}0$ and 
$d_{p}(\zeta_{n}(\theta),\zeta(\theta);\theta)\to_{u}0$, respectively. 
Here, $d_{\mcl}(\cdot,\cdot;\theta)$ and $d_{p}(\cdot,\cdot;\theta)$ denote 
any metric characterizing $\cil$ and $\cip$ under $P_{\theta}$, respectively: 
for example, we may take
\begin{align}
d_{\mcl}(\zeta_{n}(\theta),\zeta(\theta);\theta)
&=\sup\left\{\left|
E_{\theta}\left\{f(\zeta_{n}(\theta))\right\}-E\left\{f(\zeta(\theta))\right\}
\right|:~\|f\|_{BL}\le 1\right\},
\label{hm:uan_metric_L} \\
d_{p}(\zeta_{n}(\theta),\zeta(\theta);\theta)
&=E_{\theta}\left(\frac{|\zeta_{n}(\theta)-\zeta(\theta)|}{1+|\zeta_{n}(\theta)-\zeta(\theta)|}\right),
\label{hm:uan_metric_p}
\end{align}
where
\begin{equation}
\|f\|_{BL}:=\sup_{x\ne y}\frac{|f(x)-f(y)|}{|x-y|}+\sup_{x}|f(x)|
\nonumber
\end{equation}
is the bounded-Lipschitz norm; 
by the definitions, $\zeta_{n}(\cdot)\cip_{u}\zeta(\cdot)$ implies $\zeta_{n}(\cdot)\cil_{u}\zeta(\cdot)$; 
e.g., \cite[Appendix A.1]{BasSco83}. 
Finally, let $A_{n}(\theta)$ be as in \eqref{hm:def_An}, now satisfying that $a_{jn}(\theta)\to_{u}0$.

Recall that $P^{n}_{\theta}$ stands for the restriction of $P_{\theta}$ to $\sig(X_{t^{n}_{j}};j\le n)$. 
With the above-mentioned notation, we will say that the family of probability measures $(P^{n}_{\theta}; n\in\mbbn)$ is 
{\it uniform LAN (ULAN) in $\Theta$ with rate $A_{n}(\theta)$ and Fisher information $\mci(\theta)$} if 
there exists a non-random function $\mci:\Theta\to\mbbr^{p}\otimes\mbbr^{p}$ 
with $\mci(\theta)$ being positive definite for any $\theta\in\Theta$, such that 
$A_{n}(\theta)\mcs_{n}(\theta)\cil_{u}N_{p}(0,\mci(\theta))$, 
that $A_{n}(\theta)\mci_{n}(\theta)A_{n}(\theta)\cip_{u}\mci(\theta)$, and that
\begin{align}
& \ell_{n}(\theta+A_{n}(\theta)u_{n})-\ell_{n}(\theta) \nn\\
& \qquad{}
-\left(\mcs_{n}(\theta)[A_{n}(\theta)u_{n}]
-\frac{1}{2}\mci_{n}(\theta)[A_{n}(\theta)u_{n},~A_{n}(\theta)u_{n}]\right)
\cip_{u}0
\label{hm:uan_sa+1}
\end{align}
for any non-random bounded sequence $(u_{n})\subset\mbbr^{p}$.

\medskip

The normalized observed information matrix is defined by
\begin{equation}
H_{n}(\theta):=A_{n}(\theta)\mci_{n}(\theta)A_{n}(\theta)
=\left[a_{kn}(\theta)a_{ln}(\theta)\mci^{(kl)}_{n}(\theta)\right]_{k,l=1}^{p}.
\nonumber
\end{equation}
The following theorem provides a simple tool for verifying ULAN, uniform asymptotic normality, 
and asymptotic efficiency.

\begin{thm}
Assume that
\begin{description}
\item[(a)] The log-likelihood  functions $\ell_{n}(\cdot)$ are of class $\mcc^{2}(\Theta)$,
\item[(b)] For each $k\in\{1,\dots,p\}$ and $a>0$, 
\begin{equation}
{\sup_{\rho}}
\left|a_{kn}(\theta)a_{kn}(\rho)^{-1}-1\right|\to_{u}0,
\label{hm:thm_uan_A}
\end{equation}
where the supremum is taken over all $\rho\in\Theta$ such that $|A_{n}(\theta)^{-1}(\rho-\theta)|\le a$.
\item[(c)] There exists a continuous map $\theta\mapsto\mci(\theta)$ 
with $\mci(\theta)$ being positive definite for each $\theta\in\Theta$, such that 
$E_{\theta}\{H_{n}(\theta)\}\to_{u}\mci(\theta)$ and that 
$\var_{\theta}\{H_{n}^{(kl)}(\theta)\}\to_{u}0$ for each $k,l\in\{1,\dots,p\}$. 
\end{description}
Then, we have the following.
\begin{enumerate}
\item The family of probability measures $(P^{n}_{\theta};n\in\mbbn)$ is ULAN with rate $A_{n}(\theta)$ 
and Fisher information matrix $\mci(\theta)$.
\item There exists a local maximizer $\hat{\theta}_{n}$ of $\ell_{n}(\theta)$ 
with probability tending to one, for which 
$A_{n}(\theta)^{-1}(\hat{\theta}_{n}-\theta)\cil_{u}N_{p}(0,\mci(\theta)^{-1})$.
\end{enumerate}
\label{hm:thm_uan}
\end{thm}

\begin{proof}
First we prove 2. We have
\begin{align}
& \hspace{-1cm}\left\{E_{\theta}\left(
\frac{|H_{n}(\theta)-\mci(\theta)|}{1+|H_{n}(\theta)-\mci(\theta)|}\right)
\right\}^{2}
\nn\\
&\le E_{\theta}\left(|H_{n}(\theta)-\mci(\theta)|^{2}\right)
\nn\\
&\lesssim\sum_{k,l}\left\{\var_{\theta}\{H_{n}^{(kl)}(\theta)\}
+\left(E_{\theta}\{H^{(kl)}_{n}(\theta)\}-\mci^{(kl)}(\theta)\right)^{2}\right\}\to_{u}0,
\nn
\end{align}
so that by \eqref{hm:uan_metric_p},
\begin{equation}
H_{n}(\theta)\cip_{u}\mci(\theta).
\label{hm:thm_uan-1}
\end{equation}
For $(\rho_{k})_{k=1}^{p}\subset\Theta$ and a constant $a>0$ we let 
$\mci_{n}(\rho_{1},\dots,\rho_{p}):=[\mci_{n}^{(kl)}(\rho_{k})]_{k,l=1}^{p}$, and
\begin{equation}
F_{n}(\theta;a):=\sup_{\rho_{1},\dots,\rho_{p}}
\left|A_{n}(\theta)\{\mci_{n}(\rho_{1},\dots,\rho_{p})-\mci_{n}(\theta)\}A_{n}(\theta)\right|,
\nonumber
\end{equation}
where the supremum is taken over all $\rho_{1},\dots,\rho_{p}\in\Theta$ 
such that $|A_{n}(\theta)^{-1}(\rho_{k}-\theta)|\le a$ for $k=1,\dots,p$. 
We have
\begin{equation}
\left[A_{n}(\theta)\{\mci_{n}(\rho_{1},\dots,\rho_{p})-\mci_{n}(\theta)\}A_{n}(\theta)\right]^{(kl)}
=a_{kn}(\theta)a_{ln}(\theta)\{\mci^{(kl)}(\rho_{k})-\mci^{(kl)}(\theta)\}
\nonumber
\end{equation}
for each $(k,l)$, hence
\begin{align}
& |F_{n}(\theta;a)| \nn\\
&\lesssim\sum_{k,l}{\sup_{\rho_{k}}}
\left|a_{kn}(\theta)a_{ln}(\theta)\{\mci^{(kl)}_{n}(\rho_{k})-\mci^{(kl)}_{n}(\theta)\}\right|
\nn\\
&\lesssim\sum_{k,l}\left\{
{\sup_{\rho_{k}}}
\left|\left\{a_{kn}(\theta)a_{ln}(\theta)a_{kn}(\rho_{k})^{-1}a_{ln}(\rho_{k})^{-1}\right\}
\left\{H^{(kl)}_{n}(\rho_{k})-\mci^{(kl)}(\rho_{k})\right\}\right|\right.
\nn\\
&{}\qquad
+{\sup_{\rho_{k}}}
\left|\left\{a_{kn}(\theta)a_{ln}(\theta)a_{kn}(\rho_{k})^{-1}a_{ln}(\rho_{k})^{-1}-1\right\}
\mci^{(kl)}(\rho_{k})\right|
\nn\\
&{}\qquad\left.
+{\sup_{\rho_{k}}}\left|\mci^{(kl)}(\rho_{k})-\mci^{(kl)}(\theta)\right|
+\left|\mci^{(kl)}(\theta)-H^{(kl)}_{n}(\theta)\right|
\right\}.
\label{hm:thm_uan-3}
\end{align}
Given any functions $f_{n}$ on $\Theta$, 
we have $f_{n}\to_{u}0$ if and only if $f_{n}(\theta_{n})\to 0$ for any convergent $(\theta_{n})\subset\Theta$. 
It follows from \eqref{hm:thm_uan-1} that
\begin{equation}
F_{n}(\theta;a)\cip_{u}0.
\label{hm:thm_uan-2}
\end{equation}
Based on \eqref{hm:thm_uan_A}, \eqref{hm:thm_uan-1}, and \eqref{hm:thm_uan-2}, 
the claims 2 follows from \cite[Theorems 1 and 2]{Swe80}. %\cite[Corollary 2.2]{Swe83}

Turning to the claim 1, 
since we also have $A_{n}(\theta)\mcs_{n}(\theta)\cil_{u}N_{p}(0,\mci(\theta))$ from \cite{Swe80}, 
it remains to prove \eqref{hm:uan_sa+1}. 
But this readily follows from a similar estimate to \eqref{hm:thm_uan-3} about the upper bound of
\begin{align}
& \left|\ell_{n}(\theta+A_{n}(\theta)u_{n})-\ell_{n}(\theta)-\mcs_{n}(\theta)[A_{n}(\theta)u_{n}]
+\frac{1}{2}H_{n}(\theta)[u_{n},u_{n}]\right|
\nn\\
&\lesssim \left|A_{n}(\theta)\left\{
\mci_{n}(\tilde{\theta}_{n}(u_{n}))-\mci_{n}(\theta)\right\}A_{n}(\theta)\right|
\nonumber
\end{align}
for a point $\tilde{\theta}_{n}(u_{n})$ lying in the segment joining $\theta+A_{n}(\theta)u_{n}$ and $\theta$.
\qed\end{proof}

Needless to say, we can remove the condition (b) in Theorem \ref{hm:thm_uan} 
as soon as $A_{n}(\theta)$ is free of $\theta$. 
We should note that under the conditions of Theorem \ref{hm:thm_uan}, 
the convolution theorem automatically ensures the asymptotic optimality of the MLE among the class of all regular estimators, 
in terms of the maximal concentration and the minimal asymptotic covariance matrix: 
recall \eqref{hm:mc_ineq} and \eqref{hm:mincovmat} in Section \ref{hm:sec_conti.obs}.

%%%

\subsubsection{Example: Gamma subordinator}

Let $X$ be the gamma subordinator such that $\mcl(X_{t})=\Gamma(\del t,\gamma)$ whose density is given by
\begin{equation}
p_{t}(x;\del,\gamma)=\frac{\gamma^{\del t}}{\Gamma(\del t)}x^{\del t-1}
\exp(-\gamma x)\mathbf{1}_{\mbbrp}(x).
\label{hm:g-d}
\end{equation}
The L\'evy density of $X$ is given by
\begin{equation}
g(z;\del,\gamma)%=\lim_{t\to 0}\frac{1}{t}p_{t}(z;\del,\gam)
=\frac{\del}{z}\exp(-\gamma z)\mathbf{1}_{\mbbrp}(z).
\nn
\end{equation}
In this model, the stable approximation in small time through \eqref{hm:lem_ltst-1-1} fails to hold, 
but a certain nonlinear transform $(\dd_{j}X)_{j=1}^{n}$ is in force instead (Remark \ref{hm:rem_nltrans}). 
We also note that, given any $\theta_{i}=(\del_{i},\gamma_{i})$, $i=1,2$, and $T>0$, 
it follows from Theorem \ref{hm:thm_coac} that 
$P^{T}_{\theta_{1}}$ and $P^{T}_{\theta_{2}}$ are not mutually absolutely continuous when $\del_{1}\ne\del_{2}$.

\medskip

The log-likelihood function based on $(\dd_{j}X)_{j=1}^{n}$ is given by
\begin{equation}
\ell_{n}(\theta)=\sum_{j=1}^{n}\bigg\{
\del h_{n}\log\gamma-\log\Gamma(\del h_{n})+\del h_{n}\log(\dd_{j}X)-\gamma\dd_{j}X\bigg\}.
\label{hm:ll-g}
\end{equation}
Denoting by $\psi(x):=\p_{x}\Gamma(x)/\Gamma(x)$ the digamma function, 
we get the following likelihood equations for $(\del,\gam)$:
\begin{align}
\sum_{j=1}^{n}h_{n}\{\log(\del h_{n})-\psi(\del h_{n})\}&=
T_{n}\log\bigg(\frac{X_{T_{n}}}{T_{n}}\bigg)-
\sum_{j=1}^{n}(h_{n})\log\bigg(\frac{\dd_{j}X}{h_{n}}\bigg), \label{hm:g-ee1} \\
\gamma&=\del\frac{T_{n}}{X_{T_{n}}}.
\label{hm:g-ee2}
\end{align}
It is easy to see that the equation $\sum_{j=1}^{n}h_{n}\{\log(\del h_{n})-\psi(\del h_{n})\}=K$ 
admits a unique root $\hat{\del}_{n}$ for each positive $K$, 
hence it is straightforward to solve \eqref{hm:g-ee1} numerically.

\medskip

The following results can be obtained by a direct application of Theorem \ref{hm:thm_uan}:

\begin{thm}
Let $X$ be the gamma subordinator such that $\mcl(X_{1})=\Gamma(\del,\gamma)$ 
with $\theta=(\del,\gamma)\in\Theta$ where $\overline{\Theta}\subset(0,\infty)^{2}$, 
and let $\ell_{n}(\theta)$ and $\hat{\theta}_{n}=(\hat{\del}_{n},\hat{\gamma}_{n})$ be 
as in (\ref{hm:ll-g}) and the solution to (\ref{hm:g-ee1}) and (\ref{hm:g-ee2}), respectively. 
If $T_{n}\to\infty$ and $h_{n}\to 0$, then we have the ULAN with rate
\begin{equation}
A_{n}=\diag\left(\frac{1}{\sqrt{n}},\frac{1}{\sqrt{T_{n}}}\right)
\nonumber
\end{equation}
and Fisher information matrix
\begin{equation}
\mci(\theta)=\left(
\begin{array}{cc}
1/\del^{2} & 0 \\
0          & \del/\gamma^{2}
\end{array}
\right).
\label{hm:g-fi}
\end{equation}
Further, we have $A_{n}^{-1}(\hat{\theta}_{n}-\theta)\cil_{u}N_{2}(0,\mci(\theta)^{-1})$.
\label{hm:thm-uan-g}
\end{thm}

\begin{rem}{\rm 
Here are some observations concerning Theorem \ref{hm:thm-uan-g}.
\begin{itemize}
\item If $T_{n}$ does not tend to infinity, then the observed information associated with $\gamma$ 
is stochastically bounded in $n$ without normalization: we have $-\p_{\gam}^{2}\ell_{n}(\theta)=O_{p}(1)$. 
That is to say, data over fixed time period does not carry enough information to estimate $\gam$ consistently.

\item In contrast, it is possible to deduce $\sqrt{n}(\hat{\del}_{n}-\del)\cil_{u}N_{2}(0,\del^{2})$ 
even when $T_{n}$ is bounded, with leaving the true value of $\gam$ unknown; 
note that we can still use the estimating equation \eqref{hm:g-ee1} for $\del$. 
We then have the ULAN for $\del$ with rate $1/\sqrt{n}$ and Fisher information $\del^{-2}$, 
and the MLE is asymptotically efficient.

\item Using a naive estimator may result in essential loss of asymptotic efficiency, and even worse, 
we may have a slower rate of convergence. 
For example, consider the method of moments based on
\begin{equation}
\frac{1}{T_{n}}\sum_{j=1}^{n}\dd_{j}X\cip\frac{\del}{\gamma}\quad\text{and}\quad
\frac{1}{T_{n}}\sum_{j=1}^{n}(\dd_{j}X)^{2}\cip\frac{\del}{\gamma^{2}}.
\nonumber
\end{equation}
By means of the Lindeberg-Feller central limit theorem and the delta method，
it is easy to prove that the resulting moment estimator $\hat{\theta}_{M,n}=(\hat{\del}_{M,n},\hat{\gam}_{M,n})$ 
satisfies the asymptotic normality with the slower rate of convergence for estimating $\del$ 
and with the non-diagonal asymptotic covariance matrix:
\begin{equation}
\sqrt{T_{n}}(\hat{\theta}_{M,n}-\theta)\cil N_{2}
\bigg(0,
\left(
\begin{array}{cc}
2\del & 2\gamma \\
2\gamma & 3\gamma^{2}/\del
\end{array}
\right)\bigg).
\nonumber
\end{equation}
Thus, considerable amount of information of $\del$ contained in high frequency of data has been thrown away. 
As for $\hat{\gam}_{M,n}$, the rate is optimal but the relative efficiency is $1/3$.

\item In the low-frequency sampling case where $h_{n}\equiv h>0$, Theorem \ref{hm:thm_uan} gives 
\begin{equation}
\sqrt{n}\left(
\begin{array}{c}
\hat{\del}_{n}-\del \\
\hat{\gamma}_{n}-\gamma
\end{array}
\right)\cil_{u}N_{2}\left(
\left(
\begin{array}{c}
0 \\
0
\end{array}
\right),
\left(
\begin{array}{cc}
h^{2}\psi'(\del h) & -h/\gamma \\
-h/\gamma & h\del/\gamma^{2}
\end{array}
\right)^{-1}\right).
\label{hm:g-fi_lfs}
\end{equation}
Since $\ep^{2}\psi'(\ep)\to 1$ as $\ep\to 0$, we see that 
formally letting $h\to 0$ in \eqref{hm:g-fi_lfs} 
after multiplying the matrix $\diag(1,\sqrt{h_{n}})$ on the both sides results in \eqref{hm:g-fi}. 
This exemplifies quite different features between low- and high-frequency sampling schemes.
\end{itemize}
\label{hm:rem1_gamma_uan}
}\qed\end{rem}

%%%

\subsubsection{Example: Inverse-Gaussian subordinator}\label{hm:sec_ig_subo_uan}

Let $X$ be the inverse-Gaussian subordinator such that $\mcl(X_{t})=IG(\del t,\gamma)$, which admits the density
\begin{equation}
p_{t}(x;\del,\gamma)=\frac{\del te^{\del t\gamma}}{\sqrt{2\pi}}x^{-3/2}
\exp\bigg\{-\frac{1}{2}\bigg(\gamma^{2}x+\frac{(\del t)^{2}}{x}\bigg)\bigg\}\mathbf{1}_{\mbbrp}(x).
\label{hm:ig-d}
\nonumber
\end{equation}
The positive half-stable subordinator appears as the limit for $\gam\to 0$. 
The L\'evy measure admits the density
\begin{equation}
g(z;\del,\gamma)%=\lim_{t\to 0}\frac{1}{t}p_{t}(z;\del,\gam)
=\frac{\del}{\sqrt{2\pi}}z^{-3/2}\exp\bigg(-\frac{\gamma^{2}z}{2}\bigg)\mathbf{1}_{\mbbrp}(z).
\nn
\end{equation}
In case where a continuous-time data $(X_{t})_{t\in[0,T]}$ is available, 
Theorem \ref{hm:thm_coac} tells us that, given any $\theta_{i}=(\del_{i},\gamma_{i})$, $i=1,2$, and $T>0$, 
the measures $P^{T}_{\theta_{1}}$ and $P^{T}_{\theta_{2}}$ fail to be mutually absolutely continuous 
if $\del_{1}\ne\del_{2}$.

The log-likelihood function of $(X_{t^{n}_{j}})_{j=0}^{n}$ is
\begin{equation}
\ell_{n}(\theta)=\sum_{j=1}^{n}\bigg\{\log\del+\del\gamma h_{n}-
\frac{1}{2}\bigg(\frac{\del^{2}h_{n}^{2}}{\dd_{j}X}+\gamma^{2}\dd_{j}X\bigg)\bigg\},
\label{hm:ll-ig}
\end{equation}
based on which the MLE is explicitly given by
\begin{equation}
\hat{\del}_{n}=\bigg\{\frac{1}{n}\bigg(
\sum_{j=1}^{n}\frac{h_{n}^{2}}{\dd_{j}X}
-\frac{T_{n}^{2}}{X_{T_{n}}}\bigg)\bigg\}^{-1/2},\quad
\hat{\gamma}_{n}=\frac{T_{n}\hat{\del}_{n}}{X_{T_{n}}}.
\label{hm:ig1}
\end{equation}
As soon as $\del,\gam>0$, we have $E_{\theta}(X_{h}^{k})<\infty$ for each $h>0$ and $k\in\mbbz$. 
In fact, it can be shown that
\begin{equation}
E(X_{h}^{k})=\sqrt{\frac{2}{\pi}}e^{\del h\gam}\gam^{1/2-k}(\del h)^{1/2+k}
K_{1/2-k}(\del h\gam),\quad k\in\mbbz,
\nonumber
\end{equation}
where $K_{w}(y)$, $y>0$, denotes the modified Bessel function of the third kind with index $w\in\mbbr$ 
(see \cite{AbrSte92}; 
sometimes also referred to as ``modified Bessel function of the second kind'' or ``modified Hankel function''):
\begin{equation}
K_{w}(y):=\frac{1}{2}\int_{0}^{\infty}x^{w-1}\exp\bigg\{-\frac{y}{2}\bigg(x+\frac{1}{x}\bigg)\bigg\}dx.
\label{hm:K_def}
\end{equation}
In particular, for the negative-order moments we have
\begin{equation}
E(X_{h}^{-k})=\gam^{k}(\del h)^{-k}\left\{
1+\sum_{j=1}^{k}\frac{(k+j)!}{(k-j)!j!}(2\gam\del)^{-j}h^{-j}
\right\},\quad k\in\mbbn,
\nonumber
\end{equation}
which follows from the formula
\begin{equation}
K_{l+1/2}(z)=e^{-z}\sqrt{\frac{\pi}{2z}}\sum_{j=0}^{l}\frac{(l+j)!}{(l-j)!j!(2z)^{j}},\quad l\in\mbbn.
\nonumber
\end{equation}
It follows that
\begin{equation}
\sup_{h\in(0,1]}E\bigg\{\bigg(\frac{h^{2}}{X_{h}}\bigg)^{k}\bigg\}<\infty,\quad k\in\mbbn,
\label{hm:ig_inv_moments}
\end{equation}
and also that
\begin{align}
& E_{\theta}(X_{h})=\frac{\del h}{\gam}, \quad 
E_{\theta}\{(X_{h})^{2}\}=\frac{\del h}{\gam^{3}}+\bigg(\frac{\del h}{\gam}\bigg)^{2}, \nonumber \\
& E_{\theta}\{(X_{h})^{-1}\}=\frac{1}{(\del h)^{2}}+\frac{\gam}{\del h}, \quad 
E_{\theta}\{(X_{h})^{-2}\}=\frac{3}{(\del h)^{4}}+\frac{3\gam}{(\del h)^{3}}
+\left(\frac{\gam}{\del h}\right)^{2}.
\nonumber
\end{align}
Now we can apply Theorem \ref{hm:thm_uan} to derive the following, 
a quite similar phenomenon to Theorem \ref{hm:thm-uan-g}:

\begin{thm}
Let $X$ be the inverse-Gaussian subordinator such that $\mcl(X_{1})=IG(\del,\gamma)$ 
with $\theta=(\del,\gamma)\in\Theta$ where $\overline{\Theta}\subset(0,\infty)^{2}$, 
and let $\ell_{n}(\theta)$ and $\hat{\theta}_{n}=(\hat{\del}_{n},\hat{\gamma}_{n})$ be 
as in (\ref{hm:ll-ig}) and (\ref{hm:ig1}), respectively. 
If $T_{n}\to\infty$ and $h_{n}\to 0$, then we have the ULAN with rate
\begin{equation}
A_{n}=\diag\left(\frac{1}{\sqrt{n}},\frac{1}{\sqrt{T_{n}}}\right)
\nonumber
\end{equation}
and Fisher information
\begin{equation}
\mci(\theta)=\left(
\begin{array}{cc}
2/\del^{2} & 0 \\
0          & \del/\gam
\end{array}
\right),
\label{hm:g-fi+1}
\end{equation}
and moreover, $A_{n}^{-1}(\hat{\theta}_{n}-\theta)\cil_{u}N_{2}(0,\mci(\theta)^{-1})$.
\label{hm:thm-uan-ig}
\end{thm}

Analogous remarks to the items in Remark \ref{hm:rem1_gamma_uan} are valid for Theorem \ref{hm:thm-uan-ig}. 
In particular, we can consistently estimate $\del$ even when $T_{n}\lesssim 1$; 
then, for each $\del>0$ we have LAN at rate $1/\sqrt{n}$ with Fisher information $2/\del^{2}$, and moreover
\begin{equation}
\sqrt{n}(\hat{\del}_{n}-\del)\cil_{u}\mcn(0,\del^{2}/2),
\nonumber
\end{equation}
with $\hat{\del}_{n}$ being the same one as in \eqref{hm:ig1}.

%%%

\subsubsection{Example: Normal inverse-Gaussian {\lp}}\label{hm:sec_niglp}

In this section, we will present a fully explicit example of a real-valued {\lp} whose likelihood is well-behaved.

The normal inverse-Gaussian (NIG) distribution ${\rm NIG}(\al,\beta,\del,\mu)$ on $\mbbr$ is defined by the density
\begin{equation}
p(y;\al,\beta,\sig,\mu)=\frac{\al\del}{\pi}
\exp\{\del\sqrt{\al^{2}-\beta^{2}}+\beta(y-\mu)\}
\frac{K_{1}\left(\al \sqrt{\del^{2}+(y-\mu)^{2}}\right)}{\sqrt{\del^{2}+(y-\mu)^{2}}},
\label{hm:nig_density}
\end{equation}
where $K_{1}$ is the modified Bessel function given by \eqref{hm:K_def}. 
We will consider estimation of $\theta:=(\al,\beta,\del,\mu)\in\Theta\subset\mbbr^{4}$, 
with $\Theta$ being a bounded convex domain whose closure satisfies that
\begin{equation}
\overline{\Theta}\subset\left\{(\al,\beta,\del,\mu);~\al>0,~|\beta|\in[0,\al),~\del>0,~\mu\in\mbbr\right\}.
\label{hm:nig_Theta_region}
\end{equation}
Note that we precluded the Cauchy case ($\al=|\beta|=0$), which occurs as 
the total-variation limit of ${\rm NIG}(\al,\beta,\del,\mu)$ for $\al=|\beta|\to 0$.

The distribution ${\rm NIG}(\al,\beta,\del,\mu)$ is infinitely divisible whose generating triplet 
$(b_{\theta},c_{\theta},\nu_{\theta})$ of the form \eqref{hm:lk} is given as follows:
\begin{itemize}
\item The L\'evy measure $\nu_{\theta}$ admits the density
\begin{equation}
g(z;\al,\beta,\del)=\frac{\al\del}{\pi|z|}e^{\beta z}K_{1}(\al|z|),\quad z\ne 0;
\label{hm:nig_ld}
\end{equation}
\item $c_{\theta}=0$;
\item $b_{\theta}=m_{\theta}-\int_{|z|>1}zg(z;\al,\beta,\del)dz$, 
with $m_{\theta}:=\mu+\beta\del/\sqrt{\al^{2}-\beta^{2}}$ denoting the mean of $X_{1}$.
%\item $b_{\theta}=\mu+\beta\del/\sqrt{\al^{2}-\beta^{2}}$, which equals the mean of $X_{1}$.
\end{itemize}
We refer to \cite{Bar95} and \cite{Bar98} for more details of the NIG distribution and the NIG L\'evy process.

Let $X$ be the univariate NIG L\'evy process such that $\mcl(X_{1})=NIG(\al,\beta,\del,\mu)$. 
Once again, some of the parameters could be estimated without error if a continuous-time data were available. 
Fix any $T>0$, and let $P^{T}_{\theta_{k}}$, $k=1,2$, denote the distribution of $(X_{t})_{t\le T}$ associated with 
$\theta_{k}=(\al_{k},\beta_{k},\del_{k},\mu_{k})\in\Theta$, $k=1,2$. 
Applying Theorem \ref{hm:thm_coac}, we see that 
$P_{\theta_{1}}^{T}$ and $P_{\theta_{2}}^{T}$ are equivalent if and only if $\del_{1}=\del_{2}$ and $\mu_{1}=\mu_{2}$.

\medskip

We now specify what will occur for high-frequency data. 
The LAN and non-degeneracy of the Fisher information has been previously obtained by \cite{KawMas13}, 
where the most materials given in the proof of the next theorem were presented. 
In the light of Theorem \ref{hm:thm_uan}, we can refine \cite[Theorem 3.1]{KawMas13} as follows:

\begin{thm}
Assume the aforementioned setting, and let $T_{n}\to\infty$ and $h_{n}\to 0$. 
Then we have the ULAN with rate
\begin{equation}
A_{n}=\diag\left(\frac{1}{\sqrt{T_{n}}},\frac{1}{\sqrt{T_{n}}},\frac{1}{\sqrt{n}},\frac{1}{\sqrt{n}}\right),
\label{hm:rate}
\end{equation}
and Fisher information matrix
\begin{equation}
\mci(\theta)=\left(
\begin{array}{cccc}
\mci_{11}(\theta) & \mci_{12}(\theta) & 0 & 0 \\
 & \mci_{22}(\theta) & 0 & 0 \\
 & & \mci_{33}(\theta) & 0 \\
\text{{\rm sym.}} & & & \mci_{44}(\theta)
\end{array}
\right),
\label{hm:nig_FI_form}
\end{equation}
where the entries are given as follows:
\begin{align}
& \mci_{11}(\theta)=\frac{\del}{\al\pi}\int_{0}^{\infty}(e^{(\beta/\al)y}+e^{-(\beta/\al)y})y
\frac{\{K_{0}(y)\}^{2}}{K_{1}(y)}dy,
\nn\\
& \mci_{12}(\theta)=\frac{-\al\beta\del}{(\al^{2}-\beta^{2})^{3/2}},\quad
\mci_{22}(\theta)=\frac{\al^{2}\del}{(\al^{2}-\beta^{2})^{3/2}},
\nn\\
& \mci_{33}(\theta)=\frac{1}{2\del^{2}}, \quad 
\mci_{44}(\theta)=\frac{1}{2\del^{2}}.
\nonumber
\end{align}
For each $\theta\in\Theta$ the integral in $\mci_{11}(\theta)$ is finite and $\mci(\theta)$ is positive definite. 
Further, we have $A_{n}^{-1}(\hat{\theta}-\theta)\cil_{u}N_{4}(0,\mci(\theta)^{-1})$.
\label{hm:nig_thm}
\end{thm}

\begin{rem}{\rm 
We took this opportunity to correct an error about the expression $\mci_{12}(\theta)$ of \cite{KawMas13}, 
which contains ``$\arctan$''. As seen by the proof given below, Theorem 3.1 of \cite{KawMas13} remains to hold if 
we replace $\mci_{12}(\theta)$ therein by the correct one specified in Theorem \ref{hm:nig_thm}.
}\qed\end{rem}

\begin{proof}[Theorem \ref{hm:nig_thm}]
In view of Theorem \ref{hm:thm_uan}, we need to verify the uniform convergence of 
$E_{\theta}\{H_{n}(\theta)\}$ and $\var\{H_{n}^{(kl)}(\theta)\}$ 
for the observed information matrix $H_{n}(\theta):=A_{n}\mci_{n}(\theta)A_{n}$. 
The proof is divided into several steps.

\medskip

{\it Step 1.} 
We begin with the locally Cauchy distributional property in small time. We have
\begin{equation}
\vp_{X_{1}}(u)=\exp\left\{iu\mu+\del\left(
\sqrt{\al^{2}-\beta^{2}}-\sqrt{\al^{2}-(iu+\beta)^{2}}
\right)\right\},
\nn%\label{hm:nig_vp}
\end{equation}
from which
\begin{equation}
\mcl\left(a(X_{h}-\mu h)\right)
=NIG\left(\frac{\al}{|a|},\frac{\beta}{a},\del|a|h_{n},0\right)
\label{hm:nig_rp}
\end{equation}
for any $h>0$ and $a\ne 0$. 
Observe that for each $n\in\mbbn$ the i.i.d. triangular array
\begin{equation}
\ep_{nj}=\ep_{nj}(\del,\mu):=\frac{\dd_{j}X-\mu h_{n}}{\del h_{n}}
\nonumber
\end{equation}
has the common distribution $NIG(\al\del h_{n},\beta\del h_{n},1,0)$. 
We denote by $f_{h_{n}}:\mbbr\to(0,\infty)$ the density of $\mcl\{(X_{h_{n}}-h_{n})/(\del h_{n})\}$. 
The goal of this first step is to prove that
\begin{equation}
\forall k\in\mbbzp,\quad 
\lim_{h_{n}\to 0}\sup_{y\in\mbbr}\left|\p_{y}^{k}f_{h_{n}}(y)-\p_{y}^{k}\phi_{1}(y)\right|=0,
\label{hm:nig_step1_goal}
\end{equation}
where $\mbbzp:=\mbbn\cup\{0\}$ and $\phi_{1}(y)=(1+y^{2})^{-1}/\pi$ denotes 
the standard symmetric Cauchy density corresponding to the characteristic function $u\mapsto\exp(-|u|)$. 

Let $m:=\al^{2}-\beta^{2}>0$. Then, we trivially have
\begin{equation}
\vp_{\ep_{n1}}(u)=\exp\left\{\del h_{n}\sqrt{m}-\sqrt{(\al\del h_{n})^{2}-(iu+\beta\del h_{n})^{2}}\right\}.
\nonumber
\end{equation}
Put $A=(\del h_{n})^{2}m+u^{2}$ and $B=-2\beta\del h_{n}u$. 
Then, simple manipulation gives
\begin{equation}
\vp_{\ep_{n1}}(u)=e^{\del h_{n}\sqrt{m}}\exp\left\{
-\sqrt{\frac{1}{2}(A+\sqrt{A^{2}+B^{2}})}-\frac{iB}{\sqrt{2(A+\sqrt{A^{2}+B^{2}})}}
\right\}.
\label{hm:rev_eq+2}
\end{equation}
It follows that $\vp_{\ep_{n1}}(u)\to\exp(-|u|)$ for each $u\in\mbbr$. 
The expression (\ref{hm:rev_eq+2}) also leads to the estimate
\begin{equation}
|\vp_{\ep_{n1}}(u)|\lesssim\exp\bigg\{-\sqrt{\frac{1}{2}(A+\sqrt{A^{2}+B^{2}})}\bigg\}
\le e^{-\sqrt{A}}\le e^{-|u|}.
\label{hm:nig_lem2-1}
\end{equation}
By means of the Fourier inversion formula we have
\begin{align}
\sup_{y\in\mbbr}\left|\p_{y}^{k}f_{h_{n}}(y)-\p_{y}^{k}\phi_{1}(y)\right|
&\lesssim\int|u|^{k}\left|\vp_{\ep_{n1}}(u)-e^{-|u|}\right|du.
\label{hm:nig_lem2-2}
\end{align}
Then \eqref{hm:nig_step1_goal} follows on applying the dominated convergence theorem 
to the upper bound of \eqref{hm:nig_lem2-2} under \eqref{hm:nig_lem2-1}.

\medskip

{\it Step 2.}
We introduce the functions
\begin{align}
\eta(y)&:=\phi_{1}'(y)/\phi_{1}(y),\quad y\in\mbbr, \nonumber \\
H(y)&:=y^{-1}\{1+yK_{1}'(y)/K_{1}(y)\}=-K_{0}(y)/K_{1}(y),\quad y\in[0,\infty),
\label{hm:H_def}
\end{align}
where we used the identity $K_{w}'(y)=-K_{w-1}(y)-(w/y)K_{w}(y)$ for (\ref{hm:H_def}). 
The function $H$ and its derivatives are to be defined at $y=0$ as limits from the right. 
These functions will play important roles later on. 
In this step, we will prove the following three properties.
\begin{description}
\item[(a)] The functions $y\mapsto \eta(y)$, $y\eta(y)$, and $y^{2}\eta'(y)$ are bounded in $\mbbr$.
\item[(b)] $y\mapsto H(y)$ is bounded and continuous in $[0,\infty)$. Moreover, 
$H(y)\sim -y\log(1/y)$ as $y\to 0$ and $H(y)=-1+1/(2y)-3/(8y^{2})+O(y^{-3})$ as $y\to\infty$.
\item[(c)] $H'(y)\sim -\log(1/y)$ as $y\to 0$ and $y^{2}H'(y)=-1/2+O(y^{-1})$ as $y\to\infty$. 
In particular, $y\mapsto yH'(y)$ is bounded and continuous in $[0,\infty)$.
\end{description}

\medskip

The claim (a) follows from the fact 
$\sup_{y\in\mbbr}|y|^{k}|\p_{y}^{k}\phi_{1}(y)|/\phi_{1}(y)<\infty$ for each $k\in\mbbzp$. 
As for (b), the continuity of $H$ is obvious. It is known that
\begin{align}
K_{w}(y)&\sim
\left\{
\begin{array}{ll}
\log(1/y)+\log 2-\mfC &\quad \text{if $w=0$}, \\
\Gam(|w|)2^{|w|-1}y^{-|w|} &\quad \text{if $w\ne 0$},
\end{array}
\right.\quad\text{as $y\to 0$}, \label{hm:bessel1} \\
K_{w}(y)&=\sqrt{\frac{\pi}{2y}}e^{-y}\left\{
1+\frac{\kappa-1}{8y}+\frac{(\kappa-1)(\kappa-9)}{(8y)^{2}2!}+O(y^{-3})
\right\}
\quad\text{as $y\to\infty$}, \label{hm:bessel2}
\end{align}
where $\mfC$ $(\approx 0.5772)$ denotes the Euler-Mascheroni constant and $\kappa:=4w^{2}$ (see \cite{AbrSte92}). 
The desired behavior of $H(y)$ as $y\to 0$ follows on applying (\ref{hm:bessel1}) to \eqref{hm:H_def}. 
Further, we can deduce the desired behavior of $H(y)$ as $y\to\infty$ 
by applying (\ref{hm:bessel2}) for $w=0,1$ and then expanding the fraction $-K_{0}(y)/K_{1}(y)$ 
as a power series of $y^{-1}$. Now the boundedness of $H$ is trivial. 

Turning to (c), we note the identity $H'(y)=1+H(y)/y-\{H(y)\}^{2}$, hence $y^{2}H'(y)=y^{2}+yH(y)-y^{2}\{H(y)\}^{2}$. 
This follows on applying \eqref{hm:H_def} together with the identity 
$K_{w}(y)=K_{-w}(y)$, which is valid for each $w,y>0$. These expressions combined with (b) prove (c).

\medskip

{\it Step 3.}
In view of \eqref{hm:nig_density} and \eqref{hm:nig_rp}, we can express the log-likelihood function as
\begin{align}
\ell_{n}(\theta)&=\sum_{j=1}^{n}\bigg\{
\log\al+\del h_{n}(\sqrt{m}+\beta\ep_{nj})+\log\phi_{1}(\ep_{nj})
\nn\\
&{}\qquad
+\frac{1}{2}\log(1+\ep_{nj}^{2})+\log K_{1}\left(\al\del h_{n}\sqrt{1+\ep_{nj}^{2}}\right)\bigg\}.
\label{hm:nig_lf}
\end{align}
The introduction of the standard Cauchy density $\phi_{1}$ in the expression \eqref{hm:nig_lf} 
will turn out to be convenient 
in the process of deriving various limiting values as well as deducing estimates of stochastically small terms. 

Let
\begin{equation}
q_{nj}=q_{nj}(\al,\del,\mu):=\al\del h_{n}\sqrt{1+\ep_{nj}^{2}}.
\nonumber
\end{equation}
Noting that $\p_{\mu}\ep_{nj}=-\del^{-1}$, $\p_{\mu}^{2}\ep_{nj}=0$, 
$\p_{\del}\ep_{nj}=-\del^{-1}\ep_{nj}$, $\p_{\del}^{2}\ep_{nj}=2\del^{-2}\ep_{nj}$, 
and $\p_{\del}\p_{\mu}\ep_{nj}=\del^{-2}$, 
we can differentiate (\ref{hm:nig_lf}) to get the following partial derivatives:
\begin{align}
\p_{\al}\ell_{n}(\theta)&=\sum_{j=1}^{n}
\left(
\frac{\al\del h_{n}}{\sqrt{m}}+\frac{1}{\al}q_{nj}H(q_{nj})
\right),
\nn\\%\label{hm:nig_p11} \\
\p_{\beta}\ell_{n}(\theta)&=\sum_{j=1}^{n}
\left\{\del h_{n}\left(\ep_{nj}-\frac{\beta}{\sqrt{m}}\right)
\right\},
\nn\\%\label{hm:nig_p12} \\
\p_{\del}\ell_{n}(\theta)&=\sum_{j=1}^{n}
\bigg\{-\frac{1}{\del}(\ep_{nj}\eta(\ep_{nj})+1)
+h_{n}\left(\sqrt{m}+\frac{\al}{\sqrt{1+\ep_{nj}^{2}}}H(q_{nj})\right)
\bigg\},
\nn\\%\label{hm:nig_p13} \\
\p_{\mu}\ell_{n}(\theta)&=\sum_{j=1}^{n}
\bigg\{
-\frac{1}{\del}\eta(\ep_{nj})-h_{n}\left(\beta
+\frac{\al\ep_{nj}}{\sqrt{1+\ep_{nj}^{2}}}H(q_{nj})\right)
\bigg\},
\nn\\%\label{hm:nig_p14}
\p_{\al}^{2}\ell_{n}(\theta)&=\sum_{j=1}^{n}
\bigg(
-\frac{\beta^{2}\del h_{n}}{m^{3/2}}+\frac{q_{nj}^{2}}{\al^{2}}H'(q_{nj})
\bigg),
\nn\\%\label{hm:nig_p21} \\
\p_{\beta}^{2}\ell_{n}(\theta)&=\sum_{j=1}^{n}
\left(
-\frac{\al^{2}\del h_{n}}{m^{3/2}}
\right)
=-\frac{\al^{2}\del T_{n}}{m^{3/2}},
\nn\\%\label{hm:nig_p22} \\
\p_{\del}^{2}\ell_{n}(\theta)&=\sum_{j=1}^{n}
\bigg\{
\frac{1}{\del^{2}}(1+2\ep_{nj}\eta(\ep_{nj})+\ep_{nj}^{2}\eta'(\ep_{nj}))
\nn\\
&{}\qquad
+\frac{\al h_{n}}{\del}\bigg(
\frac{q_{nj}H'(q_{nj})}{(1+\ep_{nj}^{2})^{3/2}}
+\frac{\ep_{nj}^{2}H(q_{nj})}{(1+\ep_{nj}^{2})^{3/2}}
\bigg)\bigg\},
\nn\\%\label{hm:nig_p23} \\
\p_{\mu}^{2}\ell_{n}(\theta)&=\sum_{j=1}^{n}
\bigg\{
\frac{1}{\del^{2}}\eta'(\ep_{nj})+\frac{\al h_{n}}{\del}
\bigg(
\frac{\ep_{nj}^{2}q_{nj}H'(q_{nj})}{(1+\ep_{nj}^{2})^{3/2}}
+\frac{H(q_{nj})}{(1+\ep_{nj}^{2})^{3/2}}
\bigg)\bigg\},
\nn\\%\label{hm:nig_p24}
\p_{\al}\p_{\beta}\ell_{n}(\theta)&=\sum_{j=1}^{n}
\frac{\al\beta\del h_{n}}{m^{3/2}}
=\frac{\al\beta\del T_{n}}{m^{3/2}},
\nn\\%\label{hm:nig_p2o1} \\
\p_{\al}\p_{\del}\ell_{n}(\theta)&=\sum_{j=1}^{n}
\bigg(
\frac{\al h_{n}}{\sqrt{m}}+\frac{h_{n}}{\sqrt{1+\ep_{nj}^{2}}}(H(q_{nj})+q_{nj}H'(q_{nj}))
\bigg),
\nn\\%\label{hm:nig_p2o2} \\
\p_{\al}\p_{\mu}\ell_{n}(\theta)&=\sum_{j=1}^{n}
\bigg\{
-\frac{h_{n}\ep_{nj}}{\sqrt{1+\ep_{nj}^{2}}}(H(q_{nj})+q_{nj}H'(q_{nj}))
\bigg\},
\nn\\%\label{hm:nig_p2o3} \\
\p_{\beta}\p_{\del}\ell_{n}(\theta)&=\sum_{j=1}^{n}
\bigg(
-\frac{\beta h_{n}}{\sqrt{m}}
\bigg)
=-\frac{\beta T_{n}}{\sqrt{m}},
\nn\\%\label{hm:nig_p2o4} \\
\p_{\beta}\p_{\mu}\ell_{n}(\theta)&=\sum_{j=1}^{n}
(-h_{n})=-T_{n},
\nn\\%\label{hm:nig_p2o5} \\
\p_{\del}\p_{\mu}\ell_{n}(\theta)&=\sum_{j=1}^{n}
\bigg(\frac{1}{\del^{2}}(\eta(\ep_{nj})+\ep_{nj}\eta'(\ep_{nj}))
\nn\\
&{}\qquad
-\frac{\al h_{n}}{\del}\frac{\ep_{nj}}{(1+\ep_{nj}^{2})^{3/2}}(q_{nj}H'(q_{nj})-H(q_{nj}))
\bigg).
\nn%\label{hm:nig_p2o6}
\end{align}
The task is to verify the uniform convergences of 
$E_{\theta}\{H_{n}(\theta)\}$ and $\var\{H_{n}^{(kl)}(\theta)\}$, 
and also the positive definiteness of $\mci(\theta)$. 
For the former we will only prove $E_{\theta}\{H_{n}(\theta)\}\to_{u}\mci(\theta)$; 
as a matter of fact, we can prove that $\var_{\theta}\{H_{n}^{(kl)}(\theta)\}\to_{u}0$ 
in an analogous and simpler way, making use of the statements (a)$\sim$(c) in Step 2. 

It is straightforward to deduce the convergences $E_{\theta}\{H_{n}^{(kl)}(\theta)\}\to_{u}\mci_{kl}(\theta)$ 
except for the case $(k,l)=(1,1)$, 
by using the identities $E_{\theta}(\ep_{nj})=\beta m^{-1/2}$ and 
$E_{\theta}\{(\ep_{nj}-\beta m^{-1/2})^{2}\}=(\del h_{n})^{-1}\al^{2}m^{-3/2}$, 
the convergence \eqref{hm:nig_step1_goal}, and (a)$\sim$(c) together with the bounded convergence theorem, 
and also by reminding the identity $H'(y)=1+H(y)/y-\{H(y)\}^{2}$; for example,
\begin{align}
E_{\theta}\left\{H^{(34)}_{n}(\theta)\right\}
&=-E_{\theta}\left\{\frac{1}{n}\p_{\del}\p_{\mu}\ell_{n}(\theta)\right\}
\nn\\
&=-\frac{1}{\del^{2}}\frac{1}{n}\sum_{j=1}^{n}
E_{\theta}\left\{\eta(\ep_{nj})+\ep_{nj}\eta'(\ep_{nj})\right\}+o_{p}^{\ast}(1)
\nn\\
&\cip_{u}-\frac{1}{\del^{2}}\int_{\mbbr}\frac{\phi_{1}'(y)}{\phi_{1}(y)}
\left\{1+y\frac{\phi_{1}'(y)}{\phi_{1}(y)}\right\}\phi_{1}(y)dy=0=\mci_{34}(\theta).
\nonumber
\end{align}
Here and in the sequel, the asterisk means that it holds uniformly over each compact subset of $\Theta$. 
To prove the remaining
\begin{equation}
E_{\theta}\{H_{n}^{(11)}(\theta)\}\to_{u}\mci_{11}(\theta),
\label{hm:nig_11-conv}
\end{equation}
we need some preliminary facts.

\medskip

{\it Step 4.}
Let
\begin{equation}
A_{k}(\theta):=(-1)^{k}\frac{\al\del}{\pi}\int_{0}^{\infty}(e^{(\beta/\al)y}+e^{-(\beta/\al)y})y^{k-1}
K_{1}(y)\bigg\{\frac{K_{0}(y)}{K_{1}(y)}\bigg\}^{k}dy.
\nonumber
\end{equation}
In this step, we will prove that
\begin{equation}
\lim_{n\to\infty}\frac{1}{h_{n}}E_{\theta}\left[\{q_{n1}H(q_{n1})\}^{k}\right]
=A_{k}(\theta),\quad k\in\mbbn,
\label{hm:nig_proof+1}
\end{equation}
each limit being finite. Applying (\ref{hm:H_def}), we have
\begin{align}
& \hspace{-1.5cm}\frac{1}{h_{n}}E_{\theta}\left[\{q_{n1}H(q_{n1})\}^{k}\right] \nonumber \\
&=\frac{1}{h_{n}}\int_{\mbbr}
\left\{\al\del h_{n}\sqrt{1+x^{2}}H\left(\al\del h_{n}\sqrt{1+x^{2}}\right)\right\}^{k}
\nn\\
&{}\qquad
\times\frac{\al\del h_{n}}{\pi}e^{\del h_{n}\sqrt{m}+\beta\del h_{n} x}
\frac{K_{1}\left(\al\del h_{n}\sqrt{1+x^{2}}\right)}{\sqrt{1+x^{2}}}dx
\nonumber \\
&=(-1)^{k}\frac{\al\del}{\pi}e^{\del h_{n}\sqrt{m}}\bigg[
\int_{\mbbr}\al\del h_{n}e^{\beta\del h_{n} x}\left(\al\del h_{n}\sqrt{1+x^{2}}\right)^{k-1}
\nn\\
&{}\qquad
\times\bigg\{\frac{K_{0}(\al\del h_{n}\sqrt{1+x^{2}})}{K_{1}(\al\del h_{n}\sqrt{1+x^{2}})}
\bigg\}^{k}K_{1}(\al\del h_{n}\sqrt{1+x^{2}})dx\bigg] \nonumber \\
&=:(-1)^{k}\frac{\al\del}{\pi}e^{\del h_{n}\sqrt{m}}B^{(k)}_{h_{n}}
\nn\\
&\sim (-1)^{k}\frac{\al\del}{\pi}B^{(k)}_{h_{n}}.
\label{hm:nig_lem4-0}
\end{align}
Let $B^{(k)}_{h_{n}}=\int_{0}^{\infty}+\int_{-\infty}^{0}=:B^{(k)+}_{h_{n}}+B^{(k)-}_{h_{n}}$.

First we look at $B^{(k)+}_{h_{n}}$. 
The change of variable $y=\al\del h_{n}(\sqrt{1+x^{2}}-1)$ leads to 
$B^{(k)+}_{h_{n}}=\int_{0}^{\infty}b^{(k)+}_{h_{n}}(y)dy$, where
\begin{equation}
b^{(k)+}_{h_{n}}(y):=e^{(\beta/\al)\sqrt{y}\sqrt{y+2\al\del h_{n}}}
\frac{(y+\al\del h_{n})^{k}}{\sqrt{y}\sqrt{y+2\al\del h_{n}}}
\bigg\{\frac{K_{0}(y+\al\del h_{n})}{K_{1}(y+\al\del h_{n})}\bigg\}^{k}K_{1}(y+\al\del h_{n}).
\nonumber
\end{equation}
Obviously, for each $y\in(0,\infty)$
\begin{equation}
b^{(k)+}_{h_{n}}(y)\to e^{(\beta/\al)y}
y^{k-1}K_{1}(y)\bigg\{\frac{K_{0}(y)}{K_{1}(y)}\bigg\}^{k}=:b^{(k)+}_{0}(y).
\label{hm:nig_lem4-1}
\end{equation}
In order to apply the dominated convergence theorem 
to derive the limit of $B^{(k)+}_{h_{n}}$, we have to look at the behaviors of 
$b^{(k)+}_{h_{n}}(y)$ as $y\to 0$ and $y\to\infty$ uniformly in small $h_{n}\in(0,1]$.

By means of (b) in Step 2 and (\ref{hm:bessel2}), we have
\begin{align}
\sup_{h_{n}\le 1}|b^{(k)+}_{h_{n}}(y)|&\lesssim 
e^{(\beta/\al)y}y^{-1/2}(y+\al\del h_{n})^{k-1/2}K_{1}(y+\al\del h_{n})
\nonumber \\
&\lesssim e^{-(1-\beta/\al)y}y^{k-3/2},\quad y\to\infty,
\label{hm:nig_lem4-2}
\end{align}
the upper bound being Lebesgue integrable at infinity since $|\beta|<\al$. 
It follows form (\ref{hm:bessel1}) and (b) that 
$y^{k-1/2}\{K_{0}(y)/K_{1}(y)\}^{k}K_{1}(y)\sim C y^{2k-3/2}\{\log(1/y)\}^{k}\to 0$ as $y\to 0$, 
so that $\sup_{y\in(0,1]}y^{k-1/2}\{K_{0}(y)/K_{1}(y)\}^{k}K_{1}(y)<\infty$. 
Consequently, we have
\begin{align}
\sup_{h_{n}\le 1}|b^{(k)+}_{h_{n}}(y)|&\lesssim y^{-1/2}
\sup_{h_{n}\le 1}\bigg[(y+\al\del h_{n})^{k-1/2}
\bigg\{\frac{K_{0}(y+\al\del h_{n})}{K_{1}(y+\al\del h_{n})}\bigg\}^{k}K_{1}(y+\al\del h_{n})\bigg]
\nonumber \\
&\lesssim y^{-1/2},\quad y\to 0,
\label{hm:nig_lem4-3}
\end{align}
the upper bound being Lebesgue integrable near the origin. 
Having (\ref{hm:nig_lem4-1}), (\ref{hm:nig_lem4-2}) and (\ref{hm:nig_lem4-3}) in hand, 
we can apply the dominated convergence theorem to conclude that
\begin{equation}
B^{(k)+}_{h_{n}}\to\int_{0}^{\infty}b^{(k)+}_{0}(y)dy<\infty.
\nonumber
\end{equation}
In the same manner, we can deduce that
\begin{equation}
B^{(k)-}_{h_{n}}\to\int_{0}^{\infty}b^{(k)-}_{0}(y)dy,
\nonumber
\end{equation}
where $b^{(k)-}_{0}(y):=e^{-(\beta/\al)y}y^{k-1}K_{1}(y)\{K_{0}(y)/K_{1}(y)\}^{k}$. 
Thus we arrive at
\begin{equation}
B^{(k)}_{h_{n}}\to\int_{0}^{\infty}\{b^{(k)+}_{0}(y)+b^{(k)-}_{0}(y)\}dy,
\nonumber
\end{equation}
which combined with \eqref{hm:nig_lem4-0} gives \eqref{hm:nig_proof+1}.

\medskip

{\it Step 5.}
Now we prove \eqref{hm:nig_11-conv}, for which it suffices to show
\begin{equation}
E_{\theta}\left[
-\frac{\beta^{2}\del}{m^{3/2}}+\frac{1}{\al^{2}h_{n}}\left\{
q_{nj}^{2}+q_{nj}H(q_{nj})-\left(q_{nj}H(q_{nj})\right)^{2}
\right\}\right]\to_{u}-\frac{1}{\al^{2}}A_{2}(\theta).
\label{hm:nig_proof+2}
\end{equation}
By \eqref{hm:nig_proof+1}, the left-hand side of \eqref{hm:nig_proof+2} equals
\begin{align}
& -\frac{\beta^{2}\del}{m^{3/2}}+\del^{2}h_{n}\left\{1+E_{\theta}(\ep_{n1}^{2})\right\}
+\frac{1}{\al^{2}}A_{1}(\theta)-\frac{1}{\al^{2}}A_{2}(\theta)+o^{\ast}(1)
\nn\\
&=\frac{\del}{\sqrt{m}}
+\frac{1}{\al^{2}}A_{1}(\theta)-\frac{1}{\al^{2}}A_{2}(\theta)+o^{\ast}(1),
\nn
\end{align}
so that \eqref{hm:nig_proof+2} follows on showing $\del m^{-1/2}+\al^{-2}A_{1}(\theta)=0$. 
We note that $(x+x^{-1})/2\ge 1$ for any $x\ge 0$, and that, for any $|b|<1$,
\begin{align}
& \hspace{-1cm}
\frac{1}{2}\int_{0}^{\infty}\frac{1}{x}\left\{\frac{1}{2}\left(x+\frac{1}{x}\right)+b\right\}^{-1}dx
\nn\\
&=\frac{1}{\sqrt{1-b^{2}}}\left\{
\frac{\pi}{2}-\arctan\left(\frac{b}{\sqrt{1-b^{2}}}\right)
\right\}.
\label{hm:nig_proof+3}
\end{align}
By (\ref{hm:K_def}) and \eqref{hm:nig_proof+3}, 
straightforward computation with Fubini's theorem gives $\al^{-2}A_{1}(\theta)=-\del m^{-1/2}$.

\medskip

{\it Step 6.}
It remains to prove the positive definiteness of $\mci(\theta)$. 
To this end we note an alternative expression for 
$\mci_{12}(\theta)~(=-\al\beta\del(\al^{2}-\beta^{2})^{-3/2})$: 
let $\p_{\al}\ell_{n}(\theta)=\sumj\zeta^{\al}_{nj}(\theta)$ and 
$\p_{\beta}\ell_{n}(\theta)=\sumj\zeta^{\beta}_{nj}(\theta)$, 
then $\mci_{12}(\theta)=(nh_{n})^{-1}\sumj E_{\theta}\{(\zeta^{\al}_{nj}(\theta))\cdot(\zeta^{\beta}_{nj}(\theta))\}
=h_{n}^{-1}E_{\theta}\{(\zeta^{\al}_{n1}(\theta))(\zeta^{\beta}_{n1}(\theta))\}$. 
Then, a computation similar to the one in Step 5 gives
\begin{equation}
\mci_{12}(\theta)
=-\frac{\del}{\al\pi}\int_{0}^{\infty}
\left(e^{(\beta/\al)y}-e^{-(\beta/\al)y}\right)yK_{0}(y)dy,
\nonumber
\end{equation}
so that
\begin{equation}
|\mci_{12}(\theta)|\le
\frac{\del}{\al\pi}\int_{0}^{\infty}
\left(e^{\beta y/\al}+e^{-\beta y/\al}\right)yK_{0}(y)dy
=:\mcj_{12}(\theta).
\nonumber
\end{equation}
It suffices to show that, for any $\al>0$ and $\del>0$ the function
\begin{equation}
f(\beta):=\mci_{11}(\theta)\mci_{22}(\theta)-\{\mcj_{12}(\theta)\}^{2}
\nonumber
\end{equation}
is positive in $[0,\al)$. We introduce the probability density
\begin{equation}
q(y;\beta):=C(\beta)^{-1}(e^{(\beta/\al)y}+e^{-(\beta/\al)y})yK_{1}(y),\quad y>0.
\nonumber
\end{equation}
Using the identity
\begin{align}
& \hspace{-1cm}
\int_{0}^{\infty}\left\{\frac{1}{2}\left(x+\frac{1}{x}\right)+b\right\}^{-2}dx
\nn\\
&=\frac{2}{1-b^{2}}\left[\frac{1}{\sqrt{1-b^{2}}}
\left\{\frac{\pi}{2}-\arctan\left(\frac{b}{\sqrt{1-b^{2}}}\right)\right\}-b\right],
\nn%\label{hm:nig_proof+5}
\end{align}
we get
\begin{align}
C(\beta)
&:=\int_{0}^{\infty}\left(e^{(\beta/\al)y}+e^{-(\beta/\al)y}\right)yK_{1}(y)dy
\nn\\
&=\frac{1}{2}\int_{0}^{\infty}\int_{0}^{\infty}
y\left(e^{-\{(x+x^{-1})/2-\beta/\al\}y}+e^{-\{(x+x^{-1})/2+\beta/\al\}y}\right)dydx
\nonumber \\\
&=\frac{1}{2}\left[\int_{0}^{\infty}
\left\{\frac{1}{2}\left(x+\frac{1}{x}\right)-\frac{\beta}{\al}\right\}^{-2}dx
+\int_{0}^{\infty}
\left\{\frac{1}{2}\left(x+\frac{1}{x}\right)+\frac{\beta}{\al}\right\}^{-2}dx
\right] \nonumber \\
&=\frac{\al^{3}\pi}{m^{3/2}}.
\nonumber
\end{align}
Then, some elementary manipulations and Cauchy-Schwarz's inequality lead to
\begin{align}
f(\beta)
&=\al^{-2}A_{2}(\theta)\cdot\al^{2}\del m^{-3/2}
-\bigg(-\frac{\al^{2}\del}{m^{3/2}}
\int_{0}^{\infty}q(y;\beta)\frac{K_{0}(y)}{K_{1}(y)}dy
\bigg)^{2}
\nonumber \\
&=\frac{\al^{4}\del^{2}}{m^{3}}\bigg\{
\int_{0}^{\infty}q(y;\beta)\bigg(\frac{K_{0}(y)}{K_{1}(y)}\bigg)^{2}dy
\cdot\int_{0}^{\infty}q(y;\beta)dy
\nn\\
&{}\qquad
-\bigg(\int_{0}^{\infty}q(y;\beta)\frac{K_{0}(y)}{K_{1}(y)}dy
\bigg)^{2}\bigg\}>0,
\nonumber
\end{align}where the last strict inequality does hold 
since $y\mapsto K_{0}(y)/K_{1}(y)$ is not a constant on $(0,\infty)$. 
We thus get the positivity of $f$, and the proof is complete.
\qed\end{proof}

\begin{rem}{\rm 
The lower right $2\times 2$ submatrix of \eqref{hm:nig_FI_form} is the same as 
the Fisher information matrix in estimation of the Cauchy L\'evy process such that $\mcl(X_{1})$ admits the Lebesgue density 
$x\mapsto\del^{-1}\phi_{1}(\del^{-1}(x-\mu))=(\del/\pi)\{\del^{2}+(x-\mu)^{2}\}^{-1}$. 
See Theorem \ref{hm:sslp_th1} for details.
}\qed\end{rem}

%%%%%
%%%%%

\section{Estimation of stable {\lp}}\label{hm:sec_slp}

The objective of this section is parametric estimation of some stable-process models based on high-frequency sampling.

\subsection{Some preliminaries}\label{hm:sec_slp_preliminary}

The stable distributions form a pretty special subclass of general infinitely divisible distributions. 
There are several books containing a systematic account of the general stable distributions 
and the stable L\'evy processes: 
\cite{Ber96}, \cite{IbrLin71}, \cite{JanWer94}, \cite{SamTaq94}, \cite{Sat99}, and \cite{Zol86}. 
See also \cite{BorHarWer05} for discussion from financial point of view.

\medskip

The $\beta$-stable {\lp} is characterized by
\begin{equation}
\log\vp_{X_{t}}(u)=\left\{
\begin{array}{ll}
\displaystyle{
-(t^{1/\beta}\sig)^{\beta}|u|^{\beta}\bigg(1-i\rho\mathrm{sign}(u)\tan\frac{\beta\pi}{2}\bigg)+it\gam u}, 
& \quad\beta\ne 1, \\
\displaystyle{
-t\sig|u|\bigg(1+i\frac{2\rho}{\pi}\mathrm{sign}(u)\log|u|\bigg)+it\gam u}, 
& \quad\beta=1,
\end{array}
\right.
\label{hm:ex_st1}
\end{equation}
with the stable index $\beta\in(0,2]$, the scale $\sig>0$, the degree of skewness $\rho\in[-1,1]$, 
and the deterministic trend $\gam\in\mbbr$; 
then, we will write
\begin{equation}
\mcl(X_{t})=S_{\beta}(t^{1/\beta}\sig,\rho,t\gam).
\nonumber
\end{equation}
For $\beta\in(0,2)$, the stable distribution is characterized by the {\lm} $\nu(dz)=g(z)dz$ plus a trend, 
where $g$ takes the form
\begin{equation}
g(z)=\del_{+}z^{-1-\beta}\mathbf{1}_{(0,\infty)}(z)+\del_{-}|z|^{-1-\beta}\mathbf{1}_{(-\infty,0)}(z),
\nonumber
\end{equation}
with $\del{+},\del_{-}\ge 0$ satisfying $(\del_{+},\del_{-})\ne(0,0)$. 
The parameters $(\rho,\sig)$ and $(\del_{+},\del_{-})$ are related by the identities
\begin{equation}
\frac{\del_{+}-\del_{-}}{\del_{+}+\del_{-}}=\rho\quad\text{and}\quad
\sig^{\beta}=\frac{1}{\beta}\Gam(1-\beta)(\del_{+}+\del_{-})\cos\frac{\beta\pi}{2},
\nonumber
\end{equation}
which readily follow on invoking, e.g., \cite[Lemma 14.11]{Sat99}. 
Here we will focus on the non-Gaussian case ($\beta\in(0,2)$), so that for each $t>0$
\begin{equation}
E(|X_{t}|^{q})<\infty\ \iff\ q\in(-1,\beta).
\nonumber
\end{equation}

\begin{rem}{\rm 
The $\beta$-stable distributions has several variants of its parametrization, 
the most typical one being (\ref{hm:ex_st1}). 
When $\rho\ne 0$, the parametrization (\ref{hm:ex_st1}) is ``discontinuous'' at $\beta=1$. 
To get rid of the inconvenience, \cite{Nol98} discussed an alternative parametrization 
via a suitable translation operation of (\ref{hm:ex_st1}).
\qed}\end{rem}

We say that a stochastic process $Y$ has the {\it selfsimilarity}, 
also referred to as the {\it scaling property}, 
if there exist positive constants $a$ and $H$ for which $(Y_{t})=(a^{-H}Y_{at})$ in distribution; 
the parameter $H$ is called the selfsimilarity (or Hurst) index. 
It is known that it is only the stable {\lp} that can have the selfsimilarity among all {\lp es}. 
Specifically, for each $t>0$ we have
\begin{align}
\left\{
\begin{array}{ll}
\mcl\left(t^{-1/\beta}\sig^{-1}(X_{t}-t\gam)\right)=S_{\beta}(1,\rho,0), & \quad\beta\ne 1, \\
\mcl\left(t^{-1}\sig^{-1}(X_{t}-t\gam)-2\rho\pi^{-1}\log(t\sig)\right)=S_{1}(1,\rho,0), & \quad\beta=1.
\end{array}
\right.
\label{hm:ex_st3}
\end{align}
In other words, if $\mcl(S)=S_{\beta}(1,\rho,0)$ then
\begin{equation}
\left\{
\begin{array}{ll}
\mcl(X_{t})=\mcl\left(\sig t^{1/\beta}S+t\gam \right), & \quad\beta\ne 1, \nonumber \\
\mcl(X_{t})=\mcl\left(\sig t S+t\gam+2\rho\sig t\pi^{-1}\log(\sig t)\right), & \quad\beta=1. \nonumber
\end{array}\right.
\end{equation}
This implies that $X$ has the selfsimilarity (with index $1/\beta$) if and only if 
$\gam=0$ (resp. $\rho=0$) when $\beta\ne 1$ (resp. $\beta=1$). 
Note that the right-hand sides of \eqref{hm:ex_st3} are free of $t$. 
This fact is particularly useful when attempting simulations on computer, 
since in order to simulate $\mcl(X_{t})$ it suffices to have a recipe for 
generating $S_{\beta}(1,\rho,0)$-random numbers. 
Let us mention the highly efficient algorithm for generating univariate stable-random numbers 
(\cite{ChaMalStu76} and \cite{Wer96}), based on which 
we can readily generate a discrete-time random sample $(\dd_{j}X)_{j\le n}$.

\begin{algo}
Fix any $t>0$. 
\begin{itemize}
\item[0.] For $\beta\ne 1$, set
\begin{equation}
A_{\beta,\rho}=\bigg\{1+\bigg(\rho\tan\frac{\beta\pi}{2}\bigg)^{2}\bigg\}^{1/(2\beta)}\quad\text{and}\quad
B_{\beta,\rho}=\beta^{-1}\arctan\bigg(\rho\tan\frac{\beta\pi}{2}\bigg).
\nonumber
\end{equation}

\item[1.] Draw random numbers $U$ and $V$ independently from 
the uniform over $(0,1)$ and the exponential with unit mean, respectively, and then set
\begin{align}
S&\gets A_{\beta,\rho}\frac{\sin\{\beta(U+B_{\beta,\rho})\}}{(\cos U)^{1/\beta}}
\bigg[\frac{\cos\{U-\beta(U+B_{\beta,\rho})\}}{V}\bigg]^{(1-\beta)/\beta} & \text{if }\beta\ne 1, \nonumber \\
S&\gets \frac{2}{\pi}\bigg\{\bigg(\frac{\pi}{2}+\rho U\bigg)\tan U
-\rho\log\bigg(\frac{(\pi/2)V\cos U}{\rho U+\pi/2}\bigg)\bigg\} & \text{if }\beta=1. \nonumber
\end{align}
Then $\mcl(S)=S_{\beta}(1,\rho,0)$ in both cases. 
(The original Eq.(3.9) of \cite{Wer96} contains the error: 
for the expression of $S$ when $\beta=1$, 
we need the multiplicative constant $\pi/2$ in the numerator inside the logarithm.)

\item[2.] Set
\begin{align}
X_{t}&\gets t^{1/\beta}\sig S+t\gam & \text{if }\beta\ne 1, \nonumber \\
X_{t}&\gets t\sig S+\frac{2t\sig\rho}{\pi}\log(t\sig)+t\gam. & \text{if }\beta=1. \nonumber
\end{align}
Then $\mcl(X_{t})=S_{\beta}(t^{1/\beta}\sig,\rho,t\gam)$ in both cases.

\end{itemize}
\label{hm:algo_stable}
\end{algo}

\begin{rem}{\rm 
Taking formally $\beta=2$ and $\rho=0$ in Algorithm \ref{hm:algo_stable} results in 
the Box-Muller transform for generating increments of a scaled Wiener process with drift.
\qed}\end{rem}

\medskip

In the case where the {\lm} is symmetric, we have
\begin{equation}
\log\vp_{X_{t}}(u)=-(t^{1/\beta}\sig)^{\beta}|u|^{\beta}+it\gam u,\quad \beta\in(0,2].
\nonumber
\end{equation}
Denote by $y\mapsto\phi_{\beta}(y;\sig)$ the density of 
the symmetric $\beta$-stable distribution corresponding to the 
characteristic function $u\mapsto\exp\{-(\sig|u|)^{\beta}\}$; we will use the shorthands
\begin{equation}
S_{\beta}(\sig):=S_{\beta}(\sig,0,0),\quad \phi_{\beta}(y):=\phi_{\beta}(y;1).
\nonumber
\end{equation} 
The following well-known facts will be frequently used later.

\begin{itemize}
\item The map $(\beta,y)\mapsto\phi_{\beta}(y)$ is everywhere positive 
and of class $C^{\infty}((0,2)\times\mbbr)$.

\item The relation $\phi_{\beta}(y;a)=a^{-1}\phi_{\beta}(a^{-1}y)$ for all $y\in\mbbr$ and $a>0$ is valid, 
as easily seen from the Fourier inversion formula
\begin{equation}
\phi_{\beta}(y;\sig)=\frac{1}{2\pi}\int\exp\{-iuy-(\sig|u|)^{\beta}\}du.
\nonumber
\end{equation}
In particular, $\phi_{\beta}(0;\sig)=(\sig\pi)^{-1}\Gamma(1+1/\beta)$.

\item For any $k,k'\in\mbbzp$, there exist constants $c_{i}=c_{i}(\beta,k,k')>0$ such that
\begin{equation}
|\p^{k}\p_{\beta}^{k'}\phi_{\beta}(y)|\lesssim\int e^{-|u|^{\beta}}|u|^{c_{1}}\{1+(\log|u|)^{c_{2}}\}du.
\nonumber
\end{equation}

\item It follows from the series expansion of the density (e.g. \cite[Remark 14.18]{Sat99}) that 
for any $k,k'\in\mbbzp$
\begin{equation}
|\p^{k}\p_{\beta}^{k'}\phi_{\beta}(y)|\sim C_{k,k',\beta}(\log|y|)^{k'}|y|^{-\beta-1-k},\quad
|y|\to\infty,
\label{hm:sslp_asymp}
\end{equation}
for some constant $C_{k,k',\beta}>0$.
\end{itemize}

\medskip

In the rest of this section we will proceed as follows. 
In Section \ref{hm:sec_sslp_LAN}, we will look at the local asymptotics for the log-likelihood function when the {\ld} is symmetric, 
and then Section \ref{hm:sec_sym.slp_me} presents 
some practical moment estimators which are asymptotically normally distributed. 
In Section \ref{hm:sec_skewslp_me}, we will formulate a practical estimation procedure 
when the L\'evy density is skewed and the scale is time-varying. 
Sections \ref{hm:sec_gen_slp} and \ref{hm:sec_lslp} give some brief remarks concerning 
simple estimation of general $\beta$-stable {\lp es} and locally stable {\lp es}, respectively.

%%%%%

\subsection{LAN with singular Fisher information: symmetric jumps}\label{hm:sec_sslp_LAN}

This section is concerned with the LAN when we observe $(X_{t^{n}_{j}})_{j=1}^{n}$ 
under the high-frequency sampling 
from a $\beta$-stable {\lp} $X$ such that $\mcl(X_{1})=S_{\beta}(\sig,0,\gam)$. 
The parameter of interest is
\begin{equation}
\theta=(\beta,\sig,\gam),
\nonumber
\end{equation}
the parameter space $\Theta\subset\mbbr^{3}$ being a convex domain with compact closure
\begin{equation}
\overline{\Theta}\subset\left\{(\beta,\sig,\gam);~\beta\in(0,2),~\sig>0,~\gam\in\mbbr\right\}.
\nonumber
\end{equation}
It will turn out that, although the log-likelihood admits a LAN structure, 
the asymptotic Fisher information matrix is constantly singular whenever 
both the index $\beta$ and the scale $\sig$ are to be estimated 
(\cite{AitJac08} and \cite{Mas09_slp}). 
This asymptotic singularity is inevitable, so we are in a similar situation 
to the case of the Meixner {\lp} mentioned in Section \ref{hm:sec_mxlplan}.

\medskip

For $j=1,2,\dots,n$ and $\theta\in\Theta$, we write
\begin{equation}
Y_{nj}(\theta)=\sig^{-1}h_{n}^{-1/\beta}(\dd_{j}X-\gamma h_{n}).
\label{hm:sslp_Y}
\end{equation}
According to the scaling property \eqref{hm:ex_st3}, 
the random variables $\{Y_{nj}(\theta)\}_{j=1}^{n}$ under $P_{\theta}$ 
are i.i.d. with common distribution $S_{\beta}(1)$. 
The log-likelihood function of $(X_{jh_{n}})_{j=1}^{n}$ is
\begin{align}
\ell_{n}(\theta)&=\sum_{j=1}^{n}\log\phi_{\beta}\left(\dd_{j}X-\gamma h_{n};\sig h_{n}^{1/\beta}\right)
\nonumber \\
&=\sum_{j=1}^{n}\log\left\{\sig^{-1}h_{n}^{-1/\beta}\phi_{\beta}(Y_{nj}(\theta))\right\} \nonumber \\
&=\sum_{j=1}^{n}\left\{-\log\sig+\beta^{-1}\log(1/h_{n})+\log\phi_{\beta}(Y_{nj}(\theta))\right\}. \nonumber
\end{align}
Let 
\begin{equation}
A_{n}(\beta)=\mathrm{diag}\left(a_{1n},a_{2n},a_{3n}\right):=
\diag\left\{\frac{1}{\sqrt{n}\log(1/h_{n})},~\frac{1}{\sqrt{n}},~\frac{1}{\sqrt{n}h_{n}^{1-1/\beta}}\right\}.
\label{hm:sslp_A}
\end{equation}
We are assuming \eqref{hm:addeqnum-1}, hence $\sqrt{n}h_{n}^{1-1/\beta}\to\infty$ as soon as $\beta<2$.

\begin{thm} Fix any $\theta\in\Theta$. 
For each $u\in\mbbr^{3}$ we have the stochastic expansion
\begin{equation} 
\ell_{n}(\theta+A_{n}(\beta)u)-\ell_{n}(\theta)
=\mcs_{n}(\theta)[u]-\frac{1}{2}\mci(\theta)[u,u]+o_{p}(1),
\nonumber
\end{equation}
where $\mcs_{n}(\theta)\cil N_{3}(0,\mci(\theta))$ under $P_{\theta}$ with
\begin{equation}
\mci(\theta):=\left(
\begin{array}{cccc}
H_{\beta}/\beta^{4}       & H_{\beta}/(\sig\beta^{2}) & 0 \\
H_{\beta}/(\sig\beta^{2}) & H_{\beta}/\sig^{2}      & 0 \\
0                     & 0                     & M_{\beta}/\sig^{2}
\end{array}
\right),
\label{hm:sslp_FI}
\end{equation}
where
\begin{equation}
H_{\beta}:=\int\{\phi_{\beta}(y)+y\p\phi_{\beta}(y)\}^{2}\phi_{\beta}(y)^{-1}dy,\quad
M_{\beta}:=\int\{\p\phi_{\beta}(y)\}^{2}\phi_{\beta}(y)^{-1}dy,
\label{hm:sslp_HM}
\end{equation}
both being finite. In particular, the Fisher information matrix $\mci(\theta)$ is singular for any $\theta\in\Theta$.
\label{hm:sslp_th1}
\end{thm}

\begin{proof}
We may and do assume $\log(1/h_{n})>0$ without loss of generality. 
Fix any $\theta\in\Theta$ and $u\in\mbbr^{3}$ in the sequel. 
Let $p_{h_{n}}(y;\theta):=\sig^{-1}h_{n}^{-1/\beta}\phi_{\beta}(y)$ 
and $g_{nj}(\theta):=\p_{\theta}\log p_{h_{n}}(Y_{nj}(\theta);\theta)$. 
Obviously we have $E_{\theta}\{g_{nj}(\theta)\}=0$, hence it suffices to verify 
Assumptions \ref{hm:lan_assump2_dd}(b) and \ref{hm:lan_assump2_dd}(c).

\medskip

First we consider Assumption \ref{hm:lan_assump2_dd}(b):
\begin{equation}
C_{n}(\theta)=\left[C_{n}^{(kl)}(\theta)\right]_{k,l=1}^{3}:=A_{n}(\beta)\sum_{j=1}^{n}E_{\theta}
\left\{g_{nj}(\theta)^{\otimes 2}\right\}A_{n}(\beta)\to\mci(\theta).
\label{hm:sslp_LL}
\end{equation}
Put $g_{nj}(\theta)=[g_{nj,k}(\theta)]_{k=1}^{3}$ and
\begin{equation}
F_{\beta 1}(y)=\frac{\phi_{\beta}(y)+y\p\phi_{\beta}(y)}{\phi_{\beta}(y)},
\quad F_{\beta 2}(y)=\frac{\p_{\beta}\phi_{\beta}(y)}{\phi_{\beta}(y)},
\quad F_{\beta 3}(y)=\frac{\p\phi_{\beta}(y)}{\phi_{\beta}(y)}.
\nonumber
\end{equation}
Then,
\begin{align}
g_{nj,1}(\theta)&=-\beta^{-2}\log(1/h_{n})F_{\beta 1}(Y_{nj}(\theta))
+F_{\beta 2}(Y_{nj}(\theta)), \label{hm:sslp_g1} \\
g_{nj,2}(\theta)&=-\sig^{-1}F_{\beta 1}(Y_{nj}(\theta)), \label{hm:sslp_g2} \\
g_{nj,3}(\theta)&=-\sig^{-1}h_{n}^{1-1/\beta}F_{\beta 3}(Y_{nj}(\theta)). \label{hm:sslp_g3}
\end{align}
Since $Y_{nj}(\theta)$ forms an i.i.d. array with common distribution $S_{\beta}(1)$, 
it is straightforward to deduce \eqref{hm:sslp_LL} by 
substituting \eqref{hm:sslp_g1}, \eqref{hm:sslp_g2} and \eqref{hm:sslp_g3} in
\begin{equation}
C_{n}^{(kl)}(\theta)=\sum_{j=1}^{n}a_{kn}a_{ln}g_{nj,k}(\theta)g_{nj,l}(\theta),
\quad k,l\in\{1,2,3\}.
\nonumber
\end{equation}
Here, we note that the finiteness of the limit can be ensured by means of 
Schwarz's inequality together with (\ref{hm:sslp_asymp}); 
in particular, we have $\int F_{1}(y)F_{3}(y)\phi_{\beta}(y)dy=0$ 
since $y\mapsto y\{\p\phi_{\beta}(y)\}^{2}/\phi_{\beta}(y)$ is odd.

\medskip

Next we turn to verifying Assumption \ref{hm:lan_assump2_dd}(c). Fix any $a>0$. 
It follows from the expressions (\ref{hm:sslp_g1}) to (\ref{hm:sslp_g3}) that 
(note that $a_{3n}$ depends on $\beta$)
\begin{align}
& \sup_{\rho\in D_{n}(a;\theta)}
\left(a_{n1}^{4}E_{\rho}\{|g_{n1,1}(\rho)|^{4}\}+a_{n2}^{4}E_{\rho}\{|g_{n1,2}(\rho)|^{4}\}
\right)\lesssim n^{-2}, \nonumber \\
& \sup_{\rho\in D_{n}(a;\theta)}a_{n3}^{4}E_{\rho}\{|g_{n1,3}(\rho)|^{4}\}
\lesssim n^{-2}\sup_{\beta':\sqrt{n}\log(1/h_{n})|\beta'-\beta|\le a}h_{n}^{4(1/\beta-1/\beta')}
\lesssim n^{-2}.
\nn
\end{align}
For the latter estimate we used the elementary fact: 
a positive function $f$ is bounded below and above if and only if $|\log f(h)|\lesssim 1$. 
Thus we get
\begin{equation}
n\sup_{\rho\in D_{n}(a;\theta)}E_{\rho}\left\{|A_{n}(\beta)g_{n1}(\rho)|^{4}\right\}\lesssim n^{-1}\to 0.
\label{hm:sslp_ldlan+1}
\end{equation}
It remains to look at $\p_{\theta}g_{nj}(\theta)$. It is straightforward, though a bit tedious, to get
\begin{align}
\p_{\beta}^{2}\log p_{h_{n}}(y;\theta)&=
\frac{1}{\beta^{4}}\{\log(1/h_{n})\}^{2}\bigg\{yF_{\beta 3}(y)
+y^{2}\bigg(\frac{\p^{2}\phi_{\beta}(y)}{\phi_{\beta}(y)}
-F_{\beta 3}(y)^{2}\bigg)\bigg\} \nonumber \\
& {}+\frac{2}{\beta^{3}}\log(1/h_{n})\bigg\{F_{\beta 1}(y)
-\beta y\bigg(\frac{\p\p_{\beta}\phi_{\beta}(y)}{\phi_{\beta}(y)}-
F_{\beta 2}(y)F_{\beta 3}(y)\bigg)\bigg\} \nonumber \\
& {}+\frac{\p_{\beta}^{2}\phi_{\beta}(y)}{\phi_{\beta}(y)}-F_{\beta 2}(y)^{2},
\nn\\%\label{hm:11} \\
\p_{\sig}^{2}\log p_{h_{n}}(y;\theta)&=\frac{1}{\sig^{2}}\bigg\{1+2yF_{\beta 3}(y)
+y^{2}\bigg(\frac{\p^{2}\phi_{\beta}(y)}{\phi_{\beta}(y)}
-F_{\beta 3}(y)^{2}\bigg)\bigg\},
\nn\\%\label{hm:22} \\
\p_{\gamma}^{2}\log p_{h_{n}}(y;\theta)&=\frac{1}{\sig^{2}}h_{n}^{2(1-1/\beta)}
\bigg(\frac{\p^{2}\phi_{\beta}(y)}{\phi_{\beta}(y)}-F_{\beta 3}(y)^{2}\bigg),
\nn\\%\label{hm:33} \\
\p_{\sig}\p_{\beta}\log p_{h_{n}}(y;\theta)&=\frac{1}{\sig\beta^{2}}\log(1/h_{n})\bigg\{
yF_{\beta 3}(y)+y^{2}
\bigg(\frac{\p^{2}\phi_{\beta}(y)}{\phi_{\beta}(y)}-F_{\beta 3}(y)^{2}\bigg)\bigg\} \nonumber \\
& {}-\frac{1}{\sig}y\bigg(\frac{\p\p_{\beta}\phi_{\beta}(y)}{\phi_{\beta}(y)}-
F_{\beta 2}(y)F_{\beta 3}(y)\bigg),
\nn\\%\label{hm:12} \\
\p_{\gamma}\p_{\beta}\log p_{h_{n}}(y;\theta)&=
\frac{1}{\sig\beta^{2}}h_{n}^{1-1/\beta}\log(1/h_{n})\bigg\{F_{\beta 3}(y)
+y\bigg(\frac{\p^{2}\phi_{\beta}(y)}{\phi_{\beta}(y)}
-F_{\beta 3}(y)^{2}\bigg)\bigg\} \nonumber \\
& {}-\frac{1}{\sig}h_{n}^{1-1/\beta}
\bigg(\frac{\p\p_{\beta}\phi_{\beta}(y)}{\phi_{\beta}(y)}
-F_{\beta 2}(y)F_{\beta 3}(y)\bigg),
\nn\\%\label{hm:13} \\
\p_{\gamma}\p_{\sig}\log p_{h_{n}}(y;\theta)&=
\frac{1}{\sig^{2}}h_{n}^{1-1/\beta}\bigg\{F_{\beta 3}(y)
+y\bigg(\frac{\p^{2}\phi_{\beta}(y)}{\phi_{\beta}(y)}
-F_{\beta 3}(y)^{2}\bigg)\bigg\}.
\nn%\label{hm:23}
\end{align}
From these expressions we can deduce as before that
\begin{equation}
n\sup_{\rho\in D_{n}(a;\theta)}
E_{\rho}\{|A_{n}(\theta)\p_{\theta}[g_{n1}(\rho)^{\top}]A_{n}(\theta)|^{2}\}\to 0,
\nn
\end{equation}
which combined with \eqref{hm:sslp_ldlan+1} verifies Assumption \ref{hm:lan_assump2_dd}(c), 
completing the proof.
\qed\end{proof}

\begin{rem}{\rm 
We refer to \cite{DuM83} and the references therein for the asymptotic normality 
in the low-frequency sampling case, where the Fisher information matrix is non-singular. 
As long as the high-frequency sampling is concerned, 
the constant asymptotic singularity also emerges in the case of $\beta$-stable 
subordinators (see \cite{Mas09_slp}), and 
we conjecture that this is the case for the whole class of stable {\lp es}.
}\qed\end{rem}

\begin{rem}{\rm 
If $X$ is the Cauchy L\'evy process such that $\mcl(X_{1})=S_{1}(\sig,0,\gam)$, then 
the LAN holds at each $\theta:=(\sig,\gam)$ 
with rate $\sqrt{n}$ and Fisher information matrix $\diag\{1/(2\sig^{2}),1/(2\sig^{2})\}$; 
we refer to \cite[Sections 3.4 and 4.2]{Mas09_slp} for further exposition. 
Concerning the MLE $\hat{\theta}_{n}$ of $(\sig,\gam)$, 
Theorem \ref{hm:thm_uan} ensures the asymptotic efficiency as well as 
the uniform asymptotic normality 
$\sqrt{n}(\hat{\theta}_{n}-\theta)\cil_{u}N_{2}(0,2\sig^{2}I_{2})$. 
}\qed\end{rem}

%%%%%

\subsection{Symmetric {\lm}}\label{hm:sec_sym.slp_me}

In this section, we discuss how to construct 
asymptotically normally distributed estimators of $\theta=(\beta,\sig,\gam)$ 
with non-singular asymptotic covariance matrix.

\subsubsection{Scenario for constructions of easy joint estimators}

As we have seen in Theorem \ref{hm:sslp_th1}, 
the MLE has a disadvantage for the joint estimation of the index $\beta$ and the scale $\sig$. 
More suitable for practical use would be to adopt an $M$-estimation based on moment fitting, 
with giving preference to simplicity of implementation over theoretical asymptotic efficiency. 

In view of the scaling property of $X$, the sample mean 
$\bar{X}_{n}:=T_{n}^{-1}\sum_{j=1}^{n}\dd_{j}X=T_{n}^{-1}X_{T_{n}}$ 
satisfies that
\begin{equation}
\mcl\left(T_{n}^{1-1/\beta}(\bar{X}_{n}-\gamma)\right)=S_{\beta}(\sig)
\nn%\label{hm:sslp_sm+1}
\end{equation}
for each $n\in\mbbn$ under $P_{\theta}$. 
We immediately notice the following unpleasant features:
\begin{itemize}
\item $\bar{X}_{n}$ has infinite variance for each $n\in\mbbn$;

\item $\bar{X}_{n}$ is $T_{n}^{1-1/\beta}$-consistent only when $\beta>1$ and $T_{n}\to\infty$;

\item Since $T_{n}^{1-1/\beta}/(\sqrt{n}h_{n}^{1-1/\beta})=n^{1/2-1/\beta}\to 0$ for $\beta<2$, 
$\bar{X}_{n}$ is not rate-optimal for estimating $\gam$ (see Theorem \ref{hm:sslp_th1}).

\end{itemize}
Hence the sample mean is of rather limited use as an estimator of $\gamma$, and we need something else. 
In what follows, we will prove that the sample median based estimator 
(equivalently, the least absolute deviation estimator) of $\gamma$ attains 
the optimal rate $\sqrt{n}h_{n}^{1-1/\beta}$ whatever $\beta\in(0,2)$ is. 
Then, it will be used to construct an asymptotically normally distributed joint estimator of $\theta$, 
which has a non-singular asymptotic covariance matrix 
in compensation for the optimal $\sqrt{n}\log(1/h_{n})$-rate of convergence in estimating $\beta$.

In the rest of this section we fix a $\theta\in\Theta$ as the true value, 
and the stochastic convergences are taken under $P_{\theta}$.

\medskip

In order to construct estimators via moment fitting, we will of course make use of the law of large numbers 
for $\{Y_{nj}(\theta)\}$ of \eqref{hm:sslp_Y}:
\begin{equation}
\frac{1}{n}\sumj f\left(h_{n}^{-1/\beta}(\dd_{j}X-\gamma h_{n})\right)
\cip\int f(y)\phi_{\beta}(y;\sig)dy
\label{hm:sslp_me+1}
\end{equation}
for suitable functions $f$. The point here is two-fold: 
first, we can replace the unknown $\gamma$ in \eqref{hm:sslp_me+1} 
by the sample median $\hat{\gamma}_{n}$ 
(to be defined later) by virtue of the forthcoming Theorem \ref{hm:smed_bthm}; 
second, applying the delta method suitably we can eliminate 
the effect of $\beta$ involved in the summand $f(h_{n}^{-1/\beta}(\dd_{j}X-\gamma h_{n}))$. 
Specifically, we will consider the two moment-matching procedures:
\begin{itemize}
\item the logarithmic moments 
%$\int\{\log|y|\}^{k}\phi_{\beta}(y;\sig)dy$ for $k=1,2$ 
$f(x)=\log|y|$ and $(\log|y|)^{2}$;
\item the lower-order fractional moments 
%$\int|y|^{r}\phi_{\beta}(y)dy$ and $\int|y|^{2r}\phi_{\beta}(y)dy$ 
$f(x)=|y|^{r}$ and $|y|^{2r}$ for $r\in(0,\beta/6)$.
\end{itemize}

%%%

\subsubsection{Median-adjusted central limit theorem}

Here we step away from the main context, and prove a simple yet unusual central limit theorem, 
which will play an important role later on. 

\medskip

The setting we consider in this section is as follows. 
Let $Y_{1},Y_{2},\dots$ be an i.i.d. sequence of $\mbbr$-valued random variables 
with common continuous Lebesgue density $f$, and let $m_{k}$ denote the sample median of 
$(Y_{j})_{j=1}^{n}$ with odd $n=2k+1$, i.e., 
\begin{equation}
m_{k}:=Y_{n,k+1},
\nonumber
\end{equation}
where $Y_{n1}<Y_{n2}<\dots<Y_{nn}$ denotes the ordered sample. 
Let $g=(g_{l})_{l=1}^{L}:\mbbr\to\mbbr^{L}$ be a measurable function. 
The objective here is to derive a central limit theorem for 
$(\sqrt{2k}(G_{k}-\rho),~\sqrt{2k}(m_{k}-\gam))$ where
\begin{align}
G_{k}&:=\frac{1}{2k}\sum_{j=1,\dots,2k+1 \atop j\ne k+1}g(Y_{nj}-m_{k}),
\nn\\
\rho&:=\int g(y-\gamma)f(y)dy,
\nonumber
\end{align}
and $\gam$ denotes the median of $\mcl(Y_{1})$. 
The form of $G_{k}$, where the term $g(Y_{n,k+1}-m_{k})\equiv g(0)$ is eliminated from 
the sum, enables us to deal with cases where $g(0)$ cannot be defined, 
such as $g_{l}(y)=\log|y|$ and $g_{l}(y)=|y|^{-\del}$ for some $\del>0$.

If $g$ happens to be sufficiently smooth and integrable and if $f(\gam)>0$, 
then Taylor's formula and the stochastic expansion for $\sqrt{2k}(m_{k}-\gam)$ 
(e.g., \cite[Example 5.24]{vdV98}) yield that
\begin{align}
&{}\left(\sqrt{2k}(G_{k}-\rho),~\sqrt{2k}(m_{k}-\gam)\right) \nn\\
&=\frac{1}{\sqrt{2k}}\sum_{j=1,\dots,2k+1 \atop j\ne k+1}
\left(
\begin{array}{cc}
\ds{g(Y_{j}-\gam)-\rho-\frac{\rho_{1}}{2f(\gam)}\sgn(Y_{j}-\gam)} \\
\ds{\frac{1}{2f(\gam)}\sgn(Y_{j}-\gam)}
\end{array}
\right)+o_{p}(1),
\nonumber
\end{align}
where $\rho_{1}:=\int\p g(y-\gam)f(y)dy$, for which we can readily apply the usual central limit theorem. 
However, it is not so clear what will occur when $g$ is not smooth enough. 
Theorem \ref{hm:smed_bthm} below clarifies that we do have a central limit theorem in that case as well, 
provided that $f$ satisfies a mild smoothness assumption.

\begin{thm}
Assume that
\begin{description}
\item[(a)] $f$ is of class $\mcc^{1}(\mbbr)$ and admits a unique median $\gamma$ for which $f(\gamma)>0$, 
and there exist a constant $\ep_{0}>0$ and a Lebesgue integrable function $\zeta$ such that 
\begin{equation}
\sup_{|a|\le\ep_{0}}\left|\p_{y}f(y+a)\right|\le\frac{\zeta(y)}{1+|g(y-\gam)|^{3}},\qquad y\in\mbbr;
\nonumber
\end{equation}

\item[(b)] $\ds{\int|g(y-\gamma)|^{3}f(y)dy<\infty}$.
\end{description}
Then, we have
\begin{equation}
\left(\sqrt{2k}(G_{k}-\rho),~\sqrt{2k}(m_{k}-\gamma)\right)\cil
N_{L+1}(0,\Sigma)
\label{hm:smed_bthm_eq01}
\end{equation}
as $k\to\infty$, where
$\ds{\Sig=\left(
\begin{array}{cc}
\Sig_{11} & \Sig_{12} \\
\text{{\rm sym.}} & \Sig_{22}
\end{array}
\right)
}$ is given by
\begin{align}
\Sig_{11}&:=\int\left\{g(y-\gam)-\rho\right\}^{\otimes 2}f(y)dy
+\frac{1}{4f(\gam)^{2}}\left(\int g(y-\gam)\p f(y)dy\right)^{\otimes 2}
\nn\\
&{}\qquad+
\frac{1}{f(\gam)}\int g(y-\gam)\p f(y)dy
\nn\\
&{}\qquad\qquad\cdot
\left(\int_{\gam}^{\infty}g(y-\gam)f(y)dy-\int_{-\infty}^{\gam}g(y-\gam)f(y)dy\right)^{\top},
\nn\\
\Sig_{12}&:=\frac{1}{2f(\gam)}\left(\int_{\gam}^{\infty}g(y-\gam)f(y)dy-\int_{-\infty}^{\gam}g(y-\gam)f(y)dy\right)
\nn\\
&{}\qquad
+\frac{1}{4f(\gam)^{2}}\int g(y-\gam)\p f(y)dy,
\nn\\
\Sig_{22}&:=\frac{1}{4f(\gam)^{2}}.
\nn
\end{align}
\label{hm:smed_bthm}
\end{thm}

Theorem \ref{hm:smed_bthm} allows the function $g=(g_{l})_{l\le L}$ to be non-differentiable. 
If $\mcl(Y_{1})$ admits a finite variance, $L=1$, and $g_{1}(y)=|y|$, then the claim of 
Theorem \ref{hm:smed_bthm} reduces to the result of Zeigler \cite{Zei50} 
since we have $\int |y-\gam|\p f(y)dy=0$.

\begin{proof}[Theorem \ref{hm:smed_bthm}]
Let $\vp_{k}$ denote the characteristic function 
of the random variable on the left-hand side of (\ref{hm:smed_bthm_eq01}), 
and fix any $u=(v^{\top},u_{L+1})=(v_{1},\dots,v_{L},u_{L+1})\in\mbbr^{L+1}$ in the sequel. 
Write
\begin{equation}
a_{k}(y)=a_{k}(y,v):=\exp\bigg(\frac{i}{\sqrt{2k}}(g(y)-\rho)[v]\bigg).
\nonumber
\end{equation}
Then we have
\begin{align}
& \vp_{k}(u)\nn\\
&=E\left\{\exp\left(i\sqrt{2k}
(G_{k}-\rho)[v]+iu_{L+1}\sqrt{2k}(m_{k}-\gamma)\right)\right\} \nonumber \\
&=E\left[\exp\left\{\frac{i}{\sqrt{2k}}\left(
\sum_{j=1,\dots,2k+1 \atop j\ne k+1}(g(Y_{nj}-m_{k})-\rho)
\right)[v]+iu_{L+1}\sqrt{2k}(m_{k}-\gamma)\right\}\right] \nonumber \\
&=E\bigg\{\bigg(\prod_{j=1}^{k}a_{k}(Y_{nj}-m_{k})\bigg)
\bigg(\prod_{j=k+1}^{2k+1}a_{k}(Y_{nj}-m_{k})\bigg)
\exp\left(iu_{L+1}\sqrt{2k}(m_{k}-\gamma)\right)\bigg\}.
\label{hm:smed_bthm_eq1}
\end{align}
We will show that $\lim_{k\to\infty}\vp_{k}(u)=\exp\{-(\Sig/2)[u,u]\}$.

\medskip

The joint distribution $\mcl(Y_{n1},\dots,Y_{nn})$ admits the density
\begin{equation}
(y_{j})_{j\le n}\mapsto(2k+1)!\prod_{j=1}^{2k+1}f(y_{j}),\qquad
y_{1}\le y_{2}\le\dots\le y_{2k+1},
\label{hm:smed_bthm_eq2}
\end{equation}
and moreover the variables $Y_{n1},\dots,Y_{nn}$ form a Markov chain, 
so that $(Y_{n1},\dots,Y_{nk})$ and $(Y_{n,k+2},\dots,Y_{nn})$ are independent conditional on $Y_{k+1}$; 
see \cite[Chapter 2]{DavNag03}. 
Substituting (\ref{hm:smed_bthm_eq2}) in (\ref{hm:smed_bthm_eq1}) and then changing the 
order of the integrations so that the integration with respect to $m_{k}$ is carried out last, 
we get
\begin{align}
\vp_{k}(u)&=\int(2k+1)!\bigg\{
\int_{-\infty}^{m}\int_{-\infty}^{y_{k}}\dots\int_{-\infty}^{y_{2}}
\bigg(\prod_{j=1}^{k}a_{k}(y_{j}-m)\bigg)f(y_{1})dy_{1}\dots f(y_{k})dy_{k}
\bigg\} \nonumber \\
&\qquad{}\cdot\bigg\{
\int_{m}^{\infty}\int_{y_{k+2}}^{\infty}\dots\int_{y_{2k}}^{\infty}
\bigg(\prod_{j=k+2}^{2k+1}a_{k}(y_{j}-m)\bigg)f(y_{k+2})dy_{k+2}\dots f(y_{2k+1})dy_{2k+1}\bigg\} \nonumber \\
& \qquad{}\cdot \exp\left\{iu_{L+1}\sqrt{2k}(m-\gamma)\right\}f(m)dm \nonumber \\
&=:\int(2k+1)! A^{-}_{k}(m)A^{+}_{k}(m)\exp\left\{iu_{L+1}\sqrt{2k}(m-\gamma)\right\}f(m)dm.
\label{hm:smed_bthm_eq3}
\end{align}
Put $\xi_{k}(y_{2})=\int_{-\infty}^{y_{2}}a_{k}(y_{1}-m)f(y_{1})dy_{1}$, 
then $\lim_{y\to-\infty}\xi_{k}(y)=0$ and the integration by parts for the Lebesgue-Stieltjes integral yields that 
$\int_{-\infty}^{y_{3}}\xi_{k}(y_{2})a_{k}(y_{2}-m)f(y_{2})dy_{2}=\int_{-\infty}^{y_{3}}\xi_{k}(y_{2})\xi_{k}(dy_{2})
=(1/2)\xi_{k}(y_{3})^{2}$. 
Repeating this inductively and also handling $A^{+}_{k}(m)$ in a similar manner, we get
\begin{align}
A_{k}^{-}(m)&=\frac{1}{k!}\left(\int_{-\infty}^{m}a_{k}(y-m)f(y)dy\right)^{k},
\nn\\
A_{k}^{+}(m)&=\frac{1}{k!}\left(\int_{m}^{\infty}a_{k}(y-m)f(y)dy\right)^{k}.
\nonumber
\end{align}
Substituting these two expressions in (\ref{hm:smed_bthm_eq3}) and then going through the change of variable 
$s=\sqrt{2k}(m-\gamma)$, we can continue as follows:
\begin{align}
\vp_{k}(u)&=\int\frac{(2k+1)!}{(2^{k})^{2}(k!)^{2}}
\bigg(2\int_{-\infty}^{m}a_{k}(y-m)f(y)dy\bigg)^{k}
\bigg(2\int_{m}^{\infty}a_{k}(y-m)f(y)dy\bigg)^{k} \nonumber \\
&{}\quad\cdot\exp\left\{iu_{L+1}\sqrt{2k}(m-\gamma)\right\}f(m)dm \nonumber \\
&=\int\frac{(2k+1)!}{\sqrt{2k}(2^{k})^{2}(k!)^{2}}
\bigg[\bigg\{2\int_{-\infty}^{\gamma+s/\sqrt{2k}}a_{k}
\bigg(y-\gamma-\frac{s}{\sqrt{2k}}\bigg)f(y)dy\bigg\}^{k} \nonumber \\
& \quad{}\cdot\bigg\{2\int_{\gamma+s/\sqrt{2k}}^{\infty}a_{k}
\bigg(y-\gamma-\frac{s}{\sqrt{2k}}\bigg)f(y)dy\bigg\}^{k}
\exp(iu_{L+1}s)f\bigg(\gamma+\frac{s}{\sqrt{2k}}\bigg)\bigg]ds \nonumber \\
&=:\int C_{k}\bigg[\{B_{1,k}(s)\}^{k}\{B_{2,k}(s)\}^{k}\exp(iu_{L+1}s)f\bigg(\gamma+\frac{s}{\sqrt{2k}}\bigg)\bigg]ds
\nn\\
&=:\int C_{k}\Xi_{k}(s)ds.
\label{hm:smed_bthm_eq4}
\end{align}
We know Stirling's formula $C_{k}\sim\sqrt{2/\pi}$, hence it suffices to look at the integral $\int\Xi_{k}(s)ds$.

Let $\del_{k}(s):=\int^{\gamma+s/\sqrt{2k}}_{\gam}f(y)dy$, 
with the standard convention $\int_{a}^{b}=-\int_{b}^{a}$ for $a>b$. Then
\begin{align}
|\Xi_{k}(s)|
&\le f\bigg(\gamma+\frac{s}{\sqrt{2k}}\bigg)
\left(2\int^{\gamma+s/\sqrt{2k}}_{-\infty}f(y)dy\right)^{k}
\left(2\int_{\gamma+s/\sqrt{2k}}^{\infty}f(y)dy\right)^{k}
\nn\\
&=f\bigg(\gamma+\frac{s}{\sqrt{2k}}\bigg)\left\{1-(2\del_{k}(s))^{2}\right\}^{k}
\label{hm:smed_bthm_rev+1}\\
&\le\|f\|_{\infty}
\left\{1-\frac{1}{k}\left(2\sqrt{k}\del_{k}(s)\right)^{2}\right\}^{k}.
\label{hm:smed_bthm_rev+2}
\end{align}
Pick a constant $\del_{0}>0$ such that $\inf_{y: |y-\gam|\le\del_{0}}f(y)\ge f(\gam)/\sqrt{2}>0$, 
and introduce the event $\mcs_{k}:=\{t\in\mbbr:~|t|/\sqrt{2k}\le\del_{0}\}$. 
Since $c_{0}:=\sup_{k}\sup_{s\in\mcs_{k}^{c}}\{1-(2\del_{k}(s))^{2}\}\in(0,1)$, 
\eqref{hm:smed_bthm_rev+1} gives
\begin{equation}
\int_{\mcs_{k}^{c}}|\Xi_{k}(s)|ds\le c_{0}^{k}\int f\bigg(\gamma+\frac{s}{\sqrt{2k}}\bigg)ds
=c_{0}^{k}\sqrt{2k}\to 0,
\nonumber
\end{equation}
hence $\int_{\mcs_{k}^{c}}\Xi_{k}(s)ds\to 0$. 
Moreover, we have $2\sqrt{k}\del_{k}(s)\ge f(\gam)|s|$ whenever $s\in\mcs_{k}$, 
which together with \eqref{hm:smed_bthm_rev+2} implies that 
$\sup_{k}|\Xi_{k}(s)|\mathbf{1}_{\mcs_{k}}(s)\lesssim\exp\{-f(\gam)^{2}|s|^{2}\}$. 
To apply the dominated convergence theorem for $\int_{\mcs_{k}}\Xi_{k}(s)ds$, 
it remains to look at the point-wise convergence of $\Xi_{k}(\cdot)$.

\medskip

Fix any $s\in\mbbr$ and let $k$ be large enough so that $s\in\mcs_{k}$. 
First we look at $B_{1,k}(s)$. Divide the domain of integration as
\begin{align}
B_{1,k}(s)&=2\int_{-\infty}^{\gamma}a_{k}\bigg(y-\gamma-\frac{s}{\sqrt{2k}}\bigg)f(y)dy
+2\int_{\gamma}^{\gamma+s/\sqrt{2k}}a_{k}\bigg(y-\gamma-\frac{s}{\sqrt{2k}}\bigg)f(y)dy \nonumber \\
&=:B_{11,k}(s)+B_{12,k}(s).
\label{hm:smed_bthm_eq4.5}
\end{align}
For brevity, we write $b(y)=(g(y)-\rho)[v]$, so that $a_{k}(y)=\exp\{(i/\sqrt{2k})b(y)\}$. 
Using the inequality
\begin{equation}
\bigg|e^{iv}-\sum_{j=0}^{M}\frac{(iv)^{j}}{j!}\bigg|\le\frac{|v|^{M+1}}{(M+1)!},\quad M\in\mbbn,
\nonumber
\end{equation}
and the identity $\int_{-\infty}^{\gamma}f(y)dy=1/2$, we have
\begin{align}
B_{11,k}(s)&=1+\frac{2i}{\sqrt{2k}}\int_{-\infty}^{\gamma}
b\bigg(y-\gamma-\frac{s}{\sqrt{2k}}\bigg)f(y)dy \nonumber \\
&{}\quad-\frac{1}{2k}\int_{-\infty}^{\gamma}
b\bigg(y-\gamma-\frac{s}{\sqrt{2k}}\bigg)^{2}f(y)dy
+\frac{1}{k^{3/2}}\int_{-\infty}^{\gam}r_{k}(y,s)f(y)dy
\nn\\
&=1+\frac{2i}{\sqrt{2k}}\int_{-\infty}^{\gamma}
b\bigg(y-\gamma-\frac{s}{\sqrt{2k}}\bigg)f(y)dy \nonumber \\
&{}\quad-\frac{1}{2k}\int_{-\infty}^{\gamma}
b\bigg(y-\gamma-\frac{s}{\sqrt{2k}}\bigg)^{2}f(y)dy+O(k^{-3/2}).
\label{hm:smed_bthm_eq5}
\end{align}
Here, the second equality follows from the estimate 
$|r_{k}(y,s)|\lesssim |b(y-\gamma-s/\sqrt{2k})|^{3}$ and the assumption (a):
\begin{align}
\int_{-\infty}^{\gam}|r_{k}(y,s)|f(y)dy
&\lesssim 1+\int|g(z-\gamma)-\rho|^{3}f\left(z+\frac{s}{\sqrt{2k}}\right)dz \nn\\
&\lesssim 1+\int|g(z-\gamma)-\rho|^{3}f(z)dz+\frac{s}{\sqrt{2k}}\int\zeta(z)dz
\nn\\
&<\infty.
\nonumber
\end{align}
For $B_{12,k}(s)$ we proceed in a similar manner to the above:
\begin{align}
B_{12,k}(s)
&=2\int_{\gam}^{\gam+s/\sqrt{2k}}f(y)dy+
\frac{2i}{\sqrt{2k}}\int_{\gam}^{\gam+s/\sqrt{2k}}b\bigg(y-\gamma-\frac{s}{\sqrt{2k}}\bigg)f(y)dy
\nn\\
&{}\qquad-\frac{1}{2k}\int_{\gam}^{\gam+s/\sqrt{2k}}b\bigg(y-\gamma-\frac{s}{\sqrt{2k}}\bigg)^{2}f(y)dy
+O(k^{-3/2})
\nn\\
&=2\int_{\gam}^{\gam+s/\sqrt{2k}}\!\!f(y)dy+
\frac{2i}{\sqrt{2k}}\int_{\gam}^{\gam+s/\sqrt{2k}}\!\!\!b\bigg(y-\gamma-\frac{s}{\sqrt{2k}}\bigg)f(y)dy
+O(k^{-3/2})
\nn\\
&=\frac{2}{\sqrt{2k}}\int_{0}^{s}f\bigg(\gamma+\frac{z}{\sqrt{2k}}\bigg)dz \nonumber \\
&\qquad
+\frac{i}{k}\int_{0}^{s}b\bigg(\frac{z-s}{\sqrt{2k}}\bigg)f\bigg(\gamma+\frac{z}{\sqrt{2k}}\bigg)dz
+O(k^{-3/2}). \label{hm:smed_bthm_eq6}
\end{align}
Combining (\ref{hm:smed_bthm_eq4.5}), (\ref{hm:smed_bthm_eq5}), and (\ref{hm:smed_bthm_eq6}) gives
\begin{equation}
B_{1,k}(s)=1+\frac{2}{\sqrt{2k}}c_{11,k}(s)-\frac{1}{2k}c_{12,k}(s)+O(k^{-3/2}),
\label{hm:smed_bthm_eq7}
\end{equation}
where
\begin{align}
c_{11,k}(s)&:=i\int_{-\infty}^{\gamma}b\bigg(y-\gamma-\frac{s}{\sqrt{2k}}\bigg)f(y)dy
+\int_{0}^{s}f\bigg(\gamma+\frac{z}{\sqrt{2k}}\bigg)dz, \nonumber \\
c_{12,k}(s)&:=\int_{-\infty}^{\gamma}b\bigg(y-\gamma-\frac{s}{\sqrt{2k}}\bigg)^{2}f(y)dy
-2i\int_{0}^{s}b\bigg(\frac{z-s}{\sqrt{2k}}
\bigg)f\bigg(\gamma+\frac{z}{\sqrt{2k}}\bigg)dz.
\nonumber
\end{align}

\medskip

We deal with $B_{2,k}(s)$ in the same way: 
dividing the domain of integration $\int_{\gamma+s/\sqrt{2k}}^{\infty}$ as 
$\int_{\gamma}^{\infty}-\int_{\gamma}^{\gamma+s/\sqrt{2k}}$, we get
\begin{equation}
B_{2,k}(s)=1+\frac{2}{\sqrt{2k}}c_{21,k}(s)-\frac{1}{2k}c_{22,k}(s)+O(k^{-3/2}),
\label{hm:smed_bthm_eq8}
\end{equation}
where
\begin{align}
c_{21,k}(s)&:=i\int_{\gamma}^{\infty}b\bigg(y-\gamma-\frac{s}{\sqrt{2k}}\bigg)f(y)dy
-\int_{0}^{s}f\bigg(\gamma+\frac{z}{\sqrt{2k}}\bigg)dz, \nonumber \\
c_{22,k}(s)&:=\int_{\gamma}^{\infty}b\bigg(y-\gamma-\frac{s}{\sqrt{2k}}\bigg)^{2}f(y)dy
+2i\int_{0}^{s}b\bigg(\frac{z-s}{\sqrt{2k}}
\bigg)f\bigg(\gamma+\frac{z}{\sqrt{2k}}\bigg)dz.
\nonumber
\end{align}
From (\ref{hm:smed_bthm_eq7}) and (\ref{hm:smed_bthm_eq8}), we arrive at
\begin{align}
B_{1,k}(s)B_{2,k}(s)&=1-\frac{1}{k}\bigg[
-\sqrt{2k}\{c_{11,k}(s)+c_{21,k}(s)\} \nonumber \\
&{}\quad+\frac{1}{2}\{c_{12,k}(s)-4c_{11,k}(s)c_{21,k}(s)+c_{22,k}(s)\}+O(k^{-1/2})
\bigg].
\label{hm:smed_bthm_eq9}
\end{align}
Let us obtain the limit of the $[\cdots]$ part in \eqref{hm:smed_bthm_eq9}. Direct computation gives
\begin{equation}
\sqrt{2k}\left\{c_{11,k}(s)+c_{21,k}(s)\right\}=is
\left\{\int g(y-\gam)\int_{0}^{1}\p f\left(y+\ep\frac{s}{\sqrt{2k}}\right)d\ep dy\right\}[v].
\nonumber
\end{equation}
Under the assumption (a) we can apply the dominated convergence theorem to conclude that
\begin{align}
\sqrt{2k}\left\{c_{11,k}(s)+c_{21,k}(s)\right\}
=is\left(\int g(y-\gam)\p f(y)dy\right)[v]+o(1).
\label{hm:smed_bthm_eq10}
\end{align}
Making use of a similar argument for the integrals 
(plugging in $\int_{0}^{s}f(\gamma+z/\sqrt{2k})dz=sf(\gam)+O(k^{-1/2})$, 
$\int_{\gam}^{\infty}b(y-\gamma-s/\sqrt{2k})^{l}f(y)dy
=\int_{\gam}^{\infty}b(y-\gamma-s/\sqrt{2k})^{l}f(y)dy+O(k^{-1/2})$, and so on), 
we can proceed as
\begin{align}
& \frac{1}{2}\left\{c_{12,k}(s)-4c_{11,k}(s)c_{21,k}(s)+c_{22,k}(s)\right\} \nonumber \\
&=\frac{1}{2}\bigg[\int b(y-\gamma)^{2}f(y)dy
+4\left(\int_{-\infty}^{\gam}b(y-\gamma)f(y)dy\right)\left(\int_{\gam}^{\infty}b(y-\gamma)f(y)dy\right)
\nn\\
&{}\qquad+4s^{2}f(\gam)^{2}
-4isf(\gam)\left(\int_{\gam}^{\infty}b(y-\gamma)f(y)dy
-\int_{-\infty}^{\gam}b(y-\gamma)f(y)dy\right)
\bigg]\nn\\
&{}\qquad+O(k^{-1/2}).
\nn\\
&=\frac{1}{2}M[v,v]+2\bigg\{s^{2}f(\gamma)^{2}
+(H_{-}H_{+}^{\top})[v,v]
-isf(\gamma)(H_{+}-H_{-})[v]\bigg\}+O(k^{-1/2}),
\label{hm:smed_bthm_eq11}
\end{align}
where we wrote 
$M=\int\{g(y-\gamma)-\rho\}^{\otimes 2}f(y)dy$, 
$H_{-}=\int_{-\infty}^{\gamma}\{g(y-\gamma)-\rho\}f(y)dy$, 
and $H_{+}=\int_{\gamma}^{\infty}\{g(y-\gamma)-\rho\}f(y)dy$. 

Piecing together (\ref{hm:smed_bthm_eq9}), (\ref{hm:smed_bthm_eq10}) and (\ref{hm:smed_bthm_eq11}) 
yields that for each $s$,
\begin{align}
& \{B_{1,k}(s)B_{2,k}(s)\}^{k} \nn\\
&=\exp\bigg\{-\frac{1}{2}M[v,v]-2s^{2}f(\gamma)^{2}-2(H_{-}H_{+}^{\top})[v,v] \nn\\
&{}\qquad
+2isf(\gamma)(H_{+}-H_{-})[v]+is\left(\int g(y-\gam)\p f(y)dy\right)[v]\bigg\}+o(1),
\label{hm:smed_bthm_eq12}
\end{align}
which specifies the point-wise limit of $\Xi_{k}(\cdot)$.

Now we can apply the dominated convergence theorem: 
substituting (\ref{hm:smed_bthm_eq12}) in (\ref{hm:smed_bthm_eq4}), 
eliminating the Gaussian-density factor integrating to $1$, and noting that $H_{+}+H_{-}=0$, 
we finally obtain
\begin{align}
\vp_{k}(u)&\to\int\sqrt{\frac{2}{\pi}}
\exp\bigg\{
-\frac{1}{2}M[v,v]-2(H_{-}H_{+}^{\top})[v,v]
\nonumber \\
& \qquad{}-2s^{2}f(\gamma)^{2}+2isf(\gamma)(H_{+}-H_{-})[v]
\nn\\
&{}\qquad+is\left(\int g(y-\gam)\p f(y)dy,~1\right)[u]\bigg\}f(\gamma)ds \nonumber \\
&=\exp\bigg\{-\frac{1}{2}\bigg(\Sig_{11}[v,v]+2\Sig_{12}[v,u_{L+1}]+\Sig_{22}u_{L+1}^{2}\bigg)\bigg\}
\nonumber \\
&=\exp\bigg(-\frac{1}{2}\Sig[u,u]\bigg).
\nonumber
\end{align}
The proof is complete.
\qed\end{proof}

%%%

\subsubsection{Rate-efficient sample median}

Let us return to our model. We will keep to set $n=2k+1$ in what follows. 
Denote by $Y_{nj}(\beta,\gamma)$ 
the order statistics of the $S_{\beta}(\sig)$-i.i.d. array $\{h_{n}^{-1/\beta}(\dd_{j}X-\gamma h_{n})\}_{j=1}^{n}$: 
$Y_{n1}(\beta,\gamma)<Y_{n2}(\beta,\gamma)<\dots<Y_{nn}(\beta,\gamma)$ a.s. 
Let $m_{n}$ denote the sample median of $(\dd_{j}X)_{j=1}^{n}$ and
\begin{equation}
\hat{\gamma}_{n}:=\frac{1}{h_{n}}m_{n}.
\label{hm:sslp_smed}
\end{equation}
Observe that the central limit theorem \eqref{hm:smed_bthm_eq01}, 
or the standard asymptotic theory for the least absolute deviation estimator, gives
\begin{align}
\sqrt{n}h_{n}^{1-1/\beta}(\hat{\gamma}_{n}-\gamma)
&=\sqrt{n}h_{n}^{-1/\beta}(m_{n}-\gamma h_{n}) \nonumber \\
&=\sqrt{n}\{Y_{n,k+1}(\beta,\gamma)-0\} \nonumber \\
&=\sqrt{n}\left\{\left(\textrm{Sample median of }
\{h_{n}^{-1/\beta}(\dd_{j}X-\gamma h_{n})\}_{j\le n}\right)-0\right\} \nonumber \\
&\cil N_{1}(0,\{2\phi_{\beta}(0;\sig)\}^{-2}) \nn\\
&=N_{1}\bigg(0,\bigg(\frac{\sig\pi}{2\Gamma(1+1/\beta)}\bigg)^{2}\bigg). \nonumber
\end{align}
This means that the pretty simple statistic $\hat{\gam}_{n}$ 
serves as a rate-optimal and asymptotically normally distributed estimator of $\gamma$. 
We note that the unbiasedness of $\hat{\gamma}_{n}$ follows from the argument \cite[p.241]{Zol86}.

Figure \ref{hm:sslp_fig1} shows the asymptotic variance of $\sqrt{n}h_{n}^{1-1/\beta}(\hat{\gamma}_{n}-\gamma)$. 
It is expected that the estimator $\hat{\gam}_{n}$ shows good performance especially for small $\beta\in(0,1)$, 
since the rate of convergence of $\hat{\gamma}_{n}$ is then faster than $\sqrt{n}$ 
with the asymptotic variance being quite small; for $\sig>0$ fixed, the function 
$\beta\mapsto\{2\phi_{\beta}(0;\sig)\}^{-2}$ decreases to zero as $\beta\to 0$. 

\begin{figure}[htb!]
\centering
\includegraphics[width=10cm]{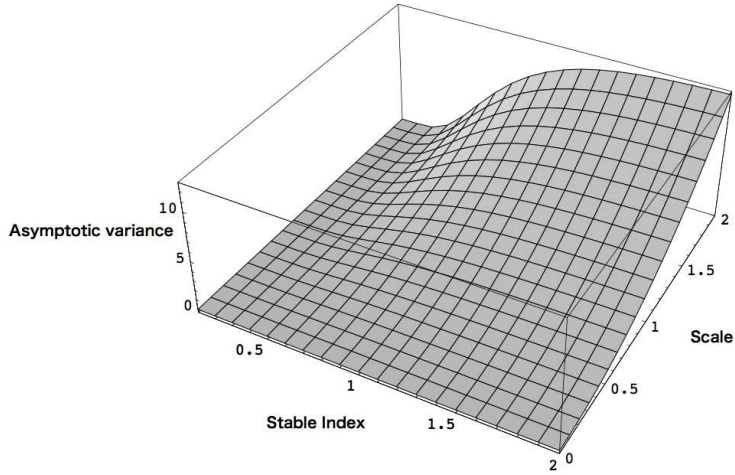}
\caption{
A plot of the asymptotic variance of $\sqrt{n}h_{n}^{1-1/\beta}(\hat{\gamma}_{n}-\gamma)$ 
as a function of $(\beta,\sig)$ on $(0,2)\times(0,2)$.}
\label{hm:sslp_fig1}
\end{figure}

Unlike with the Gaussian case we may consider a bounded-domain asymptotics (i.e. $\limsup_{n}T_{n}\lesssim 1$) 
for estimating $\gamma$, while 
the optimal rate $\sqrt{n}h_{n}^{1-1/\beta}$ may become arbitrarily slow as $\beta$ gets close to $2$; 
we have $\sqrt{n}h_{n}^{1-1/\beta}=T^{1-1/\beta}n^{1/\beta-1/2}$ if $T_{n}\equiv T$. 
This is in accordance with the fact that we need $T_{n}\to\infty$ for consistent estimation of 
the drift in the Gaussian case ($\beta=2$). 

%%%

\subsubsection{Preliminary formulation for estimating $\beta$ and $\sig$}

By the definition, for each $j\le n$
\begin{align}
& Y_{nj}(\beta,\gamma)-\left(\textrm{Sample median of }
\{h_{n}^{-1/\beta}(\dd_{j}X-\gamma h_{n})\}_{j=1}^{n}\right) \nonumber \\
&=Y_{nj}(\beta,\gamma)-Y_{n,k+1}(\beta,\gamma)
=Y_{nj}(\beta,\hat{\gamma}_{n}). \nonumber
\end{align}
Let $g_{l}$ be measurable functions symmetric around the origin such that
\begin{equation}
\int|g_{l}(y)|^{3}\phi_{\beta}(y;\sig)dy<\infty,\quad l=1,2.
\nonumber
\end{equation}
A direct application of Theorem \ref{hm:smed_bthm} with $L=2$ yields that
\begin{align}
S_{n}(\theta)&:=\sqrt{n}\left(
\begin{array}{c}
G^{(1)}_{n}(\theta)-\rho_{1}(\theta) \\
G^{(2)}_{n}(\theta)-\rho_{2}(\theta) \\
Y_{n,k+1}(\beta,\gamma)-0
\end{array}
\right)
\nn\\
&=\left(
\begin{array}{c}
\sqrt{n}\{G^{(1)}_{n}(\theta)-\rho_{1}(\theta)\} \\
\sqrt{n}\{G^{(2)}_{n}(\theta)-\rho_{2}(\theta)\} \\
\sqrt{n}h_{n}^{1-1/\beta}(\hat{\gamma}_{n}-\gamma)
\end{array}
\right)\cil N_{3}(0,\Sigma(\theta)),\quad n\to\infty,
\label{hm:sslp_clt1}
\end{align}
where
\begin{align}
\rho_{l}(\theta)&:=\int g_{l}(y)\phi_{\beta}(y;\sig)dy,\nn\\
G^{(l)}_{n}(\theta)&:=\frac{1}{2k}\bigg\{\sum_{j=1}^{k}
g_{l}(Y_{nj}(\beta,\hat{\gamma}_{n}))+\sum_{j=k+2}^{2k+1}g_{l}(Y_{nj}(\beta,\hat{\gamma}_{n}))\bigg\},
\nonumber
\end{align}
and the asymptotic covariance matrix $\Sig=[\Sig^{(pq)}(\theta)]_{p,q=1}^{3}$ is given by
\begin{equation}
\Sig^{(pq)}(\theta)=\left\{
\begin{array}{ll}
\displaystyle{\int\{g_{p}(y)-\rho_{p}(\theta)\}\{g_{q}(y)-\rho_{q}(\theta)\}\phi_{\beta}(y;\sig)dy}, 
& 1\le p,q\le 2, \\[3mm]
\displaystyle{0}, & p=3, 1\le q\le 2, \\[3mm]
\displaystyle{\{2\phi_{\beta}(0;\sig)\}^{-2}}, & p=q=3.
\end{array}
\right.
\label{hm:sslp_clt2}
\end{equation}
For convenience, we introduce the notation 
$\Sig^{\ast}(\theta)=[\Sig^{\ast}_{kl}(\theta)]_{k,l=1}^{2}\in\mbbr^{2}\otimes\mbbr^{2}$ 
for the second leading principal submatrix of $\Sig(\theta)$:
\begin{equation}
\Sig^{\ast}_{kl}(\theta):=\Sig^{(kl)}(\theta),\quad k,l=1,2.
\label{hm:sslp_clt4}
\end{equation}
Having (\ref{hm:sslp_clt1}) in hand, we can apply the delta method. 
Assume that the function $F(\theta):=(\rho_{1}(\theta),\rho_{2}(\theta),\gam)$ 
has an inverse $F^{-1}$ at $\theta=(\beta,\sig,0)$, 
and let $(\hat{\beta}_{n},\hat{\sig}_{n})$ denote a solution of the estimating equation
\begin{equation}
\left(
\begin{array}{c}
G^{(1)}_{n}(\theta)-\rho_{1}(\theta) \\
G^{(2)}_{n}(\theta)-\rho_{2}(\theta) 
\end{array}
\right)=\left(
\begin{array}{c}
0 \\
0
\end{array}
\right),
\label{hm:sslp_Geq}
\end{equation}
which uniquely exists with $P_{\theta}$-probability tending to $1$; 
note that \eqref{hm:sslp_Geq} is free of the unknown quantity $\gam$ since we have replaced it by $\hat{\gamma}_{n}$. 
Let
\begin{equation}
K(\theta):=\p_{\theta}F(\beta,\sig,0)=\left(
\begin{array}{ccc}
\p_{\beta}\rho_{1}(\theta) & \p_{\sig}\rho_{1}(\theta) & 0 \\
\p_{\beta}\rho_{2}(\theta) & \p_{\sig}\rho_{2}(\theta) & 0 \\
0 & 0 & 1
\end{array}
\right)=:\left(
\begin{array}{cc}
K^{\ast}(\theta) & 
\left.
\begin{array}{c}
0 \\
0
\end{array}
\right. \\
\left.
\begin{array}{cc}
0 & 0
\end{array}
\right. & 1
\end{array}
\right).
\label{hm:sslp_K_def}
\end{equation}
Then we obtain the joint estimator 
$\hat{\theta}_{n}=(\hat{\beta}_{n},\hat{\sig}_{n},\hat{\gamma}_{n})$ such that
\begin{align}
\left(
\begin{array}{c}
\sqrt{n}(\hat{\beta}_{n}-\beta) \\
\sqrt{n}(\hat{\sig}_{n}-\sig) \\
\sqrt{n}h_{n}^{1-1/\beta}(\hat{\gamma}_{n}-\gamma)
\end{array}
\right)&=\sqrt{n}\left\{
F^{-1}\left(
\begin{array}{c}
G^{(1)}_{n}(\theta) \\
G^{(2)}_{n}(\theta) \\
Y_{n,k+1}(\beta,\gamma)
\end{array}
\right)-F^{-1}\left(
\begin{array}{c}
\rho_{1}(\theta) \\
\rho_{2}(\theta) \\
0
\end{array}
\right)\right\} \nonumber \\
&=K(\theta)^{-1}S_{n}(\theta)+o_{p}(1) \nonumber \\
&\cil N_{3}\big(0,V(\theta)\big),
\label{hm:sslp_clt3}
\end{align}
where, in view of (\ref{hm:sslp_clt4}) and (\ref{hm:sslp_K_def}), the asymptotic variance 
$V(\theta)=[V_{kl}(\theta)]_{k,l=1}^{3}:=K(\theta)^{-1}\Sig(\theta)K(\theta)^{-1 \top}$ 
is block diagonal:
\begin{equation}
V(\theta)
=\diag\left(
K^{\ast}(\theta)^{-1}\Sig^{\ast}(\theta)K^{\ast}(\theta)^{-1 \top},
~\Sig^{(33)}(\theta)\right).
\label{hm:sslp_V}
\end{equation}
Here are some remarks on \eqref{hm:sslp_clt3}.
\begin{itemize}
\item The estimator $\hat{\theta}_{n}$ so constructed is rate-efficient for $(\sig,\gamma)$, 
while not so for $\beta$ (recall Theorem \ref{hm:sslp_th1}). 
\item The estimations of $(\beta,\sig)$ and $\gamma$ are asymptotically independent. 
\item Thanks to the $\sqrt{n}$-consistency of $\hat{\beta}_{n}$ and \eqref{hm:addeqnum-1}, 
we have $h_{n}^{1/\beta-1/\hat{\beta}_{n}}\cip 1$, 
so that $\sqrt{n}h_{n}^{1-1/\hat{\beta}_{n}}=\sqrt{n}h_{n}^{1-1/\beta}(1+o_{p}(1))$. 
Then we can readily get the $100\al\%$-confidence interval about $\gamma$:
\begin{equation}
\hat{\gam}_{n}-\frac{z_{\al/2}}{\sqrt{n}h_{n}^{1-1/\hat{\beta}_{n}}}
\bigg(\frac{\hat{\sig}_{n}\pi}{2\Gamma(1+1/\hat{\beta}_{n})}\bigg)
<\gam<\hat{\gam}_{n}+\frac{z_{\al/2}}{\sqrt{n}h_{n}^{1-1/\hat{\beta}_{n}}}
\bigg(\frac{\hat{\sig}_{n}\pi}{2\Gamma(1+1/\hat{\beta}_{n})}\bigg),
\nonumber
\end{equation}
with $z_{\al/2}$ denoting the upper $100\al/2$th percentile of $N_{1}(0,1)$.
\end{itemize}
In order to make use of \eqref{hm:sslp_clt3}, 
we are left to computing $[V_{kl}(\theta)]_{k,l=1}^{2}$ for each specific choices of $g_{l}$, $l=1,2$. 

%%%

\subsubsection{Logarithmic moments}\label{hm:sslp_logMEsec}

Set
\begin{equation}
g_{l}(y)=(\log|y|)^{l},\quad l=1,2.
\nonumber
\end{equation}
The distribution $S_{\beta}(\sig)$ admit finite 
logarithmic moments of any positive order, the first two being given by
\begin{align}
\int(\log|y|)\phi_{\beta}(y;\sig)dy&=\mathfrak{C}\bigg(\frac{1}{\beta}-1\bigg)+\log\sig, \nonumber \\
\int(\log|y|)^{2}\phi_{\beta}(y;\sig)dy&=\frac{\pi^{2}}{6}\bigg(\frac{1}{\beta^{2}}+\frac{1}{2}\bigg)
+\bigg\{\mathfrak{C}\bigg(\frac{1}{\beta}-1\bigg)+\log\sig\bigg\}^{2}, \nonumber
\end{align}
where $\mathfrak{C}$ ($\fallingdotseq 0.5772$) denotes the Euler constant; see \cite[p.69]{NikSha95}. 
Write $x_{n1}<\cdots<x_{nn}$ for the ordered $\dd_{1}X,\dots,\dd_{n}X$. 
Solving the corresponding (\ref{hm:sslp_Geq}) gives the explicit solutions
\begin{align}
\hat{\beta}_{\log,n}&:=
\bigg\{\frac{6}{2k\pi^{2}}\sum_{j\le n, j\ne k+1}\bigg(\log\big|h_{n}^{-1/\beta}
(x_{nj}-\hat{\gamma}_{n}h_{n})\big| \nonumber \\
& \qquad {}-\frac{1}{2k}\sum_{j\le n, j\ne k+1}\log\big|h_{n}^{-1/\beta}
(x_{nj}-\hat{\gamma}_{n}h_{n})\big|\bigg)^{2}-\frac{1}{2}\bigg\}^{-1/2} \nonumber \\
&=\bigg\{\frac{6}{2k\pi^{2}}\sum_{j\le n, j\ne k+1}
\bigg(\log|x_{nj}-\hat{\gamma}_{n}h_{n}| \nonumber \\
& \qquad {}
-\frac{1}{2k}\sum_{j\le n, j\ne k+1}\log|x_{nj}-\hat{\gamma}_{n}h_{n}|\bigg)^{2}-\frac{1}{2}\bigg\}^{-1/2}, 
\label{hm:sslp_logal} \\
\hat{\sig}_{\log,n}&:=\exp\bigg\{\frac{1}{2k}\sum_{j\le n, j\ne k+1}
\log\Big|h_{n}^{-1/\hat{\beta}_{\log,n}}
(x_{nj}-\hat{\gamma}_{n}h_{n})\Big|-\mathfrak{C}
\bigg(\frac{1}{\hat{\beta}_{\log,n}}-1\bigg)\bigg\} \nonumber \\
&=\exp\bigg\{\frac{1}{\hat{\beta}_{\log,n}}\log(1/h_{n})
+\frac{1}{2k}\sum_{j\le n, j\ne k+1}\log|x_{nj}-\hat{\gamma}_{n}h_{n}|
-\mathfrak{C}\bigg(\frac{1}{\hat{\beta}_{\log,n}}-1\bigg)\bigg\}.
\label{hm:sslp_logsig}
\end{align}
Observe that 
the unknown factor ``$h_{n}^{-1/\beta}$'' involved in $Y_{nj}(\beta,\gamma)$ were cancelled out 
in the computation of $\hat{\beta}_{\log,n}$, 
making the quantities (\ref{hm:sslp_logal}) and (\ref{hm:sslp_logsig}) usable.

\medskip

Let us compute the corresponding $[V_{kl}(\theta)]_{k,l=1}^{2}$. 
We denote by $\nu_{1}$ and $\nu_{k}$ ($k\ge 2$) the mean and $k$th central moments of 
$\log|Y|$ with $\mcl(Y)=S_{\beta}(\sig)$, respectively. Then,
\begin{align}
& \nu_{1}=\mathfrak{C}\bigg(\frac{1}{\beta}-1\bigg)+\log\sig, \quad 
\nu_{2}=\frac{\pi^{2}}{6}\bigg(\frac{1}{\beta^{2}}+\frac{1}{2}\bigg), \nn\\
& \nu_{3}=2\zeta(3)(\beta^{-3}-1), \quad 
\nu_{4}=\pi^{4}\bigg(\frac{3}{20}\beta^{-4}+\frac{1}{12}\beta^{-2}+\frac{19}{240}\bigg),
\nonumber
\end{align}
where $\zeta(\cdot)$ denotes Riemann's zeta function; 
$\zeta(3)\approx 1.202057$. From (\ref{hm:sslp_clt2}), we get 
$\Sig^{(11)}(\theta)=\nu_{2}$, $\Sig^{(12)}(\theta)=\nu_{3}+2\nu_{1}\nu_{2}$, 
and $\Sig^{(22)}(\theta)=\nu_{4}+4\nu_{1}^{2}\nu_{2}+4\nu_{1}\nu_{3}-\nu_{2}^{2}$. 
Further, we note that $\mathrm{det}\Sig(\theta)=\Sig_{33}(\theta)(\nu_{4}\nu_{2}-\nu_{3}^{2}-\nu_{2}^{3})>0$, 
and that 
\begin{equation}
\mathrm{det}K(\theta)=\mathrm{det}K^{\ast}(\theta)=\mathrm{det}
\left(
\begin{array}{cc}
-\mathfrak{C}\beta^{-2} & \sig^{-1} \\
-\pi^{2}\beta^{-3}/3-2\mathfrak{C}\beta^{-2}\nu_{1} & 2\sig^{-1}\nu_{1}
\end{array}
\right)=\frac{\pi^{2}}{3\beta^{3}\sig}>0.
\nonumber
\end{equation}
Therefore $V(\theta)$ is positive definite. 
After some computations we obtain the explicit expressions 
for the matrix $V(\theta)=:V^{\log}(\theta)=[V^{\log}_{kl}(\theta)]_{k,l=1}^{3}$:
\begin{align}
V_{11}^{\log}(\theta)&=\frac{11}{10}\beta^{2}+\frac{1}{2}\beta^{4}+\frac{13}{20}\beta^{6},
\nn\\
V_{12}^{\log}(\theta)&=\frac{\sig}{\pi^{4}}\left\{
9\mathfrak{C}\beta^{4}(\nu_{4}-\nu_{2}^{2})-3\pi^{2}\beta^{3}\nu_{3}\right\},
\nn\\
V_{22}^{\log}(\theta)&=\frac{\sig^{2}}{\pi^{4}}\left\{
9\mathfrak{C}^{2}\beta^{2}(\nu_{4}-\nu_{2}^{2})+\pi^{4}\nu_{2}-6\mathfrak{C}\pi^{2}\beta\nu_{3}\right\},
\nn\\
V_{13}^{\log}(\theta)&=V_{23}^{\log}(\theta)=0, \nn\\
V_{33}^{\log}(\theta)&=\left(\frac{\sig\pi}{2\Gam(1+1/\beta)}\right)^{2}.
\nn%\label{hm:sslp_logV}
\end{align}
Finally using the continuity of $\theta\mapsto V^{\log}(\theta)$, we arrive at the following.

\begin{thm} Fix any $\theta\in\Theta$ and define
\begin{equation}
\hat{\theta}_{\log,n}=(\hat{\beta}_{\log,n},\hat{\sig}_{\log,n},\hat{\gamma}_{n})
\label{hm:sslp_logME}
\end{equation}
by (\ref{hm:sslp_smed}), (\ref{hm:sslp_logal}) and (\ref{hm:sslp_logsig}). Then,
\begin{equation}
V^{\log}(\hat{\theta}_{\log,n})^{-1/2}
\diag\big(\sqrt{n},\sqrt{n},\sqrt{n}h_{n}^{1-1/\hat{\beta}_{\log,n}}\big)
(\hat{\theta}_{\log,n}-\theta)\cil N_{3}(0,I_{3}),
\label{hm:sslp_wclogME}
\end{equation}
where $V^{\log}(\theta)$ is positive-definite.
\label{hm:sslp_thm0}
\end{thm}

\begin{rem}{\rm 
Sometimes it would be more convenient to take the logarithm in estimating the positive quantity $\sig$ 
for approximate normality of $\hat{\sig}_{\log,n}$ in moderate sample size:
\begin{align}
&
\left\{\sqrt{n}\left(\hat{\beta}_{\log,n}-\beta\right),~
\sqrt{n}\left(\log\hat{\sig}_{\log,n}-\log\sig\right),~
\sqrt{n}h_{n}^{1-1/\hat{\beta}_{\log,n}}\left(\hat{\gamma}_{n}-\gamma\right)
\right\}
\nn\\
& \qquad{} \cil N_{3}\left(0,\left(
\begin{array}{ccc}
V^{\log}_{11}(\theta) & \sig^{-1}V^{\log}_{12}(\theta) & 0 \\
\sig^{-1}V^{\log}_{12}(\theta) & \sig^{-2}V^{\log}_{22}(\theta) & 0 \\
0 & 0 & V^{\log}_{33}(\theta).
\end{array}
\right)\right).
\nn
\end{align}
The second leading principal submatrix of the asymptotic covariance matrix is free of $\sig$, 
hence a function only of $\beta$. 
We also note that the variance-stabilizing transform for $\hat{\beta}_{\log,n}$ is available: 
we have $\sqrt{n}\{\Psi(\hat{\beta}_{\log,n})-\Psi(\beta)\}\cil N_{1}(0,1)$ for
\begin{equation}
\Psi(x):=\sqrt{\frac{5}{22}}\left\{
2\log x-\log\left(22+5x^{2}+\sqrt{22(22+10x^{2}+13x^{4})}\right)
\right\}.
\nonumber
\end{equation}
Moreover, since $\mcl(x_{nj}-\hat{\gamma}_{n}h_{n})
=\mcl\{(x_{nj}-\gam h_{n})-(x_{n,k+1}-\gam h_{n})\}=S_{\beta}(2h_{n}^{1/\beta}\sig)$ 
for $j\ne k+1$, we see that
\begin{align}
\hat{\beta}_{\log,n}^{-2}&=
\frac{6}{(2k-1)\pi^{2}}\sum_{j\le n, j\ne k+1}\bigg(\log|x_{nj}-\hat{\gamma}_{n}h_{n}|
\nn\\
&{}\qquad
-\frac{1}{2k}\sum_{j\le n, j\ne k+1}\log|x_{nj}-\hat{\gamma}_{n}h_{n}|\bigg)^{2}-\frac{1}{2},
\nonumber
\end{align}
which satisfies that 
$\sqrt{n}(\hat{\beta}_{\log,n}^{-2}-\beta^{-2})\cil N_{1}(0,4\beta^{-6}V^{\log}_{11}(\theta))$, 
is an unbiased estimator of $\beta^{-2}$.
\label{hm:sslp_rem_dm+1}
}\qed\end{rem}

\begin{rem}{\rm 
If we beforehand know the true value of $\sig$ for some reason, 
then it is possible to construct a rate-efficient estimator of $\beta$ simply via the logarithmic-moment fitting. 
In fact, simple manipulation leads to the relation
\begin{align}
& \hspace{-1cm}-\frac{1}{\sqrt{2k}}\sum_{j\le n, j\ne k+1}
\big(\log\big|h_{n}^{-1/\beta}(x_{nj}-\hat{\gamma}_{n}h_{n})\big|-\nu_{1}\big)
\nn\\
&=\sqrt{n}\log(1/h_{n})\bigg\{
\frac{-(S_{n}+\mathfrak{C}-\log\sig)}{\log(1/h_{n})-\mathfrak{C}}
-\frac{1}{\beta}\bigg\}\{1+o_{p}(1)\},
\nonumber
\end{align}
where $S_{n}:=(2k)^{-1}\sum_{j\le n,j\ne k+1}\log|x_{nj}-\hat{\gamma}_{n}h_{n}|$, 
but from Theorem \ref{hm:smed_bthm} we know that the left-hand side tends in distribution to $N_{1}(0,\nu_{2})$. 
It follows from Slutsky's theorem that
\begin{equation}
\tilde{\beta}_{n}(\sig):=\frac{\log(1/h_{n})-\mathfrak{C}}{(\log\sig)-\mathfrak{C}-S_{n}}
\label{hm:sslp_real2}
\end{equation}
can serve as an asymptotically normally distributed rate-efficient estimator. 
As is expected, the estimator $\tilde{\beta}_{n}(\sig)$ exhibits excellent finite-sample performance; 
see Tables \ref{hm:sslp_table1} and \ref{hm:sslp_table2} in Section \ref{hm:sslp_comparison_sim}.
\label{hm:rem_rate-eff-beta}
}\qed\end{rem}

%%%

\subsubsection{Lower-order fractional moments: power-variation statistics}\label{hm:sec_sslp_lfm}

Now let set
\begin{equation}
g_{l}(y)=|y|^{pl},\quad l=1,2,\quad p\in(0,\beta/6),
\nonumber
\end{equation}
in applying (\ref{hm:sslp_clt3}); 
we can also pick a $p\in(-1,0)$, but do not consider it here. 
Especially when $T_{n}\equiv T$, this setting is related to the power-variation statistics 
applicable to a general class of semimartingales driven by a stable {\lp}; 
see \cite{CorNuaWoe07} and \cite{Tod13} together with their references. 
When concerned with joint estimation of $\theta$, 
we should be careful in applying the power-variation result directly because the effect of $\gamma\ne 0$ 
may not be ignorable.

\medskip

We know that $\int|y|^{q}\phi_{\beta}(y;\sig)dy=C(\beta,q)\sig^{q}$ where (e.g., \cite[Section 3.3]{NikSha95})
\begin{equation}
C(\beta,q):=
\frac{2^{q}\Gamma\left((q+1)/2\right)\Gamma(1-q/\beta)}{\sqrt{\pi}\Gamma(1-q/2)}.
\nonumber
\end{equation}
Using Theorem \ref{hm:smed_bthm} together with the present choice of $g_{l}$ 
we obtain a moment estimator of $(\beta,\sig)$ 
as a solution $\hat{\theta}_{p,n}=(\hat{\beta}_{p,n},\hat{\sig}_{p,n})$ to
\begin{equation}
\frac{1}{n}\sumj|h_{n}^{-1/\beta}(\dd_{j}X-\hat{\gamma}_{n}h_{n})|^{pl}=C(\beta,pl)\sig^{pl},
\quad l=1,2.
\label{hm:sslp_lm2}
\end{equation}
The solution takes the convenient form:
\begin{equation}
\frac{H_{1n}^{2}}{H_{2n}}=\frac{C(\hat{\beta}_{p,n},p)^{2}}
{C(\hat{\beta}_{p,n},2p)},\qquad
\hat{\sig}_{p,n}=\bigg(\frac{h_{n}^{-p/\hat{\beta}_{p,n}}H_{1n}}{C(\hat{\beta}_{p,n},p)}\bigg)^{1/p},
\label{hm:sslp_lm3-4}
\end{equation}
where
\begin{equation}
H_{ln}:=\frac{1}{n}\sumj|\dd_{j}X-\hat{\gamma}_{n}h_{n}|^{pl},\quad l=1,2.
\nonumber
\end{equation}
The factor ``$h_{n}^{-1/\beta}$'' in (\ref{hm:sslp_lm2}) can be effectively cancelled out 
in the first equation in \eqref{hm:sslp_lm3-4}. 
We can see that, for each $p\in(0,\beta/6)$, the right-hand side of the first one in (\ref{hm:sslp_lm3-4}) 
is a constant multiple of the map
\begin{equation}
\beta\mapsto\frac{\Gamma(1-p/\beta)^{2}}{\Gamma(1-2p/\beta)}.
\nonumber
\end{equation}
Since this map is strictly increasing in $\beta\in(6p,2)$, 
it is straightforward to find the root $\hat{\beta}_{p,n}$ by a standard numerical procedure.

\medskip

Let $\eta_{p}(\beta):=\psi(1-p/\beta)-\psi(1-2p/\beta)$; 
recall that $\psi(z):=\p_{z}\log\Gamma(z)$ denotes the digamma function. 
Then, the asymptotic covariance matrix $V^{p}(\theta)=[V^{p}_{kl}(\theta)]_{k,l=1}^{3}$ 
in the present case can be explicitly computed as follows:

\begin{align}
V_{11}^{p}(\theta)&=\frac{\beta^{4}}{p^{2}\eta_{p}(\beta)^{2}}\bigg\{
\frac{C(\beta,2p)}{C(\beta,p)^{2}}-\frac{C(\beta,3p)}{C(\beta,p)C(\beta,2p)}
+\frac{1}{4}\bigg(\frac{C(\beta,4p)}{C(\beta,2p)^{2}}-1\bigg)\bigg\},
\nn\\
V_{12}^{p}(\theta)&=\frac{\beta^{2}\sig}{p^{2}\eta_{p}(\beta)^{2}}\bigg\{
\psi\bigg(1-\frac{2p}{\beta}\bigg)
\bigg(\frac{C(\beta,3p)}{2C(\beta,p)C(\beta,2p)}-\frac{C(\beta,2p)}{C(\beta,p)^{2}}+\frac{1}{2}\bigg)
\nn\\
& {}\qquad\qquad
+\psi\bigg(1-\frac{p}{\beta}\bigg)
\bigg(\frac{C(\beta,3p)}{2C(\beta,p)C(\beta,2p)}
-\frac{C(\beta,4p)}{4C(\beta,2p)^{2}}-\frac{1}{4}\bigg)\bigg\},
\nn\\
V_{22}^{p}(\theta)&=\frac{\sig^{2}}{p^{2}\eta_{p}(\beta)^{2}}\bigg[\bigg\{
\psi\bigg(1-\frac{2p}{\beta}\bigg)\bigg\}^{2}\bigg(\frac{C(\beta,2p)}{C(\beta,p)^{2}}-1\bigg)
\nn\\
& {}\qquad\qquad
-\psi\bigg(1-\frac{p}{\beta}\bigg)\psi\bigg(1-\frac{2p}{\beta}\bigg)
\bigg(\frac{C(\beta,3p)}{C(\beta,p)C(\beta,2p)}-1\bigg)
\nn\\
& {}\qquad\qquad
+\frac{1}{4}\bigg\{\psi\bigg(1-\frac{p}{\beta}\bigg)\bigg\}^{2}
\bigg(\frac{C(\beta,4p)}{C(\beta,2p)^{2}}-1\bigg)\bigg],
\nn\\
V_{13}^{p}(\theta)&=V_{23}^{p}(\theta)=0, \nn\\
V_{33}^{p}(\theta)&=\left(\frac{\sig\pi}{2\Gam(1+1/\beta)}\right)^{2}.
\nn
\end{align}
We can prove that $V^{p}(\theta)$ is positive definite for any admissible $\theta\in\Theta$, 
for the details of which we refer to \cite[Section 3.2]{Mas09_slp}.

\begin{thm} Fix any $\theta\in\Theta$ and $p\in(0,\beta/6)$ and define
\begin{equation}
\hat{\theta}_{p,n}:=(\hat{\beta}_{p,n},\hat{\sig}_{p,n},\hat{\gamma}_{n}),
\label{hm:sslp_lfME}
\end{equation}
where $(\hat{\beta}_{p,n},\hat{\sig}_{p,n})$ is a solution of (\ref{hm:sslp_lm3-4}). 
Then,
\begin{equation}
V^{p}(\hat{\theta}_{p,n})^{-1/2}\diag\big(\sqrt{n},\sqrt{n},\sqrt{n}h_{n}^{1-1/\hat{\beta}_{p,n}}\big)
(\hat{\theta}_{p,n}-\theta)\cil N_{3}(0,I_{3}),
\label{hm:sslp_wclfME}
\end{equation}
where $V^{p}(\theta)$ is positive-definite.
\label{hm:sslp_thm1}
\end{thm}

%%%

\subsubsection{Simulation experiments}\label{hm:sslp_comparison_sim}

In this section, we will first make some comparisons between the asymptotic covariances 
$V^{\log}(\theta)$ and $V^{p}(\theta)$ given in Sections \ref{hm:sslp_logMEsec} and \ref{hm:sec_sslp_lfm}, respectively, 
and then observe finite-sample performance of our estimator through simulation experiments.

%\medskip

\runinhead{Comparing asymptotic variances.}

For conciseness, we will focus on comparisons between ``$\hat{\beta}_{\log,n}$ and $\hat{\beta}_{p,n}$'' 
and ``$\hat{\sig}_{\log,n}$ and $\hat{\sig}_{p,n}$'' individually. 
The function $V^{\log}(\theta)$ has a simple structure, while 
the dependence structure of $(\beta,p)$ on $V^{p}(\theta)$ is somewhat more messy. 
According to the construction of $\hat{\theta}_{p,n}$, 
the value $p\mapsto V^{p}(\theta)$ for a given $p$ diverges as $\beta$ decreases. 
Figure \ref{hm:sslp_figV11V22} shows plots of $\beta\mapsto V_{11}^{\log}(\theta),V_{11}^{p}(\theta)$ and 
$\beta\mapsto V_{22}^{\log}(\theta),V_{22}^{p}(\theta)$ on $(0,2)$ for $p=0.05,0.1$ and $0.2$; 
we refer to \cite{Tod13} for plots of the asymptotic variances 
in estimation of the stable-index and the integrated-scale parameters 
in a general class of pure-jump It\^o-process models.

\begin{figure}[htb!]
\centering
\includegraphics[width=10cm]{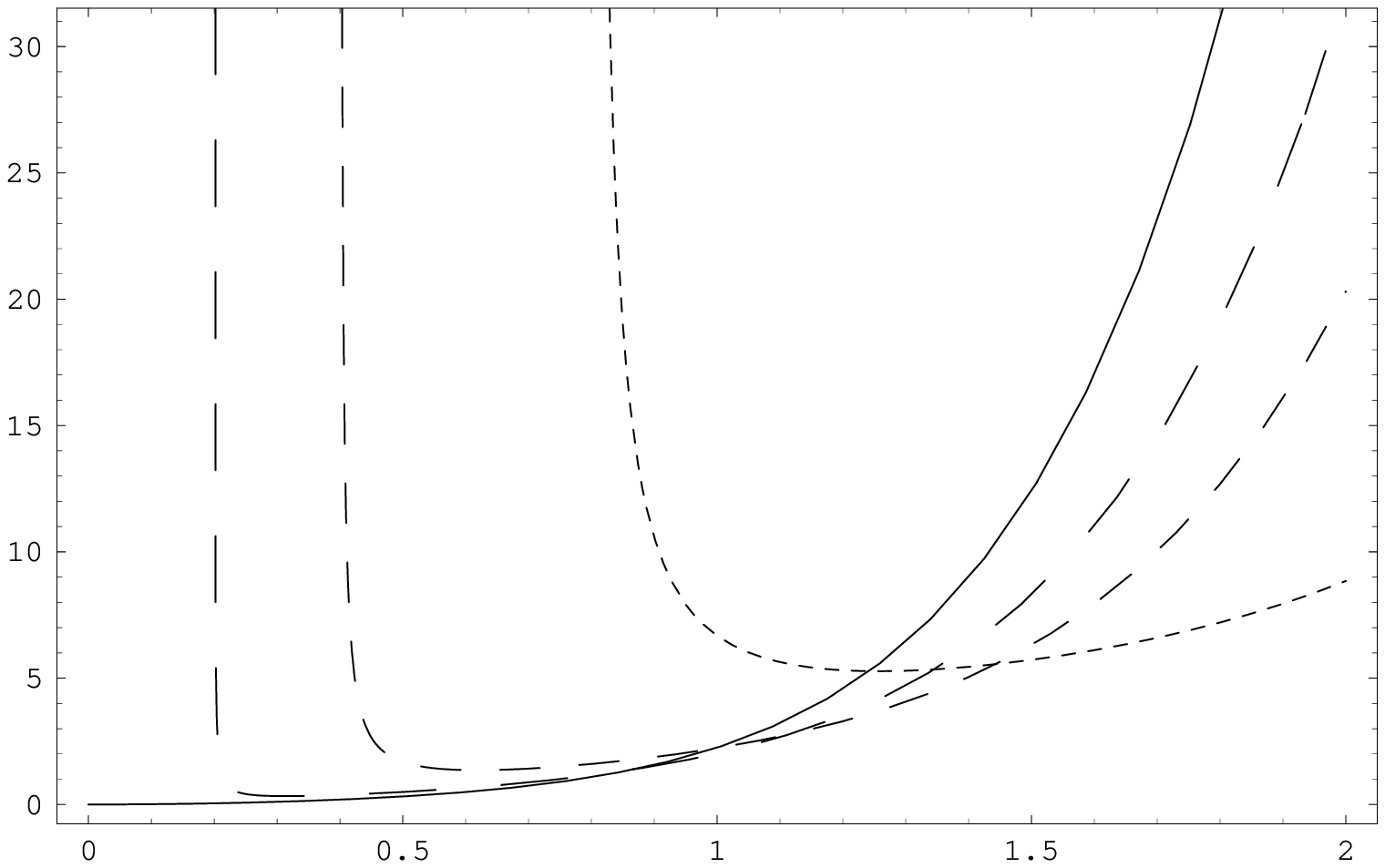}%\includegraphics[width=10cm]{gam.ps}
\\$\ $ \\%\hspace{1cm}
\includegraphics[width=10cm]{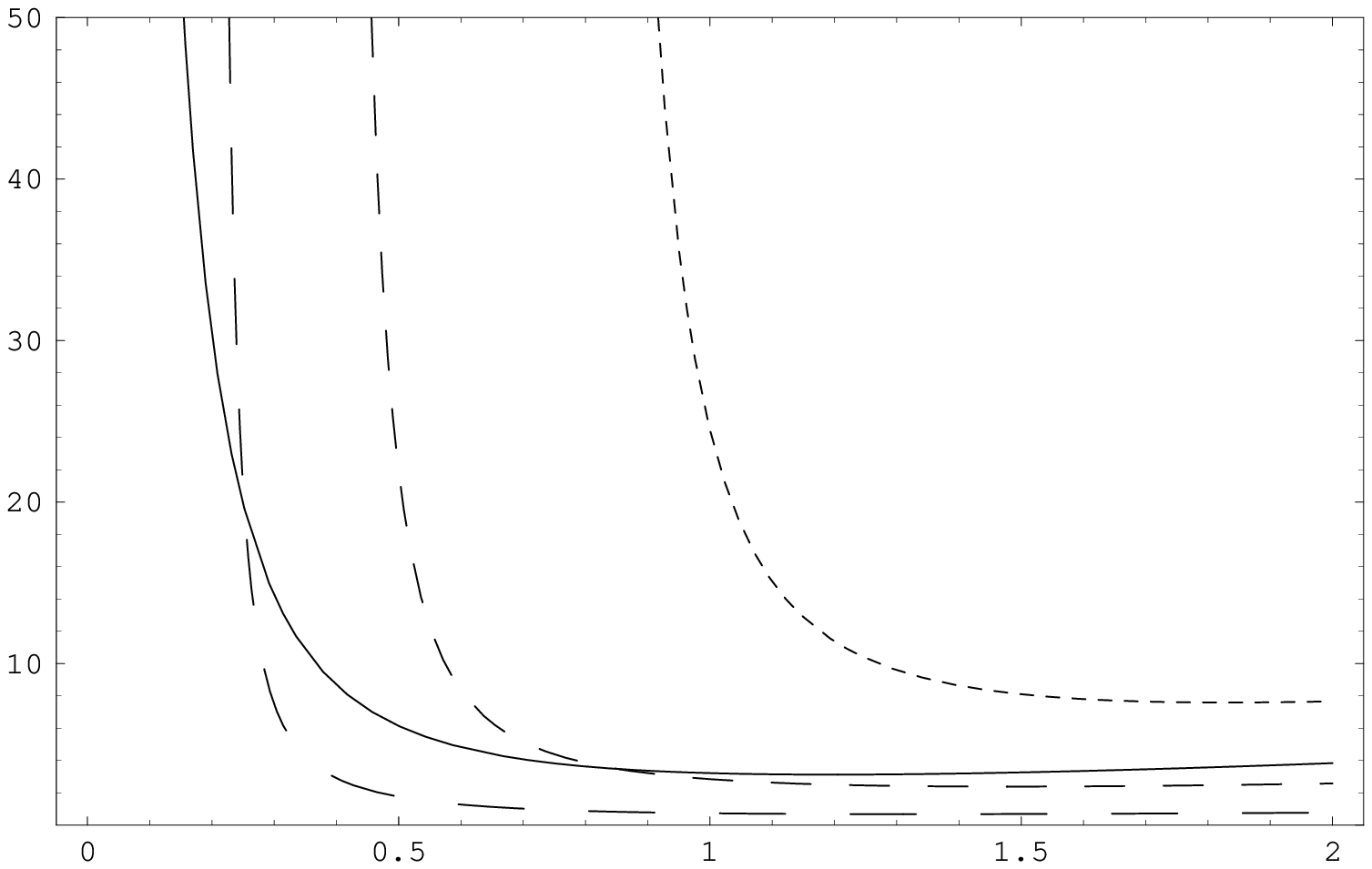}%\includegraphics[width=10cm]{gam.ps}
\caption{
Plots of asymptotic variances of the estimators of $\beta$ (top) and $\sig$ (bottom), 
each panel containing the logarithmic moment based one (solid line), 
the three lower-order moment based ones with $p=0.05$ (the sparsest dashed line), 
$p=0.1$ (moderate dashed line), and $p=0.2$ (the finest dashed line). 
Upper panel: $\beta\mapsto V_{11}^{\log}(\beta,1)$ and $V_{11}^{p}(\beta,1)$. 
Lower panel: $\beta\mapsto V_{22}^{\log}(\beta,1)$ and $V_{22}^{p}(\beta,1)$.
}
\label{hm:sslp_figV11V22}
\end{figure}

From Figure \ref{hm:sslp_figV11V22}, we can observe the following.

\begin{itemize}
\item Concerning $(\hat{\beta}_{\log,n},\hat{\sig}_{\log,n})$: 
the asymptotic performance of $\hat{\beta}_{\log,n}$ monotonically changes with $\beta$ over $(0,2]$, 
better for smaller $\beta$ and worse for larger $\beta$. 
Further, the asymptotic performance of $\hat{\sig}_{\log,n}$ gets worse for smaller $\beta$; 
more precisely, the function $\beta\mapsto V_{22}^{\log}(\beta,1)$ on $(0,2)$ takes a unique minimum 
around $1.2$, increases to a finite value as $\beta\uparrow 2$ and to infinity as $\beta\downarrow 0$.

\item Concerning $(\hat{\beta}_{p,n},\hat{\sig}_{p,n})$: 
we expect that smaller (resp. larger) $p$ leads to smaller asymptotic variance of $\hat{\beta}_{p,n}$ 
for smaller (resp. larger) $\beta$, the thresholds lying in the region $(1,1.5)$. 
In contrast, given any $\beta$, smaller $p$ leads to better performance of $\hat{\sig}_{p,n}$.

\end{itemize}

For implementation we have to fix the value $p$ {\it a priori} when applying $\hat{\theta}_{p,n}$, 
the permissible zone of which depends on the unknown $\beta$. 
We will briefly discuss this point based on simulation results. 
As a matter of fact, it will turn out that selection of $p$ actually has non-negligible influence on 
the behavior of $(\hat{\beta}_{p,n},\hat{\sig}_{p,n})$.

We have remarked that given a $\beta$, the asymptotic behavior of $\hat{\sig}_{p,n}$ should be better for smaller $p$. 
This is, however, only based on the expressions of the asymptotic variance. 
Finite-sample performance of the estimator $\hat{\sig}_{p,n}$ must depend on that of $\hat{\beta}_{p,n}$, 
so that it may occur and indeed we will see shortly that, e.g., 
$\hat{\sig}_{0.2,n}$ behaves better than $\hat{\sig}_{0.1,n}$ 
for $\hat{\beta}_{0.2,n}$ behaving better than $\hat{\beta}_{0.1,n}$.

%%%%%

\medskip

\runinhead{Setting and results.}

We will observe different finite-sample behaviors according to the true value of $\beta$. 
In each simulation below, we generate $1000$ independent estimates 
of the parameter, and tabulate corresponding sample means and sample root mean-square errors (RMSEs).

We take $p=0.05,0.1$ and $0.2$ for $\hat{\theta}_{p,n}$. In each trial 
except for the cases where $p\ge \beta/6$, we tabulate the estimates
\begin{equation}
(\hat{\beta}_{\log,n},~\hat{\beta}_{0.05,n},~\hat{\beta}_{0.1,n},~\hat{\beta}_{0.2,n};
~\hat{\sig}_{\log,n},~\hat{\sig}_{0.05,n},~\hat{\sig}_{0.1,n},~\hat{\sig}_{0.2,n};~\hat{\gamma}_{n}),
\label{hm:sslp_sim1-1}
\end{equation}
all of which are computed from a single realization of $(X_{t^{n}_{j}})_{j=1}^{n}$. 
We set $\beta=0.8,1.0,1.5$ and $1.8$ as well as $(\sig,\gamma)=(0.5,-0.5)$ for the true values 
and also $h_{n}=5/n$ ($T_{n}\equiv 5$) and $h_{n}=n^{-3/5}$ ($T_{n}=n^{2/5}\to\infty$) for the sampling schemes.

\medskip

Tables \ref{hm:sslp_table1} and \ref{hm:sslp_table2} report the means and the RMSEs of the 
estimates (\ref{hm:sslp_sim1-1}) with different sample sizes $n=501,1001$, and $2001$, 
and the different true value of $\beta$. Just for reference and the sake of comparison, 
each numerical result includes the rate-efficient $\tilde{\beta}_{n}(\sig)$, 
for which the scale $\sig$ is assumed to be known (see Remark \ref{hm:rem_rate-eff-beta}); 
as was expected, $\tilde{\beta}_{n}(\sig)$ surpasses by far all the other estimators of $\beta$. 
In both of Tables \ref{hm:sslp_table1} and \ref{hm:sslp_table2}, the best estimates for $n=2001$ are 
$(\hat{\beta}_{\log,n},\hat{\sig}_{\log,n})$, $(\hat{\beta}_{0.05,n},\hat{\sig}_{0.05,n})$, 
$(\hat{\beta}_{0.2,n},\hat{\sig}_{0.2,n})$, and $(\hat{\beta}_{0.2,n},\hat{\sig}_{0.2,n})$, for 
$\beta=0.8,1.0,1.5$ and $1.8$, respectively (the bold-letter elements). 
The performances for estimating $\beta$ seem to bear no relation to sampling frequency, while 
larger $T_{n}$ may lead to better finite-sample performance in estimation of $\sig$. 
Both tables show that 
finite-sample performance of joint estimation of $(\beta,\sig)$ can exhibit a different feature 
from the individual comparison through the asymptotic variances. 
For instance, Figure \ref{hm:sslp_figV11V22} says that $\hat{\sig}_{0.05,n}$ 
individually behaves best for $\beta\ge 1.5$, while $\hat{\sig}_{0.2,n}$ is actually the best one in 
Tables \ref{hm:sslp_table1} and \ref{hm:sslp_table2}. 
This would be due to the better behaviors of $\hat{\beta}_{0.2,n}$ for $\beta\ge 1.5$.

\medskip

\begin{table}
\begin{small}
\begin{center}
\begin{tabular}{rrrrrrrrr}
\hline%\hline
&&&&&&&& \\
\multicolumn{9}{c}{Case of $h_{n}=5/n$} \\
&&&&&&&& \\
\cline{2-8}\\[-2mm]
& True $\beta$ & $n$ & 
\multicolumn{1}{c}{$\hat{\beta}_{\log,n}$} & \multicolumn{1}{c}{$\hat{\beta}_{0.05,n}$} & 
\multicolumn{1}{c}{$\hat{\beta}_{0.1,n}$} & \multicolumn{1}{c}{$\hat{\beta}_{0.2,n}$} & 
\multicolumn{1}{c}{$\tilde{\beta}_{n}(\sig)$} & \\
\cline{2-8}
&&&&&&&& \\%\cline{2-9}
 & 0.8 & 501  & 0.807 (0.049) & 0.807 (0.050) & 0.809 (0.056) & & 0.800 (0.013) & \\
 &     & 1001 & 0.803 (0.034) & 0.803 (0.034) & 0.804 (0.039) & & 0.800 (0.008) & \\
 &     & 2001 & {\bf 0.800 (0.024)} & 0.801 (0.024) & 0.801 (0.028) & & 0.800 (0.005) & \\ % to boxplot
&&&&&&&& \\%\cline{2-9}
 & 1.0 & 501  & 1.010 (0.070) & 1.009 (0.066) & 1.009 (0.067) & \phantom{1.017 (0.081)} & 1.001 (0.018) & \\
 &     & 1001 & 1.003 (0.048) & 1.003 (0.045) & 1.003 (0.046) & \phantom{1.008 (0.060)} & 1.000 (0.011) & \\
 &     & 2001 & 1.003 (0.033) & {\bf 1.003 (0.033)} & 1.003 (0.033) & \phantom{1.007 (0.044)} & 1.000 (0.007) & \\ % to boxplot
&&&&&&&& \\%\cline{2-9}
 & 1.5 & 501  & 1.526 (0.162) & 1.518 (0.130) & 1.514 (0.112) & 1.514 (0.100) & 1.500 (0.031) & \\
 &     & 1001 & 1.516 (0.115) & 1.511 (0.093) & 1.508 (0.080) & 1.507 (0.073) & 1.500 (0.018) & \\
 &     & 2001 & 1.505 (0.081) & 1.504 (0.066) & 1.504 (0.058) & {\bf 1.504 (0.053)} & 1.500 (0.011) & \\ % to boxplot
&&&&&&&& \\%\cline{2-9}
 & 1.8 & 501  & 1.857 (0.288) & 1.804 (0.151) & 1.807 (0.133) & 1.809 (0.109) & 1.801 (0.042) & \\
 &     & 1001 & 1.824 (0.189) & 1.804 (0.125) & 1.805 (1.108) & 1.805 (0.085) & 1.799 (0.024) & \\
 &     & 2001 & 1.815 (0.133) & 1.807 (0.095) & 1.807 (0.081) & {\bf 1.805 (0.062)} & 1.800 (0.016) & \\ % to boxplot
&&&&&&&& \\%\cline{2-9}
\cline{2-8}\\[-2mm]
& True $\beta$ & $n$ & 
\multicolumn{1}{c}{$\hat{\sig}_{\log,n}$} & \multicolumn{1}{c}{$\hat{\sig}_{0.05,n}$} & 
\multicolumn{1}{c}{$\hat{\sig}_{0.1,n}$} & \multicolumn{1}{c}{$\hat{\sig}_{0.2,n}$} & 
\multicolumn{1}{c}{$\hat{\gamma}_{n}$} & \\
\cline{2-8}
&&&&&&&& \\%\cline{2-9}
 & 0.8 & 501  & 0.518 (0.178) & 0.521 (0.200) & 0.531 (0.306) & & -0.500 (0.050) & \\
 &     & 1001 & 0.511 (0.139) & 0.513 (0.147) & 0.516 (0.181) & & -0.500 (0.006) & \\
 &     & 2001 & {\bf 0.513 (0.109)} & 0.513 (0.115) & 0.517 (0.139) & & -0.500 (0.003) & \\ % to boxplot
&&&&&&&& \\%\cline{2-9}
 & 1.0 & 501  & 0.511 (0.152) & 0.510 (0.147) & 0.511 (0.155) & \phantom{0.513 (0.214)} & -0.497 (0.036) & \\
 &     & 1001 & 0.511 (0.120) & 0.511 (0.116) & 0.512 (0.123) & \phantom{0.506 (0.179)} & -0.502 (0.025) & \\
 &     & 2001 & 0.504 (0.093) & {\bf 0.503 (0.089)} & 0.503 (0.095) & \phantom{0.505 (0.143)} & -0.500 (0.018) & \\ % to boxplot
&&&&&&&& \\%\cline{2-9}
 & 1.5 & 501  & 0.508 (0.134) & 0.504 (0.107) & 0.502 (0.094) & 0.501 (0.098) & -0.501 (0.180) & \\
 &     & 1001 & 0.504 (0.111) & 0.503 (0.090) & 0.502 (0.080) & 0.502 (0.080) & -0.503 (0.093) & \\
 &     & 2001 & 0.508 (0.095) & 0.504 (0.076) & 0.503 (0.067) & {\bf 0.501 (0.067)} & -0.504 (0.014) & \\ % to boxplot
&&&&&&&& \\%\cline{2-9}
 & 1.8 & 501  & 0.508 (0.144) & 0.513 (0.095) & 0.507 (0.077) & 0.502 (0.062) & -0.484 (0.306) & \\
 &     & 1001 & 0.509 (0.125) & 0.509 (0.089) & 0.505 (0.073) & 0.502 (0.059) & -0.510 (0.125) & \\
 &     & 2001 & 0.505 (0.100) & 0.503 (0.075) & 0.501 (0.062) & {\bf 0.500 (0.048)} & -0.501 (0.284) & \\ % to boxplot
&&&&&&&& \\%\cline{2-9}
\hline%\hline
\end{tabular}
\end{center}
\end{small}
\caption{Sample means with RMSEs in parentheses of the simultaneously computed nine estimates (\ref{hm:sslp_sim1-1}) 
and $\tilde{\beta}_{n}(\sig)$ in case of $h_{n}=5/n$ ($T_{n}\equiv 5$), 
based on $1000$ independent copies of $(X_{t^{n}_{j}})_{j=1}^{n}$, 
where $\sig=0.5$ and $\gamma=-0.5$ for the true values. The cases where $p\ge\beta/6$ are left in blank.
}
\label{hm:sslp_table1}
\end{table}

\begin{table}
\begin{small}
\begin{center}
\begin{tabular}{rrrrrrrrrr}
\hline%\hline
&&&&&&&&& \\
\multicolumn{10}{c}{Case of $h_{n}=n^{-3/5}$} \\
&&&&&&&&& \\
\cline{2-9}\\[-2mm]
& True $\beta$ & $n$ & $T_{n}$ &
\multicolumn{1}{c}{$\hat{\beta}_{\log,n}$} & \multicolumn{1}{c}{$\hat{\beta}_{0.05,n}$} & 
\multicolumn{1}{c}{$\hat{\beta}_{0.1,n}$} & \multicolumn{1}{c}{$\hat{\beta}_{0.2,n}$} & 
\multicolumn{1}{c}{$\tilde{\beta}_{n}(\sig)$} & \\
\cline{2-9}
&&&&&&&&& \\%\cline{2-9}
 & 0.8 & 501  & 12.021 & 0.806 (0.047) & 0.806 (0.048) & 0.807 (0.054) & & 0.801 (0.017) & \\
 &     & 1001 & 15.855 & 0.802 (0.033) & 0.802 (0.034) & 0.804 (0.038) & & 0.800 (0.011) & \\
 &     & 2001 & 20.917 & {\bf 0.802 (0.024)} & 0.802 (0.024) & 0.803 (0.027) & & 0.800 (0.006) & \\ % to boxplot
&&&&&&&&& \\%\cline{2-9}
 & 1.0 & 501  & 12.021 & 1.012 (0.071) & 1.011 (0.067) & 1.012 (0.069) & \phantom{1.020 (0.084)} & 1.001 (0.022) & \\
 &     & 1001 & 15.855 & 1.005 (0.048) & 1.005 (0.045) & 1.005 (0.046) & \phantom{1.008 (0.058)} & 1.002 (0.014) & \\
 &     & 2001 & 20.917 & 1.003 (0.033) & {\bf 1.003 (0.031)} & 1.003 (0.033) & \phantom{1.007 (0.045)} & 1.000 (0.009) & \\ % to boxplot
&&&&&&&&& \\%\cline{2-9}
 & 1.5 & 501  & 12.021 & 1.529 (0.171) & 1.520 (0.135) & 1.516 (0.115) & 1.514 (0.099) & 1.502 (0.041) & \\
 &     & 1001 & 15.855 & 1.508 (0.111) & 1.505 (0.090) & 1.504 (0.078) & 1.505 (0.071) & 1.500 (0.025) & \\
 &     & 2001 & 20.917 & 1.508 (0.085) & 1.506 (0.069) & 1.505 (0.059) & {\bf 1.504 (0.053)} & 1.501 (0.016) & \\ % to boxplot
&&&&&&&&& \\%\cline{2-9}
 & 1.8 & 501  & 12.021 & 1.878 (0.308) & 1.812 (0.158) & 1.813 (0.139) & 1.813 (0.114) & 1.803 (0.053) & \\
 &     & 1001 & 15.855 & 1.824 (0.179) & 1.807 (0.122) & 1.807 (0.104) & 1.805 (0.080) & 1.801 (0.033) & \\
 &     & 2001 & 20.917 & 1.811 (0.130) & 1.805 (0.096) & 1.804 (0.080) & {\bf 1.801 (0.062)} & 1.800 (0.020) & \\ % to boxplot
&&&&&&&&& \\%\cline{2-9}
\cline{2-9}\\[-2mm]
& True $\beta$ & $n$ & $T_{n}$ & 
\multicolumn{1}{c}{$\hat{\sig}_{\log,n}$} & \multicolumn{1}{c}{$\hat{\sig}_{0.05,n}$} & 
\multicolumn{1}{c}{$\hat{\sig}_{0.1,n}$} & \multicolumn{1}{c}{$\hat{\sig}_{0.2,n}$} & 
\multicolumn{1}{c}{$\hat{\gamma}_{n}$} & \\
\cline{2-9}
&&&&&&&&& \\%\cline{2-9}
 & 0.8 & 501  & 12.021 & 0.509 (0.131) & 0.512 (0.144) & 0.519 (0.196) & & -0.500 (0.012) & \\
 &     & 1001 & 15.855 & 0.509 (0.106) & 0.509 (0.112) & 0.510 (0.131) & & -0.500 (0.008) & \\
 &     & 2001 & 20.917 & {\bf 0.504 (0.081)} & 0.504 (0.084) & 0.504 (0.098) & & -0.500 (0.005) & \\ % to boxplot
&&&&&&&&& \\%\cline{2-9}
 & 1.0 & 501  & 12.021 & 0.503 (0.119) & 0.503 (0.125) & 0.509 (0.272) & \phantom{0.548 (1.546)} & -0.499 (0.036) & \\
 &     & 1001 & 15.855 & 0.504 (0.089) & 0.505 (0.086) & 0.505 (0.091) & \phantom{0.506 (0.120)} & -0.500 (0.025) & \\
 &     & 2001 & 20.917 & 0.501 (0.067) & {\bf 0.501 (0.065)} & 0.501 (0.072) & \phantom{0.504 (0.191)} & -0.499 (0.018) & \\ % to boxplot
&&&&&&&&& \\%\cline{2-9}
 & 1.5 & 501  & 12.021 & 0.503 (0.102) & 0.501 (0.083) & 0.501 (0.074) & 0.500 (0.072) & -0.498 (0.137) & \\
 &     & 1001 & 15.855 & 0.507 (0.084) & 0.504 (0.068) & 0.503 (0.060) & 0.502 (0.061) & -0.500 (0.110) & \\
 &     & 2001 & 20.917 & 0.502 (0.069) & 0.501 (0.056) & 0.501 (0.049) & {\bf 0.501 (0.047)} & -0.505 (0.084) & \\ % to boxplot
&&&&&&&&& \\%\cline{2-9}
 & 1.8 & 501  & 12.021 & 0.497 (0.115) & 0.507 (0.075) & 0.503 (0.062) & 0.500 (0.051) & -0.501 (0.199) & \\
 &     & 1001 & 15.855 & 0.053 (0.087) & 0.504 (0.062) & 0.502 (0.051) & 0.501 (0.040) & -0.493 (0.171) & \\
 &     & 2001 & 20.917 & 0.054 (0.073) & 0.503 (0.055) & 0.502 (0.045) & {\bf 0.502 (0.035)} & -0.505 (0.153) & \\ % to boxplot
&&&&&&&&& \\%\cline{2-9}
\hline%\hline
\end{tabular}
\end{center}
\end{small}
\caption{Sample means with RMSEs in parentheses of the 
simultaneously computed  nine estimates (\ref{hm:sslp_sim1-1}) and $\tilde{\beta}_{n}(\sig)$ 
in case of $h_{n}=n^{-3/5}$ ($T_{n}=n^{2/5}\to\infty$), 
based on $1000$ independent copies of $(X_{t^{n}_{j}})_{j=1}^{n}$, 
where $\sig=0.5$ and $\gamma=-0.5$ for the true values. The cases where $p\ge\beta/6$ are left in blank.
}
\label{hm:sslp_table2}
\end{table}

Though omitted here, we could also observe that the logarithmic transform of the estimators of $\sig$ 
mentioned in Remark \ref{hm:sslp_rem_dm+1} could gain accuracy of the normal approximations for $\beta\le 1$ 
in finite-sample. Further, we could observe reasonably accurate normal approximation 
of $\sqrt{n}h_{n}^{1-1/\hat{\beta}_{n}}(\hat{\gamma}_{n}-\gamma)$ upon a suitable choice of 
$\hat{\beta}_{n}$ within our estimators.

%%%

\runinhead{Some practical remarks.}
In practice, we may roughly proceed as follows: 
first we apply $\hat{\theta}_{\log,n}$ which has no fine-tuning parameter. 
Then, building on the estimated values $(\hat{\beta}_{\log,n},\hat{\sig}_{\log,n})$ and taking 
the interrelationship of $V^{\log}(\theta)$ and $V^{p}(\theta)$, 
we apply $(\hat{\beta}_{p,n},\hat{\sig}_{p,n})$ anew with a suitable choice of $p$, 
or keep using $(\hat{\beta}_{\log,n},\hat{\sig}_{\log,n})$ if the estimate of $\beta$ is small. 
In many applications in practice, the case of $\beta\in(1,2)$, 
i.e., finite-mean case, may be relevant. Then, we may simply adopt $(\hat{\beta}_{p,n},\hat{\sig}_{p,n})$ 
from the beginning with a small $p$ such as $p=0.1$, 
and then adaptively change $p$ according to the estimated value of $\beta$ 
(e.g., pick $p=0.2$ if the first estimate of $\beta$ is greater than $1.5$).

As a whole, we may conclude that:
\begin{itemize}
\item $\hat{\theta}_{\log,n}$ is recommended for $\beta\le 1$;
\item $\hat{\theta}_{p,n}$ with small $p$ such as $0.05\sim 0.2$ and up is recommended for $\beta>1$.
\end{itemize}
As was expected from Theorem \ref{hm:sslp_thm0} (also Figure \ref{hm:sslp_figV11V22}), 
we could observe that $\hat{\beta}_{\log,n}$ becomes more unstable for $\beta$ closer to $2$; 
several times, it returns a value greater than $2$ for $\beta=1.8$ in our simulations.

%%%%%
%%%%%

\subsection{Skewed {\lm} with possibly time-varying scale}\label{hm:sec_skewslp_me}

In the previous section, we considered a joint estimation of the index, scale, and location parameters 
when the L\'evy density is symmetric. 
There we have seen that the sample median based estimator is rate-efficient. 
The primary objective of this section is to provide a practical moment estimator of 
a process $X$ of the form $X_{t}=\int_{0}^{t}\sig_{s-}dZ_{s}$ where 
$Z$ is a possibly skewed strictly stable {\lp} without drift 
and $\sig$ is a positive {\cadlag} process independent of $Z$. 
We will consider estimation of integrated scale when the scale parameter is time-varying. 
The topic of this section is based on \cite{Mas10}; a closely related work is \cite[Section 4]{Tod13}.

Our estimation procedure utilizes empirical-sign statistics and realized multipower variations 
(MPV for short; see Section \ref{hm:sec_skewlp_preliminaries}). 
Its implementation is quite simple and requires no hard numerical optimization, hence preferable in practice.
Using MPVs essentially amounts to the classical method of moments with possibly random targets. 
Several authors investigated asymptotic behaviors of MPVs for estimating integrated-scale quantities of pure-jump models. 
Among others, we refer to \cite[Section 6]{BarShe05}, \cite{CorNuaWoe07}, \cite{Tod13} with the references therein, 
and \cite{Woe03-3}; in all the papers, the underlying model is driven by either a stable or locally stable {\lp} 
(see Section \ref{hm:sec_lslp} for the definition of the latter). 
It will turn out that estimation of the integrated time-varying scale 
by substituting a $\sqrt{n}$-consistent estimator of $\beta$ into the MPV statistics 
will lead to the slower rate of convergence $\sqrt{n}/\log n$ (Section \ref{hm:sec_tvsp}).

%%%

\subsubsection{Setup and description of estimation procedure}\label{hm:sec_estimation}

To describe the model setup, we will adopt another parameterization of a strictly $\beta$-stable distribution: 
with a slight abuse of notations, we write $\mcl(S)=S'_{\beta}(\mfp,\sig)$ for $\beta\ne 1$ if
\begin{equation}
\vp_{S}(u)=\exp\bigg\{
-\sig|u|^{\beta}\bigg(1-i\mathrm{sgn}(u)\tan\{\beta\pi(\mfp-1/2)\}\bigg)\bigg\},\quad u\in\mbbr.
\label{hm:stable_cf}
\end{equation}
Instead of the skewness parameter $\rho$ we now have the {\it positivity parameter} $\mfp:=P(S>0)$, 
whose range of value is given as follows:
\begin{equation}
\mfp\in\left\{
\begin{array}{ll}
[1-1/\beta,1/\beta], & \quad\beta\in(1,2), \\[1mm]
(0,1), & \quad\beta=1, \\[1mm]
{[0,1]}, & \quad\beta\in(0,1).
\end{array}
\right.
\nn
\end{equation}
For $\beta\ne 1$, the parametrizations \eqref{hm:ex_st1} and \eqref{hm:stable_cf} are linked 
by the one-to-one relation
\begin{equation}
\mfp=\frac{1}{2}+\frac{1}{\beta\pi}\arctan\left(\rho\tan\frac{\beta\pi}{2}\right).
\label{hm:skew_relation}
\end{equation}
For any fixed $\beta\in(1,2)$, $\mfp$ is monotonically decreasing on $(-1,1)$ as a function of $\rho$. 
Hence $\mfp-1/2$ and $\rho$ have opposite signs for $\beta\in(1,2)$, while the same signs for $\beta\in(0,1)$; 
Figure \ref{hm:figure1} illustrates this point.
\begin{figure}[htbp]
\centering
\includegraphics[width=10cm%, trim= 0cm 0cm 1cm 1cm,clip%width=15cm,height=13.0cm
]{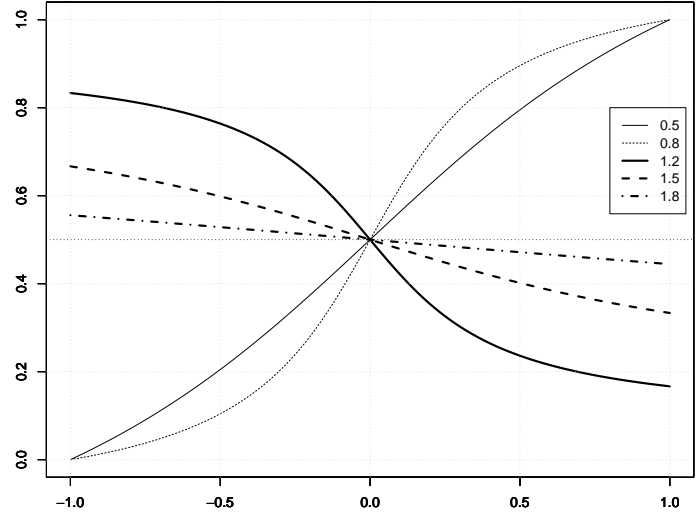}
\caption{
Plots of $\mfp$ as a function of $\rho$ for the values $\beta=1.2$, $1.5$, and $1.8$. 
Also included for comparison are the cases of $\beta=0.5$ and $0.8$.}
\label{hm:figure1}
\end{figure}
The primary reason why we have chosen the parametrization (\ref{hm:stable_cf}) is that, 
as is expected from Figure \ref{hm:figure1}, estimation performance of $\rho$ based on the empirical sign statistics, 
which we will make use of later, is destabilized for $\beta$ close to $2$: 
that is to say, the slope of the curve gets gentler for larger $\beta$, so that 
a small change of the empirical sign statistics results in a wide gap between the estimate of $\rho$ and the true value. 
Also, note the difference between the scale parameters of \eqref{hm:ex_st1} with $t=1$ and \eqref{hm:stable_cf}, 
which will turn out to be convenient for considering time-varying scale in a unified manner.

\medskip

Let $Z$ be a $\beta$-stable {\lp} such that
\begin{equation}
\mcl(Z_{t})=S'_{\beta}(\mfp,t),\quad t\in[0,1].
\label{hm:Z_def}
\end{equation}
Note that according to the scaling property, we have $\mfp=P(X_{t}>0)$ for each $t>0$. 
We will focus on the case where
\begin{equation}
\mfp\in(1-1/\beta,1/\beta),\quad \beta\in(1,2),
\label{hm:pb_region}
\end{equation}
so that jumps are not one-sided and are of infinite variation; 
nevertheless, it will be obvious from the subsequent discussion 
that our estimation procedure remains in force for $\beta\in(0,1)$ too. 
Figure \ref{hm:figure2} shows typical sample paths of $Z$. 
\begin{figure}[htbp]
\centering
\includegraphics[scale=0.5%, angle=-90%, trim= 0cm 0cm 1cm 1cm,clip%width=15cm,height=13.0cm
]{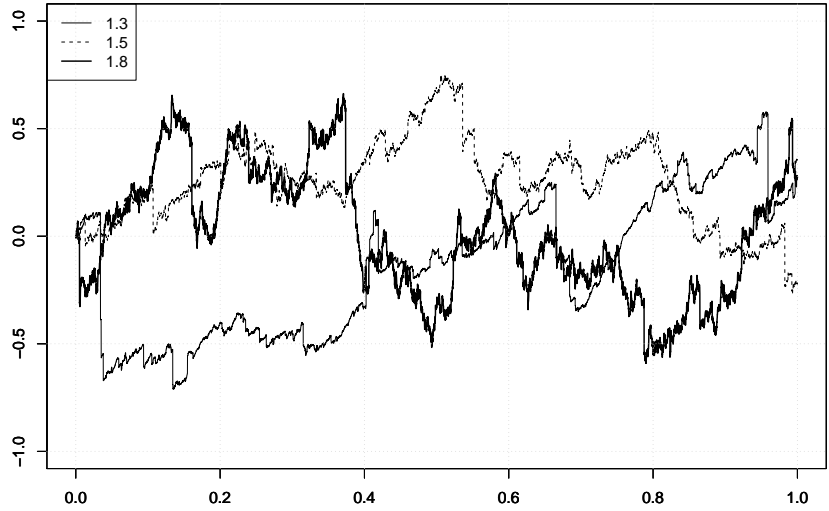}
\caption{
Plots (solid lines for clarity) of three typical sample paths of $Z$ of (\ref{hm:Z_def}) on $[0,1]$ 
for $\beta=1.3$, $1.5$, and $1.8$, with $\rho=-0.5$ and $\sig_{t}\equiv 1$ in common; 
large jumps are tend to be downward, while small fluctuations upward.
}
\label{hm:figure2}
\end{figure}

\medskip

We now accommodate a possibly time-varying scale process $\sig=(\sig_{t})_{t\in[0,1]}$, 
which is assumed to be {\cadlag} adapted and independent of $Z$, and also bounded away from zero and infinity. 
Let $X=(X_{t})_{t\in[0,1]}$ be the process given by
\begin{equation}
X_{t}=\int_{0}^{t}\sig_{s-}dZ_{s},
\label{hm:X_def}
\end{equation}
where the stochastic integral is well-defined since $P(\int_{0}^{1}\sig_{s}^{2}ds<\infty)=1$; 
see, e.g., \cite{JacShi03} and/or \cite{Pro05}.

\begin{rem}{\rm 
We may equivalently (in distribution) define $X$ of \eqref{hm:X_def} by the time-change representation 
with ``clock'' process $t\mapsto\int_{0}^{t}\sig_{s}^{\beta}ds$:
\begin{equation}
X_{t}=Z_{\int_{0}^{t}\sig_{s}^{\beta}ds}.
\nonumber
\end{equation}
Such kind of distributional equivalence 
can occur only for stable $Z$ among general {\lp es}: see \cite{KalShi01} for details. 
It is a matter of no importance that the target time period is $[0,1]$ from the very beginning: 
enlarging the length of the period is reflected 
in making $\int_{0}^{1}\sig_{s}^{\beta}ds$ larger through the process $\sig$.
\label{hm:rem_ssi_tc}
\qed}\end{rem}

\medskip

In the sequel, we fix a true value of $(\mfp,\beta,\sig_{\cdot})$. 
Note that the scaling property and the independence between $\sig$ and $Z$ give the $\sig$-conditional distribution
\begin{equation}
\mcl(X_{1}|\sig_{\cdot})=S'_{\beta}\bigg(\mfp,~\int_{0}^{1}\sig_{s}^{\beta}ds\bigg).
\nn%\label{hm:cond_law}
\end{equation}
Our objective is to estimate the following quantities under \eqref{hm:pb_region} from a sample $(X_{j/n})_{j=1}^{n}$:
\begin{description}
\item[(A)] $\theta=(\mfp,\beta,\sig)$ when $\sig_{t}\equiv \sig>0$ is constant;
\item[(B)] $\theta=(\mfp,\beta,\int_{0}^{1}\sig_{s}^{\beta}ds)$ 
when $\sig_{\cdot}$ is time-varying.
\end{description}
We will provide an explicit estimator of $\theta$ in each case, 
which is asymptotically (mixed) normal at rate $\sqrt{n}$. 
The case (A) is obviously included in the case (B), however, 
the case (B) will exhibit an essentially different feature from the case (A), 
requiring a separate argument. In both cases:
\begin{itemize}
\item We first construct a simple estimator of $(\mfp,\beta)$ with leaving $\sig_{\cdot}$ unknown;
\item Then, using the estimates of $(\mfp,\beta)$ we construct an estimator of $\sig_{\cdot}$ 
or $\int_{0}^{1}\sig_{s}^{\beta}ds$. 
\end{itemize}
It has been known that we can estimate the integrated-scale by means of MPV for pure-jump processes; 
see \cite{Tod13} and \cite{Woe03-3} together with the references therein. 
Thanks to the assumed independence between $\sig_{\cdot}$ and $Z$, 
our estimator of $(\mfp,\beta)$ can be computed without using information of $\sig$.

\begin{rem}{\rm 
For the model of the from \eqref{hm:X_def} with symmetric jumps and non-random $\sig$, 
\cite{ZhaWu09} studied the logarithmic-moment estimation of $\beta$ 
and the kernel based median-quantile estimation of $\sig$. 
Under the smoothness conditions on the sample path of $\sig$, 
they derived the asymptotic normality for $\beta$ at rate $\sqrt{n}$, 
and for $\sig_{\cdot}$ the Bahadur-Kiefer type representation, the point-wise asymptotic normality, 
and the maximal-deviation type distributional result.
\label{hm:rem_ZhaWu09}
}\qed\end{rem}

Conditional on the process $\sig_{\cdot}$, 
the random variables $\dd_{j}X$ are mutually independent and for each $n\in\mbbn$ and $j\le n$
\begin{equation}
\mcl(\dd_{j}X|\sig_{\cdot})=S'_{\beta}\bigg(\mfp,~\int_{(j-1)/n}^{j/n}\sig_{s}^{\beta}ds\bigg).
\nonumber%\label{hm:DX_law}
\end{equation}
Let us note the following two basic facts, which we will use several times without notice.

\begin{itemize}

\item Since we are concerned here with the weak property, we may proceed as if
\begin{equation}
\dd_{j}X=(\bar{\sig}_{j}/n)^{1/\beta}\zeta_{j}\quad\text{a.s.},
\label{hm:skewslp_expression}
\end{equation} 
where $\bar{\sig}_{j}:=n\int_{(j-1)/n}^{j/n}\sig_{s}^{\beta}ds$ and 
$(\zeta_{j})$ is a $S'_{\beta}(\mfp,1)$-i.i.d. sequence.

\item Let $\Lam_{n}$ be a sequence of essentially bounded functionals 
on the product space of the path spaces of $Z$ and $\sig$, 
and let $\lam_{n}(\sig):=\int\Lam_{n}(\sig,z)P^{Z}(dz)$, 
where $P^{\xi}$ denotes the image measure of a random element $\xi$. 
Assume that $\lam_{n}(\sig)\cil\lam_{0}(\sig)$ for some functional $\lam_{0}$ on the path space of $\sig$. 
By the independence between $Z$ and $\sig$, a disintegration argument gives 
$\lam_{n}(\sig)=E\{\Lam_{n}(\sig,Z)|\sig\}$ a.s., 
and moreover the boundedness of $\{\lam_{n}(\sig)\}_{n\in\mbbn}$ 
yields convergence of moments: $E\{\Lam_{n}(\sig,Z)\}=\int\lam_{n}(\sig)P^{\sig}(d\sig)\to
\int\lam_{0}(\sig)P^{\sig}(d\sig)$. That is to say, 
we may treat $\sig$ a non-random process in the process of deriving weak limit theorems. 
In particular, if some functionals $S_{n}(\sig',Z)$ for any fixed $\sig'$ are asymptotically centered normal 
with covariance matrix $V(\sig')$, then it automatically follows that 
the limit distribution of $S_{n}(\sig,Z)$ has the characteristic function 
$u\mapsto \int\exp\{-V(\sig)[u,u]/2\}P^{\sig}(d\sig)$ corresponding to the centered mixed normal distribution 
with random covariance matrix $V(\sig)$. 
\end{itemize}

For convenience, in the rest of this section 
we will use the symbol $N_{p}(\cdot,\cdot)$ also for the mixed-normal distributions.

%%%

\subsubsection{Preliminaries}\label{hm:sec_skewlp_preliminaries}

\runinhead{Lower-order fractional moments and logarithmic moments.}
The closed-form expressions of the $r$th absolute and $r'$th signed-absolute moments of $S'_{\beta}(\mfp,1)$ 
can be found in \cite{Kur01}: for any $r\in(-1,\beta)$ and $r'\in(-2,-1)\cup(-1,\beta)$,
\begin{align}
\mu_{r}
&:=\frac{\Gam(1-r/\beta)}{\Gam(1-r)}\frac{\cos(r\xi/\beta)}{\cos(r\pi/2)|\cos(\xi)|^{r/\beta}},
\label{hm:moment1} \\
\nu_{r'}
&:=\frac{\Gam(1-r'/\beta)}{\Gam(1-r')}\frac{\sin(r'\xi/\beta)}{\sin(r'\pi/2)|\cos(\xi)|^{r'/\beta}},
\label{hm:moment2}
\end{align}
where
\begin{equation}
\xi:=\beta\pi(\mfp-1/2).
\label{hm:xi_def}
\end{equation}
Therefore
\begin{equation}
E(|\zeta|^{r})=\sig^{r/\beta}\mu_{r},\quad 
E\{|\zeta|^{r'}\mathrm{sgn}(\zeta)\}=\sig^{r'/\beta}\nu_{r'}
\nonumber
\end{equation}
if $\mcl(\zeta)=S'_{\beta}(\mfp,\sig)$.

\runinhead{Empirical sign statistics.} 
To estimate $\mfp$, we make use of
\begin{equation}
\hat{\mfp}_{n}:=\frac{1}{2}(H_{n}+1),
\label{hm:rho_n}
\end{equation}
where
\begin{equation}
H_{n}:=\frac{1}{n}\sumj\sgn(\dd_{j}X).
\nonumber
\end{equation}
Then
\begin{equation}
\sqrt{n}(\hat{\mfp}_{n}-\mfp)=\sum_{i=1}^{n}\frac{1}{2\sqrt{n}}\{\mathrm{sgn}(\zeta_{i})-(2\mfp-1)\},
\label{hm:es1}
\end{equation}
from which we immediately deduce the asymptotic normality
\begin{equation}
\sqrt{n}(\mfp_{n}-\mfp)\cil N_{1}\left(0,\mfp(1-\mfp)\right).
\nonumber
\end{equation}
A nice feature is that the asymptotic variance of $\hat{\mfp}_{n}$ solely depends on $\mfp$, 
directly enabling us to provide a confidence interval of $\mfp$. 
It will be seen in Section \ref{hm:sec_skewslp_sim} that 
$\hat{\mfp}_{n}$ exhibits, despite of its simplicity, good finite-sample performance.

\runinhead{Stochastic expansion of MPV.}
Let $m\in\mbbn$ and pick a multi-index $r=(r_{1},\dots,r_{m})\subset\mbbrp^{m}$ such that
\begin{equation}
r_{+}:=\sum_{l=1}^{m}r_{l}>0,\qquad \max_{l\le m}r_{l}<\beta/2.
\label{hm:mpv_r_cond}
\end{equation}
Then the $r$th MPV is defined by
\begin{equation}
M_{n}(r)=\frac{1}{n}\sum_{j=1}^{n-m+1}\prod_{l=1}^{m}|n^{1/\beta}\dd_{j+l-1}X|^{r_{l}}.
\nn
\end{equation}
By the equivalent expression \eqref{hm:skewslp_expression}, we may write
\begin{equation}
M_{n}(r)=\frac{1}{n}\sum_{j=1}^{n-m+1}
\prod_{l=1}^{m}\bar{\sig}_{j+l-1}^{r_{l}/\beta}|\zeta_{j+l-1}|^{r_{l}}.
\nn%\label{hm:mpv1}
\end{equation}
Observe that
\begin{equation}
\sqrt{n}\big\{M_{n}(r)-\mu(r;\mfp,\beta)\sig^{\ast}_{r_{+}}\big\}
=\sum_{j=1}^{n-m+1}\frac{1}{\sqrt{n}}\chi'_{nj}(r)+R_{n}(r),
\nonumber
\end{equation}
where
\begin{align}
& \mu(r;\mfp,\beta):=\prod_{l=1}^{m}\mu_{r_{l}},\qquad
\sig^{\ast}_{q}:=\int_{0}^{1}\sig_{s}^{q}ds,\quad q>0, \nn\\
& \chi'_{nj}(r):=\bigg(\prod_{l=1}^{m}\bar{\sig}_{j+l-1}^{r_{l}/\beta}\bigg)
\bigg(\prod_{l=1}^{m}|\zeta_{j+l-1}|^{r_{l}}-\mu(r;\mfp,\beta)\bigg), \nn\\
& R_{n}(r):=\mu(r;\mfp,\beta)\bigg\{\sum_{j=1}^{n-m+1}\frac{1}{\sqrt{n}}\bigg(
\prod_{l=1}^{m}\bar{\sig}_{j+l-1}^{r_{l}/\beta}-\sig_{(j-1)/n}^{r_{+}}\bigg) \nonumber \\
& {}\qquad\qquad
+\sum_{j=1}^{n-m+1}\sqrt{n}\int_{(j-1)/n}^{j/n}(\sig_{(j-1)/n}^{r_{+}}-\sig_{s}^{r_{+}})ds\bigg\}
+O_{p}\bigg(\frac{1}{\sqrt{n}}\bigg).
\nonumber
\end{align}
Then, proceeding as in \cite{BGJPS06} or \cite{Woe07}, we can deduce that $R_{n}(r)\cip 0$; 
recall that we are assuming that $\sig$ is {\cadlag}. 
Further, straightforward but messy computations lead to 
\begin{equation}
\sum_{j=1}^{n-m+1}\frac{1}{\sqrt{n}}\chi'_{nj}(r)
=\sum_{j=m}^{n}\frac{1}{\sqrt{n}}\chi_{nj}(r)+o_{p}(1),
\nonumber
\end{equation}
where
\begin{equation}
\chi_{nj}(r):=\bigg(\prod_{l=1}^{m}\bar{\sig}_{j-m+l}^{r_{l}/\beta}\bigg)
\sum_{q=1}^{m}\left\{\bigg(\prod_{l=1}^{q-1}|\zeta_{j+l-q}|^{r_{l}}\bigg)
\bigg(\prod_{l=q+1}^{m}\mu_{r_{l}}\bigg)(|\zeta_{j}|^{r_{q}}-\mu_{r_{q}})\right\}.
\nonumber
\end{equation}
Thus we arrive at the stochastic expansion
\begin{equation}
\sqrt{n}\big\{M_{n}(r)-\mu(r;\mfp,\beta)\sig^{\ast}_{r_{+}}\big\}
=\sum_{j=m}^{n}\frac{1}{\sqrt{n}}\chi_{nj}(r)+o_{p}(1).
\label{hm:mpv2}
\end{equation}

\runinhead{A basic limit theorem.}
Let $r=(r_{l})_{l=1}^{m}$ be as in \eqref{hm:mpv_r_cond}, 
and also let $r'=(r'_{l})_{l=1}^{m}$ be another vector satisfying the same conditions. 
In what follows we set $r_{+}=r'_{+}=p$ for some $p>0$. 
We want to derive the limit distribution of the random vectors
\begin{equation}
S_{n}(r,r'):=\sqrt{n}\left(
\begin{array}{ccc}
H_{n} &-& (2\mfp-1)  \\
M_{n}(r) &-& \mu(r;\mfp,\beta)\sig^{\ast}_{p} \\
M_{n}(r') &-& \mu(r';\mfp,\beta)\sig^{\ast}_{p}
\end{array}
\right),
\nonumber%\label{hm:mpvlim1}
\end{equation}
which will serve as a basic tool for our purpose. 
From (\ref{hm:es1}) and (\ref{hm:mpv2}) we have
\begin{equation}
S_{n}(r,r')=\sum_{j=m}^{n}\frac{1}{\sqrt{n}}\left(
\begin{array}{c}
\mathrm{sgn}(\zeta_{j})-(2\mfp-1) \\
\chi_{nj}(r) \\
\chi_{nj}(r')
\end{array}
\right)+o_{p}(1)=:\sum_{j=m}^{n}\frac{1}{\sqrt{n}}\gam_{nj}+o_{p}(1).
\nonumber
\end{equation}
For the term $\sum_{j=m}^{n}n^{-1/2}\gam_{nj}$, 
we can apply the central limit theorem for martingale difference arrays (cf. \cite{Dvo77}), 
where the underlying filtration may be taken as $(\mcg_{n,j})_{j\le n}$ 
with $\mcg_{n,j}:=\sig(\zeta_{k}:k\le j)$; 
recall that we may now proceed as if the process $\sig_{\cdot}$ is non-random. 
The Lindeberg condition is easily verified 
under the condition $\max_{l\le m}(r_{l}\vee r'_{l})<\beta/2$. 
Concerning convergence of the quadratic characteristic, it is not difficult to prove that
\begin{equation}
\frac{1}{n}\sum_{j=m}^{n}E\big(\gam_{nj}^{\otimes 2}\big|\mcg_{n,j-1}\big)
\cip\Sig(\mfp,\beta,\sig_{\cdot}):=
\left(
\begin{array}{ccc}
4\mfp(1-\mfp) & A(r)\sig^{\ast}_{r_{+}} & A(r')\sig^{\ast}_{r'_{+}} \\
 & B(r,r)\sig^{\ast}_{2r_{+}} & B(r,r')\sig^{\ast}_{r_{+}+r_{+}'} \\
\text{{\rm sym.}} & & B(r',r')\sig^{\ast}_{2r_{+}'}	
\end{array}
\right),
\nonumber
\end{equation}
where
\begin{align}
A(r)&=\sum_{q=1}^{m}\bigg(\prod_{1\le l\le m, l\ne q}\mu_{r_{l}}\bigg)\{\nu_{r_{q}}-(2\mfp-1)\mu_{r_{q}}\},
\nonumber \\
B(r,r')&=\prod_{l=1}^{m}\mu_{r_{l}+r'_{l}}-(2m-1)\prod_{l=1}^{m}\mu_{r_{l}}\mu_{r'_{l}} \nonumber \\
&{}\qquad+\sum_{q=1}^{m-1}\bigg\{\bigg(\prod_{l=1}^{m-q}\mu_{r'_{l}}\bigg)
\bigg(\prod_{l=m-q+1}^{m}\mu_{r'_{l}+r_{l-m+q}}\bigg)\bigg(\prod_{l=q+1}^{m}\mu_{r_{l}}\bigg) \nonumber \\
&{}\qquad\qquad+\bigg(\prod_{l=1}^{m-q}\mu_{r_{l}}\bigg)
\bigg(\prod_{l=m-q+1}^{m}\mu_{r_{l}+r'_{l-m+q}}\bigg)\bigg(\prod_{l=q+1}^{m}\mu_{r'_{l}}\bigg)\bigg\},
\nonumber
\end{align}
with obvious analogues $A(r')$, $B(r,r)$, and $B(r',r')$. 
Thus the limit distribution of $S_{n}(r,r')$ is a normal variance mixture 
with conditional covariance matrix $\Sig(\mfp,\beta,\sig_{\cdot})$:
\begin{equation}
S_{n}(r,r')\cil N_{3}\big(0,\Sig(\mfp,\beta,\sig_{\cdot})\big).
\label{hm:mpvlim2}
\end{equation}
Note that $\Sig(\mfp,\beta,\sig_{\cdot})$ depends on the process $\sig_{\cdot}$ only through the integrated quantities 
$\sig^{\ast}_{r_{+}}$, $\sig^{\ast}_{r'_{+}}$, $\sig^{\ast}_{2r_{+}}$, $\sig^{\ast}_{2r'_{+}}$, 
and $\sig^{\ast}_{r_{+}+r'_{+}}$, for which, as will be mentioned later, 
we can readily provide consistent estimators by means of MPV.

\medskip

We write $(\hat{\mfp}_{n},\hat{\beta}_{p,n},\hat{\sig}^{\ast}_{p,n})$ for 
the solution to the estimating equation
\begin{equation}
\left(
\begin{array}{ccc}
H_{n} &-& (2\mfp-1)  \\
M_{n}(r) &-& \mu(r;\mfp,\beta)\sig^{\ast}_{p} \\
M_{n}(r') &-& \mu(r';\mfp,\beta)\sig^{\ast}_{p}
\end{array}
\right)=\left(
\begin{array}{c}
0 \\
0 \\
0
\end{array}
\right).
\label{hm:ra1}
\end{equation}
We introduce the function
\begin{equation}
F(\mfp,\beta,s):=\left(2\mfp-1,~\mu(r;\mfp,\beta)s,~\mu(r';\mfp,\beta)s\right).
\nonumber
\end{equation}
Since we are assuming that $\beta\in(1,2)$ and $\mfp\in(1-1/\beta,1/\beta)$, 
we have $\xi\in(-\pi/2,\pi/2)$, so that $\cos(\xi)>0$; recall the definition \eqref{hm:xi_def}. 
The quantities $\mu(r;\mfp,\beta)$ and 
$\mu(r';\mfp,\beta)$ are continuously differentiable with respect to $(\mfp,\beta)$, and
\begin{equation}
\nabla F(\mfp,\beta,s)=\left(
\begin{array}{ccc}
2 & 0 & 0 \\
s\p_{\mfp}\mu(r;\mfp,\beta) & s\p_{\beta}\mu(r;\mfp,\beta) & \mu(r;\mfp,\beta) \\
s\p_{\mfp}\mu(r';\mfp,\beta) & s\p_{\beta}\mu(r';\mfp,\beta) & \mu(r';\mfp,\beta)
\end{array}
\right)
\nonumber
\end{equation}
is non-singular for each $s>0$ if
\begin{equation}
\mu(r';\mfp,\beta)\p_{\mfp}\mu(r;\mfp,\beta)\ne\mu(r;\mfp,\beta)\p_{\beta}\mu(r;\mfp,\beta).
\nn%\label{hm:ra+1}
\end{equation}
We assume the non-singularity in the sequel. The delta method gives
\begin{equation}
\left(
\sqrt{n}(\hat{\mfp}_{n} - \mfp),~
\sqrt{n}(\hat{\beta}_{p,n} - \beta),~
\sqrt{n}(\hat{\sig}^{\ast}_{p,n} - \sig^{\ast}_{p})
\right)\cil N_{3}(0,V(\mfp,\beta,\sig_{\cdot})),
\label{hm:ra2}
\end{equation}
where
\begin{equation}
V(\mfp,\beta,\sig_{\cdot}):=\{\nabla F(\mfp,\beta,\sig^{\ast}_{p})\}^{-1}
\Sig(\mfp,\beta,\sig_{\cdot})\{\nabla F(\mfp,\beta,\sig^{\ast}_{p})\}^{-1\top}.
\nonumber
\end{equation}

\medskip

Now, we take $m=2$ and consider $r=(2q,0)$ and $r'=(q,q)$ with
\begin{equation}
q=p/2.
\nonumber
\end{equation}
We need $q<\beta/4$ for (\ref{hm:mpvlim2}) to be in force: 
for $\beta\in(1,2)$, a naive choice would be $q=1/4$. 
We can effectively solve (\ref{hm:ra1}) as in Section \ref{hm:sec_sslp_lfm}, 
namely, in order to compute $\hat{\beta}_{n}$ we can utilize the second and third arguments of (\ref{hm:ra1}) 
since we already have the estimator $\hat{\mfp}_{n}$ of (\ref{hm:rho_n}). 
Introduce the shorthand notation
\begin{equation}
\hat{\mu}(\cdot):=\mu(\cdot;\hat{\mfp}_{n},\hat{\beta}_{p,n}).
\nn%\label{hm:muhat_def}
\end{equation}
Then, we consider the estimating equation $M_{n}(q,q)/M_{n}(2q,0)=\hat{\mu}(q,q)/\hat{\mu}(2q,0)$:
\begin{equation}
\frac{\sum_{j=1}^{n-1}|\dd_{j}X|^{q}|\dd_{j+1}X|^{q}}{\sum_{j=1}^{n}|\dd_{j}X|^{2q}}
=C_{1}(q)C_{2}(q,\hat{\mfp}_{n})
\frac{\{\Gam(1-q/\hat{\beta}_{p,n})\}^{2}}{\Gam(1-2q/\hat{\beta}_{p,n})},
\label{hm:ra3}
\end{equation}
where 
\begin{align}
C_{1}(q)&:=\frac{\Gam(1-2q)\cos(q\pi)}{\{\Gam(1-q)\cos(q\pi/2)\}^{2}},\nn\\
C_{2}(q,\hat{\mfp}_{n})&:=\frac{[\cos\{q\pi(\hat{\mfp}_{n}-1/2)\}]^{2}}{\cos\{2q\pi(\hat{\mfp}_{n}-1/2)\}}.
\nonumber
\end{align}
Since the function
\begin{equation}
g(\beta):=\frac{\{\Gam(1-q/\beta)\}^{2}}{\Gam(1-2q/\beta)}
\label{hm:added_func_g}
\end{equation}
is strictly monotone on $\beta\in(4q\vee 1,2)$, it is easy to search the root $\hat{\beta}_{p,n}$, 
which uniquely exists with probability tending to one (Figure \ref{hm:sslp_fig+1}). 
The range of $g$ becomes narrower for smaller $q$, 
so that the root $\hat{\beta}_{p,n}$ becomes too sensitive for a small change of the sample quantity 
on the left-hand side of (\ref{hm:ra3}). 

\begin{figure}[htb!]
\centering
\includegraphics[width=10cm]{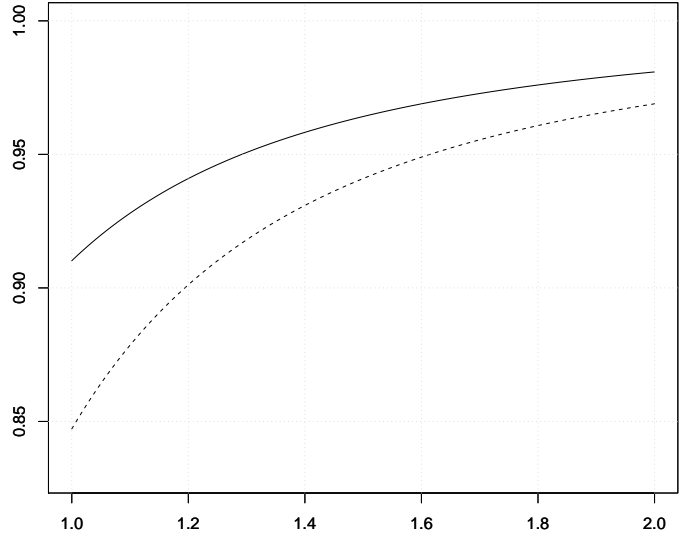}
\caption{
The function $g$ of \eqref{hm:added_func_g} on $(1,2)$ with $q=0.2$ (solid line) and $q=0.25$ (dashed line).}
\label{hm:sslp_fig+1}
\end{figure}

Thus, given a $p=2q>0$ we have got the estimates $\hat{\mfp}_{n}$ and $\hat{\beta}_{p,n}$ with leaving $\sig_{\cdot}$ unknown. 
The point here is that the bipower variation was used; 
the procedure using the first and second empirical moments as in Section \ref{hm:sec_sslp_lfm} 
is valid only when $\sig$ is constant.

We need a consistent estimator of the asymptotic covariance matrix $V(\mfp,\beta,\sig_{\cdot})$. 
Since the matrix $\Sig(\mfp,\beta,\sig_{\cdot})$ now depends on $\sig_{\cdot}$ only through 
$\sig^{\ast}_{2q}$ and $\sig^{\ast}_{4q}$, it is more accurate to use the notation 
$V(\mfp,\beta,\sig^{\ast}_{2q},\sig^{\ast}_{4q})$ instead of $V(\mfp,\beta,\sig_{\cdot})$. 
The function $V(\mfp,\beta,\sig^{\ast}_{2q},\sig^{\ast}_{4q})$ is fully explicit as 
a function of its four arguments, hence we only need to give consistent estimators of 
$\sig^{\ast}_{2q}$ and $\sig^{\ast}_{4q}$.

For example, we may proceed as follows. 
It follows from \eqref{hm:ra1} and \eqref{hm:ra2} with $p=2q$ that 
$M_{n}(2p,0)\cip\mu(2q,0)\sig^{\ast}_{2q}$. Using the estimates $(\hat{\mfp}_{n},\hat{\beta}_{p,n})$ and 
the continuous mapping theorem, we deduce that $M_{n}(2q,0)/\hat{\mu}(2q,0)\cip\sig^{\ast}_{2q}$. 
Let us remind that $\hat{\mu}(2q,0)$ can be easily computed in view of (\ref{hm:moment1}). 
Next, we replace $\beta$ by $\hat{\beta}_{p,n}$ 
in the expression $M_{n}(2q,0)=n^{2q/\beta-1}\sum_{j=1}^{n}|\dd_{j}X|^{2q}$; 
this is possible since we beforehand know that $\sqrt{n}(\hat{\beta}_{p,n}-\beta)=O_{p}(1)$. 
Thus
\begin{equation}
\hat{\sig}^{\ast}_{2q,n}:=\frac{n^{2q/\hat{\beta}_{p,n}-1}}{\hat{\mu}(2q,0)}\sum_{j=1}^{n}|\dd_{j}X|^{2q}
\cip\sig^{\ast}_{2q}.
\label{hm:ra4}
\end{equation}
By the same token, we can deduce that (still under $4q<\beta$, of course)
\begin{equation}
\hat{\sig}^{\ast}_{4q,n}:=\frac{n^{4q/\hat{\beta}_{p,n}-1}}{\hat{\mu}(2q,2q)}
\sum_{j=1}^{n-1}|\dd_{j}X|^{2q}|\dd_{j+1}X|^{2q}\cip\sig^{\ast}_{4q}.
\nonumber%\label{hm:ra5}
\end{equation}
We conclude that $V(\hat{\mfp}_{n},\hat{\beta}_{p,n},\hat{\sig}^{\ast}_{2q,n},\hat{\sig}^{\ast}_{4q,n})\cip
V(\mfp,\beta,\sig_{\cdot})$.

\medskip

Now we turn to our main objectives (A) and (B).

\subsubsection{Case (A): Skewed stable L\'evy process}%\label{hm:sec_gslp}

When $\sig_{t}\equiv\sig>0$, the process $X$ is the skewed stable L\'evy process such that 
$\mcl(X_{t})=S'_{\beta}(\mfp,\sig^{\beta}t)$, and it directly follows from (\ref{hm:ra2}) that
\begin{equation}
\left(
\sqrt{n}(\hat{\mfp}_{n} - \mfp),~
\sqrt{n}(\hat{\beta}_{p,n} - \beta),~
\sqrt{n}\{(\hat{\sig}_{p,n})^{p} - \sig^{p}\}
\right)
\cil N_{3}\big(0,V(\mfp,\beta,\sig)\big),
\label{hm:gslp1}
\end{equation}
where $V(\mfp,\beta,\sig)$ explicitly depends on $(\mfp,\beta,\sig)$; recall that $p=2q<\beta/2$. 
Hence, as soon as $V(\mfp,\beta,\sig)$ is invertible we can readily apply the delta method to (\ref{hm:gslp1}) 
to formulate the joint interval estimation of $\theta=(\mfp,\beta,\sig)$ at rate $\sqrt{n}$. 
We omit the expression of the asymptotic covariance matrix of 
$\sqrt{n}(\hat{\mfp}_{n}-\mfp,\hat{\beta}_{p,n}-\beta,\hat{\sig}_{p,n}-\sig)$.

\medskip

In summary, the following multi-step estimation is feasible for any $p\in(0,\beta/2)$.
\begin{enumerate}
\item Compute the estimate $\hat{\mfp}_{n}$ of $\mfp$ by (\ref{hm:rho_n}).
\item Using the $\hat{\mfp}_{n}$, find the root $\hat{\beta}_{p,n}$ of (\ref{hm:ra3}).
\item Using the $(\hat{\mfp}_{n},\hat{\beta}_{p,n})$ thus obtained, 
estimate $\sig$ by $(\hat{\sig}^{\ast}_{p,n})^{1/p}$ via \eqref{hm:ra4}.
\end{enumerate}
A naive choice would be $p=1/2$ ($q=1/4$), hence in particular
\begin{equation}
\hat{\sig}_{1/2,n}=\left\{
\frac{n^{1/(2\hat{\beta}_{1/2,n})-1}}{\hat{\mu}(1/2,0)}\sum_{j=1}^{n}\sqrt{|\dd_{j}X|}\right\}^{2}.
\label{hm:skewslp_eq.sig1/2}
\end{equation}

\begin{rem}{\rm 
We can deal with the case of $\beta\in(0,1)$ in an analogous way, although we then have to be 
more careful about the selection of the tuning-parameter $p$. 
In this case, more suitable would be the logarithmic-moment estimator as in the symmetric-jump case considered in 
Section \ref{hm:sslp_logMEsec}; we can derive the closed-form expressions for 
$E\{(\log|\zeta|)^{k}\}$ for $\mcl(\zeta)=S'_{\beta}(\mfp,\sig)$ 
by a slight modification of those given in \cite[Section IV]{Kur01}.
}\qed\end{rem}

\subsubsection{Case (B): time-varying scale process}\label{hm:sec_tvsp}

\runinhead{Estimator.}
We can use the same estimator of $(\mfp,\beta)$ as in the previous case, hence 
it remains to construct an estimator of $\sig^{\ast}_{\beta}=\int_{0}^{1}\sig_{s}^{\beta}ds$. 
From (\ref{hm:mpvlim2}) we have
\begin{equation}
\sqrt{n}\{M_{n}(r)-\mu(r)\sig^{\ast}_{r_{+}}\}\cil N_{1}\big(0,B(r,r)\sig^{\ast}_{2r_{+}}\big).
\label{hm:rs1}
\end{equation}
In view of the condition $\max_{l\le m}r_{l}<\beta/2$, 
we need (at least) a tripower variation for setting $r_{+}=\beta$. 
Here, setting $m=3$ and
\begin{equation}
r=r(\beta)=(\beta/3,\beta/3,\beta/3),
\nonumber
\end{equation}
we will provide an estimator of $\sig^{\ast}_{\beta}$ having specific 
rate of convergence and asymptotic distribution. 
The point here is that, different from the case (A), a direct use of (\ref{hm:ra2}) 
is not sufficient to deduce the distributional result, 
because of the dependence of $(r,r')$ on $\beta$. 
In order to utilize $M_{n}(r)$ with $r$ depending on $\beta$, we need some additional arguments. 

Let
\begin{equation}
M^{\ast}_{n}(\beta):=M_{n}(r(\beta))
=\sum_{j=1}^{n-2}\prod_{l=1}^{3}|\dd_{j+l-1}X|^{\beta/3},
\nonumber
\end{equation}
which is computable as soon as we have an estimate of $\beta$. 
We will look at the statistics $M^{\ast}_{n}(\hat{\beta}_{p,n})$, 
with the estimator $\hat{\beta}_{p,n}$ constructed beforehand.

\runinhead{Bias specification.}
We have to specify the effect of ``plugging in $\hat{\beta}_{p,n}$'', that is, how the gap
\begin{equation}
\sqrt{n}\left\{M_{n}^{\ast}(r(\beta))-\mu(r(\beta);\mfp,\beta)\sig^{\ast}_{\beta}\right\}
-\sqrt{n}\left\{M_{n}^{\ast}(\hat{\beta}_{p,n})-\hat{\mu}(r(\hat{\beta}_{p,n}))\sig^{\ast}_{\beta}\right\}
\nonumber
\end{equation}
behaves asymptotically. It will turn out that the effect is significant.

\medskip

Let
\begin{equation}
x_{nj}:=\prod_{l=1}^{3}|\dd_{j+l-1}X|.
\nonumber
\end{equation}
By means of Taylor's formula
\begin{equation}
a^{x}=a^{y}+(\log a)^{y}(x-y)+(\log a)^{2}\int_{0}^{1}(1-u)a^{y+u(x-y)}du(x-y)^{2}
\nonumber
\end{equation}
applied to the function $x\mapsto a^{x}$ ($x,y,a>0$), we get
\begin{align}
& \sqrt{n}\bigg\{M^{\ast}_{n}(\hat{\beta}_{p,n})
-\mu(r(\beta);\mfp,\beta)\sig^{\ast}_{\beta}\bigg\}
\nonumber \\
&=\sqrt{n}\bigg\{M^{\ast}_{n}(\beta)-\mu(r(\beta);\mfp,\beta)\sig^{\ast}_{\beta}\bigg\}
\nn\\
&{}\qquad
+\frac{1}{3}\sqrt{n}(\hat{\beta}_{p,n}-\beta)\sum_{j=1}^{n-2}x_{nj}^{\beta/3}\log x_{nj} \nonumber \\
&{}\qquad +\bigg\{\frac{1}{3}\sqrt{n}(\hat{\beta}_{p,n}-\beta)\bigg\}^{2}
\frac{1}{\sqrt{n}}\sum_{j=1}^{n-2}
(\log x_{nj})^{2}\int_{0}^{1}(1-u)x_{nj}^{\{\beta+u(\hat{\beta}_{p,n}-\beta)\}/3}du\nn\\
&=:\mcm_{1n}+\mcm_{2n}+\mcm_{3n}.
\label{hm:rs2}
\end{align}
Let us look at the right-hand side of (\ref{hm:rs2}) termwise.

It is evident from (\ref{hm:rs1}) that
\begin{equation}
\mcm_{1n}=O_{p}(1).
\label{hm:skweslp_m1n}
\end{equation}
Letting
\begin{equation}
y_{nj}:=\prod_{l=1}^{3}|n^{1/\beta}\dd_{j+l-1}X|=n^{3/\beta}x_{nj},
\nonumber
\end{equation}
we have
\begin{align}
\sum_{j=1}^{n-2}x_{nj}^{\beta/3}\log x_{nj}
&=\frac{1}{n}\sum_{j=1}^{n-2}y_{nj}^{\beta/3}\log y_{nj}
-\frac{3}{\beta}(\log n)\frac{1}{n}\sum_{j=1}^{n-2}y_{nj}^{\beta/3}
\nonumber \\
&=O_{p}(1)-(\log n)\frac{3}{\beta}\bigg\{
\mu(r(\beta);\mfp,\beta)\sig^{\ast}_{\beta}
+O_{p}\bigg(\frac{1}{\sqrt{n}}\bigg)\bigg\} \nonumber \\
&=O_{p}(1)-(\log n)\frac{3}{\beta}\mu(r(\beta);\mfp,\beta)\sig^{\ast}_{\beta}.
\nonumber
\end{align}
It follows that
\begin{equation}
\mcm_{2n}=-(\log n)\frac{1}{\beta}\mu(r(\beta);\mfp,\beta)\sig^{\ast}_{\beta}\sqrt{n}(\hat{\beta}_{p,n}-\beta)
+O_{p}(1).
\label{hm:skweslp_m2n}
\end{equation}
Put $\mcm_{3n}=\{\sqrt{n}(\hat{\beta}_{p,n}-\beta)/3\}^{2}\mch_{n}$. 
We will prove that $\mch_{n}=o_{p}(1)$. Fix any $\ep>0$ and $\ep_{0}\in(0,\beta/2)$. Then,
\begin{align}
P(|\mch_{n}|>\ep)&\le P\left(|\hat{\beta}_{p,n}-\beta|>\ep_{0}\right)
+P\left(|\mch_{n}|>\ep,\ |\hat{\beta}_{p,n}-\beta|\le\ep_{0}\right)
\nn\\
&=:p'_{n}+p''_{n}.
\nonumber
\end{align}
Clearly $p'_{n}\to 0$ by the $\sqrt{n}$-consistency of $\hat{\beta}_{p,n}$. As for $p''_{n}$, we first note that 
\begin{equation}
\inf_{u\in[0,1]}\frac{1}{\beta}\{\beta+u(\hat{\beta}_{p,n}-\beta)\}\ge 1-\frac{\ep_{0}}{\beta}>0
\nonumber
\end{equation}
on the event $\{|\hat{\beta}_{p,n}-\beta|\le\ep_{0}\}$. Hence,
\begin{align}
p''_{n}&=P\bigg(|\hat{\beta}_{p,n}-\beta|\le\ep_{0}, \nonumber \\
&{}\qquad\frac{1}{\sqrt{n}}\sum_{j=1}^{n-2}(\log x_{nj})^{2}\int_{0}^{1}(1-u)
y_{nj}^{\{\beta+u(\hat{\beta}_{p,n}-\beta)\}/3}n^{-\{\beta+u(\hat{\beta}_{p,n}-\beta)\}/\beta}du>\ep\bigg)
\nonumber \\
&\le P\bigg(|\hat{\beta}_{p,n}-\beta|\le\ep_{0},
\nn\\
&{}\qquad n^{\ep_{0}/\beta-1/2}\frac{1}{n}\sum_{j=1}^{n-2}(\log x_{nj})^{2}
\int_{0}^{1}(1-u)y_{nj}^{\{\beta+u(\hat{\beta}_{p,n}-\beta)\}/3}du>\ep\bigg)
\nonumber \\
&\le P\bigg(n^{\ep_{0}/\beta-1/2}
\frac{1}{n}\sum_{j=1}^{n-2}\{(\log n)^{2}+(\log y_{nj})^{2}\}(1+y_{nj})^{(\beta+\ep_{0})/3}\gtrsim\ep\bigg)
\nn \\
&\le P\bigg(n^{\ep_{0}/\beta-1/2}(\log n)^{2}
\frac{1}{n}\sum_{j=1}^{n-2}\{1+(\log|\zeta_{j}\zeta_{j+1}\zeta_{j+2}|)^{2}\}
\nn\\
&{}\qquad\cdot
(1+|\zeta_{j}\zeta_{j+1}\zeta_{j+2}|)^{(\beta+\ep_{0})/3}\gtrsim\ep\bigg)
\label{hm:ssm_pe+1}\\
&\lesssim\frac{1}{\ep}n^{\ep_{0}/\beta-1/2}(\log n)^{2}.
\label{hm:ssm_pe+2}
\end{align}
Here, for \eqref{hm:ssm_pe+1} we used 
the assumption that the process $\sig$ is bounded and bounded away from zero 
(recall the expression \eqref{hm:skewslp_expression}), 
and also Markov's inequality for \eqref{hm:ssm_pe+2}; 
the latter is possible since the condition $(\beta+\ep_{0})/3<\beta/2$ implies that
\begin{equation}
E\left[\{1+(\log|\zeta_{1}\zeta_{2}\zeta_{3}|)^{2}\}(1+|\zeta_{1}\zeta_{2}\zeta_{3}|)^{(\beta+\ep_{0})/3}\right]<\infty.
\nonumber
\end{equation}
It follows that $p''_{n}\to 0$, hence $\mch_{n}=o_{p}(1)$, from which we get
\begin{equation}
\mcm_{3n}=o_{p}(1).
\label{hm:skweslp_m3n}
\end{equation}

\medskip

Now, piecing together (\ref{hm:rs2}), \eqref{hm:skweslp_m1n}, \eqref{hm:skweslp_m2n}, and \eqref{hm:skweslp_m3n} 
we arrive at the asymptotic relation
\begin{align}
& \frac{\sqrt{n}}{\log n}\bigg\{M^{\ast}_{n}(\hat{\beta}_{p,n})
-\mu(r(\beta);\mfp,\beta)\sig^{\ast}_{\beta}\bigg\}
\nn\\
&\quad =-\frac{1}{\beta}\mu(r(\beta);\mfp,\beta)
\sig^{\ast}_{\beta}\sqrt{n}(\hat{\beta}_{p,n}-\beta)+O_{p}\bigg(\frac{1}{\log n}\bigg).
\label{hm:rs3}
\end{align}
The map $(\mfp,\beta)\mapsto\mu(r(\beta);\mfp,\beta)$ is continuously differentiable. 
Using the $\sqrt{n}$-consistency of $(\hat{\mfp}_{n},\hat{\beta}_{p,n})$ and the delta method, we obtain
\begin{equation}
\mu(r(\beta);\mfp,\beta)=
\mu(r(\hat{\beta}_{p,n});\hat{\mfp}_{n},\hat{\beta}_{p,n})+O_{p}\bigg(\frac{1}{\sqrt{n}}\bigg).
\label{hm:rs4}
\end{equation} 
Substituting (\ref{hm:rs4}) in (\ref{hm:rs3}), we end up with
\begin{equation}
\frac{\sqrt{n}}{\log n}\bigg\{\frac{M^{\ast}_{n}(\hat{\beta}_{p,n})}
{\mu(r(\hat{\beta}_{p,n});\hat{\mfp}_{n},\hat{\beta}_{p,n})}
-\sig^{\ast}_{\beta}\bigg\}=-\frac{1}{\beta}\sig^{\ast}_{\beta}\sqrt{n}(\hat{\beta}_{p,n}-\beta)+O_{p}\bigg(\frac{1}{\log n}\bigg),
\label{hm:rs5}
\end{equation}
which implies that
\begin{equation}
\hat{\sig}^{\ast}_{\beta,n}:=\frac{M^{\ast}_{n}(\hat{\beta}_{p,n})}
{\mu(r(\hat{\beta}_{p,n});\hat{\mfp}_{n},\hat{\beta}_{p,n})}
\label{hm:rs6}
\end{equation}
serves as a $(\sqrt{n}/\log n)$-consistent estimator of $\sig^{\ast}_{\beta}$ 
having the asymptotic mixed normality:
\begin{equation}
\frac{\sqrt{n}}{\log n}(\hat{\sig}^{\ast}_{\beta,n}-\sig^{\ast}_{\beta})
\cil N_{1}\left(0,
~\bigg(\frac{\sig^{\ast}_{\beta}}{\beta}\bigg)^{2}V_{22}(\mfp,\beta,\sig^{\ast}_{p},\sig^{\ast}_{2p})
\right)
\nonumber
\end{equation}
where $V_{22}$ denotes the $(2,2)$th entry of $V$; 
recall that $p$ is a tuning parameter to be given {\it a priori}. 
As mentioned in Section \ref{hm:sec_skewlp_preliminaries}, 
a consistent estimator of the asymptotic random covariance matrix 
can be constructed through plugging in consistent estimators of its arguments.

\medskip

The stochastic expansion (\ref{hm:rs5}) clarifies the asymptotic linear dependence of 
$\sqrt{n}(\hat{\beta}_{p,n}-\beta)$ and $(\sqrt{n}/\log n)(\hat{\sig}^{\ast}_{\beta,n}-\sig^{\ast}_{\beta})$, 
which occurs even for constant $\sig$ if we try to estimate $(\beta,\sig^{\beta})$ instead of $(\beta,\sig)$. 
Put simply, plugging in a $\sqrt{n}$-consistent estimator of $\beta$ into the index $r$ of the MPV $M_{n}(r)$ 
slows down estimation of $\sig^{\ast}_{\beta}$ from $\sqrt{n}$ to $\sqrt{n}/\log n$. 
We refer to \cite[Theorem 3]{Tod13} for a related result.

%%%

\subsubsection{Simulation experiments}\label{hm:sec_skewslp_sim}

\runinhead{Case (A).}
We set
\begin{equation}
(\mfp,\beta)=(0.7638,1.2),~(0.5984,1.5),~(0.5467,1.7),~(0.5132,1.9)
\nonumber
\end{equation}
for the true values, with $\rho=-0.5$ and $\sig=1$ in common. 
For each value of $(\mfp,\beta,\sig)$, we set $n=500$, $1000$, $2000$, and $5000$. 
In all cases, the tuning parameter $q=1/4$, and 
$1000$ independent sample paths of $X$ are generated; 
the estimators are given by \eqref{hm:rho_n}, \eqref{hm:ra3}, and \eqref{hm:skewslp_eq.sig1/2}. 
Empirical means and empirical RMSEs based on $1000$ independent estimates are computed. 
The results are reported in Table \ref{hm:skewslp_table1}.
\begin{itemize}
\item On the one hand, $(\hat{\mfp}_{n},\hat{\beta}_{n})$ is, despite of its simplicity, rather reliable. 

\item On the other hand, variance of $\hat{\sig}_{n}$ is larger compared with those of $\hat{\mfp}_{n}$ and $\hat{\beta}_{n}$, 
while the bias seems small. Moreover, as $\beta$ gets close to $2$, 
the performance of $\hat{\sig}_{n}$ becomes better 
while that of $(\hat{\mfp}_{n},\hat{\beta}_{p,n})$ is much less affected.

\end{itemize}

We have also conducted simulations with $q$ other than $1/4$, 
and observed that a change of $q$ within its admissible region 
does not lead to a drastic change unless it is too small.

\begin{table}[htbp]
\begin{small}
\begin{center}
\begin{tabular}{lrlcccccccc}
\hline\\[-3mm]
True $\beta$ & $n$ & & \multicolumn{2}{c}{$\hat{\mfp}_{n}$} & & 
\multicolumn{2}{c}{$\hat{\beta}_{n}$} & & \multicolumn{2}{c}{$\hat{\sig}_{n}$} \\
\hline\\
%&&&&& \\
1.2 &500 & & 0.7627 & (0.0186) && 1.2026 & (0.0790) && 1.1021 & (0.8717) \\
&1000& & 0.7634 & (0.0137) && 1.2031 & (0.0575) && 1.0450 & (0.4643) \\
&2000& & 0.7645 & (0.0096) && 1.2031 & (0.0437) && 1.0253 & (0.5102) \\
&5000& & 0.7636 & (0.0061) && 1.2023 & (0.0313) && 1.0123 & (0.2854) \\
&&&&&&&&&&\\
%&&&&&& \\
1.5 &500 & & 0.5988 & (0.0222) && 1.4929 & (0.1030) && 1.0751 & (0.4066) \\
&1000& & 0.5981 & (0.0162) && 1.5010 & (0.0757) && 1.0289 & (0.2549) \\
&2000& & 0.5986 & (0.0106) && 1.4986 & (0.0564) && 1.0284 & (0.2355) \\
&5000& & 0.5984 & (0.0073) && 1.4983 & (0.0364) && 1.0169 & (0.1516) \\
&&&&&&&&&\\
%&&&&&& \\
1.7 &500 & & 0.5476 & (0.0219) && 1.6810 & (0.1103) && 1.0633 & (0.2359) \\
&1000& & 0.5474 & (0.0158) && 1.6830 & (0.0823) && 1.0567 & (0.1948) \\
&2000& & 0.5472 & (0.0113) && 1.6930 & (0.0625) && 1.0308 & (0.1611) \\
&5000& & 0.5466 & (0.0070) && 1.6977 & (0.0375) && 1.0126 & (0.1022) \\
&&&&&&&&&\\
%&&&&&& \\
1.9 &500 & & 0.5129 & (0.0224) && 1.8553 & (0.1026) && 1.0821 & (0.1767) \\
&1000& & 0.5133 & (0.0164) && 1.8767 & (0.0808) && 1.0535 & (0.1568) \\
&2000& & 0.5131 & (0.0109) && 1.8870 & (0.0579) && 1.0330 & (0.1111) \\
&5000& & 0.5128 & (0.0073) && 1.8971 & (0.0401) && 1.0097 & (0.0809) \\
%&&&&&&&&&\\
\\
\hline
\end{tabular}
\end{center}
\end{small}
\caption{Estimation results for the true parameters $(\mfp,\beta)=(0.7638,1.2)$, $(0.5984,1.5)$, 
$(0.5467,1.7)$, and $(0.5132,1.9)$ with $\sig=1$ in common for the skewed stable L\'evy processes. 
In each case, the empirical mean and the empirical RMSE (in parenthesis) are given.}
\label{hm:skewslp_table1}
\end{table}

\runinhead{Case (B).} 
Next we observe the time-varying but non-random scale
\begin{equation}
\sig_{t}^{\beta}=\frac{2}{5}\bigg\{\cos(2\pi t)+\frac{3}{2}\bigg\},
\label{hm:sim1}
\end{equation}
so that $\sig^{\ast}_{\beta}=0.6$. 
With the same choices of $(\mfp,\beta)$, $q$, and $n$ as in the case (A), 
we obtained the results in Table \ref{hm:skewslp_table2}; 
the estimator of $\sig^{\ast}_{\beta}$ here is based on (\ref{hm:rs6}). 
The estimation performance about $(\mfp,\beta)$ shows a similar tendency to the case (A), 
while $\hat{\sig}^{\ast}_{\beta,n}$ exhibits an upward bias in most cases.

\medskip

\begin{table}[htbp]
\begin{small}
\begin{center}
\begin{tabular}{lrlcccccccc}
\hline\\[-3mm]
True $\beta$ & $n$ && \multicolumn{2}{c}{$\hat{\mfp}_{n}$} 
& & \multicolumn{2}{c}{$\hat{\beta}_{n}$} & & \multicolumn{2}{c}{$\hat{\sig}^{\ast}_{\beta,n}$} \\
\hline\\
%&&&&& \\
1.2&500 & & 0.7632 & (0.0179) && 1.1951 & (0.0794) && 0.6730 & (0.3857) \\
&1000& & 0.7636 & (0.0139) && 1.2042 & (0.0619) && 0.6274 & (0.3094) \\
&2000& & 0.7638 & (0.0098) && 1.2044 & (0.0472) && 0.6105 & (0.2323) \\
&5000& & 0.7641 & (0.0059) && 1.2025 & (0.0305) && 0.6029 & (0.1521) \\
&&&&&&&&&\\
%&&&&&& \\
1.5&500 & & 0.5978 & (0.0220) && 1.4877 & (0.1023) && 0.6697 & (0.3031) \\
&1000& & 0.5981 & (0.0159) && 1.4908 & (0.0733) && 0.6551 & (0.2488) \\
&2000& & 0.5985 & (0.0111) && 1.4960 & (0.0573) && 0.6349 & (0.2033) \\
&5000& & 0.5987 & (0.0069) && 1.4990 & (0.0376) && 0.6151 & (0.1414) \\
&&&&&&&&&\\
%&&&&&& \\
1.7&500 & & 0.5460 & (0.0216) && 1.6727 & (0.1038) && 0.6832 & (0.2465) \\
&1000& & 0.5465 & (0.0160) && 1.6801 & (0.0820) && 0.6714 & (0.2280) \\
&2000& & 0.5468 & (0.0113) && 1.6931 & (0.0600) && 0.6318 & (0.1607) \\
&5000& & 0.5465 & (0.0071) && 1.6988 & (0.0393) && 0.6116 & (0.1135) \\
&&&&&&&&&\\
%&&&&&& \\
1.9&500 & & 0.5130 & (0.0229) && 1.8440 & (0.1039) && 0.7196 & (0.2233) \\
&1000& & 0.5131 & (0.0159) && 1.8703 & (0.0823) && 0.6762 & (0.1897) \\
&2000& & 0.5138 & (0.0114) && 1.8851 & (0.0588) && 0.6412 & (0.1349) \\
&5000& & 0.5135 & (0.0068) && 1.8956 & (0.0411) && 0.6168 & (0.0998) \\
%&&&&&&&&&\\
\\
\hline
\end{tabular}
\end{center}
\end{small}
\caption{Estimation results for the true parameters $(\mfp,\beta)=(0.7638,1.2)$, $(0.5984,1.5)$, 
$(0.5467,1.7)$, and $(0.5132,1.9)$ with $\sig^{\ast}_{\beta}=0.6$ of (\ref{hm:sim1}) in common. 
In each case, the empirical mean and the empirical RMSE (in parenthesis) are given.}
\label{hm:skewslp_table2}
\end{table}

Overall, except for the relatively larger variances and upward biases in estimating the integrated scale, 
our simulation results say that finite-sample performance of our estimators is reliable despite of their simplicity.

%%%%%

\subsection{Remark on estimation of general stable L\'evy process}\label{hm:sec_gen_slp}

So far, we have separately treated the symmetric-jump case with drift and the skewed-jump case without drift 
in Sections \ref{hm:sec_sym.slp_me} and \ref{hm:sec_skewslp_me}, respectively. 
Unfortunately, none of them can directly apply to the full stable {\lp} model 
$\mcl(X_{1})=S_{\beta}(\sig,\rho,\gam)$ with $\theta=(\beta,\sig,\rho,\gam)$. 
In this section we will briefly mention a naive but promising way built on the previous results. 
The terminal sampling time $T_{n}$ may or may not be bounded.

\medskip

We may handle the general skewed case with trend 
through some convenient transformations of the increments $(\dd_{j}X)_{j=1}^{n}$. 
Let $\beta\ne 1$ and pick any $(c_{1},\dots,c_{q})\in\mbbr^{q}\backslash\{0\}$. 
%Let $a^{\la\beta\ra}:=|a|^{\beta}\sgn(a)$. 
Then, it follows from \eqref{hm:ex_st1} that
\begin{align}
&\mcl\left(\sum_{k=1}^{q}c_{k}\dd_{k}X\right)
\nn\\
&=S_{\beta}\left(
h_{n}^{1/\beta}\sig\left(\sum_{k=1}^{q}|c_{k}|^{\beta}\right)^{1/\beta},~
\frac{\ds{\rho\sum_{k=1}^{q}|c_{k}|^{\beta}\sgn(c_{k})}}{\ds{\sum_{k=1}^{q}|c_{k}|^{\beta}}},~
h_{n}\gam\sum_{k=1}^{q}c_{k}\right).
\label{hm:slp_inc_trans}
\end{align}
Making use of \eqref{hm:slp_inc_trans} as in \cite{Kur01} (see also \cite[Chapter 4]{Zol86}), 
we get the following distributional identities:
\begin{align}
\mcl(\dd_{j}X-\dd_{j-1}X)&=
S_{\beta}\left(2^{1/\beta}h_{n}^{1/\beta}\sig,~0,~0\right),
\label{hm:slptrans_sym}\\
\mcl(\dd_{j+1}X+\dd_{j-1}X-2\dd_{j}X)&=
S_{\beta}\left(
(2+2^{\beta})^{1/\beta}h_{n}^{1/\beta}\sig,~\frac{2-2^{\beta}}{2+2^{\beta}}\rho,~0\right),
\label{hm:slptrans_center}\\
\mcl(\dd_{j+1}X+\dd_{j-1}X-2^{1/\beta}\dd_{j}X)&=
S_{\beta}\left(2^{2/\beta}h_{n}^{1/\beta}\sig,~0,~(2-2^{1/\beta})h_{n}\gam\right).
\label{hm:slptrans_deskew}
\end{align}
Note that the relation \eqref{hm:slp_inc_trans} generally fails to hold for $\beta=1$; 
the symmetrization \eqref{hm:slptrans_sym} is valid even for $\beta=1$, 
but \eqref{hm:slptrans_center} and \eqref{hm:slptrans_deskew} are not.

We can adopt the estimation methods discussed in Sections \ref{hm:sec_sym.slp_me} and \ref{hm:sec_skewslp_me}. 
A naive practical way for joint estimation of $\theta=(\beta,\sig,\mfp,\gam)$ would be as follows:
\begin{itemize}
\item First, we apply \eqref{hm:slptrans_sym} to estimate $(\beta,\sig)$ as in Section \ref{hm:sec_sym.slp_me};

\item Second, changing the skewness parameter to the positivity parameter 
(recall that the relation \eqref{hm:skew_relation}) and then making use of \eqref{hm:slptrans_center}, 
we apply the results presented in Section \ref{hm:sec_skewslp_me} to estimate $\mfp$;

\item Finally, in order to estimate the remaining trend parameter $\gam$ as in Section \ref{hm:sec_sym.slp_me} 
(by the sample median), we apply \eqref{hm:slptrans_deskew} with substituting 
the estimator $\hat{\beta}_{n}$ constructed in the first step 
into $\beta$ of the deskewed increments ``$\dd_{j+1}X+\dd_{j-1}X-2^{1/\beta}\dd_{j}X$''. 
\end{itemize}
To keep having rowwise independent arrays in the above scenario, 
the actual number of data must become $[n/2]$ for \eqref{hm:slptrans_sym}, 
and $[n/3]$ for \eqref{hm:slptrans_center} and \eqref{hm:slptrans_deskew}. 
The efficiency loss caused by this data-number reduction may get diminished if we look not at 
\begin{align}
& (\dd_{2l}X-\dd_{2l-1}X)_{l=1}^{[n/2]}, \quad
(\dd_{3l}X+\dd_{3l-2}X-2\dd_{3l-1}X)_{l=1}^{[n/3]},\nn\\
& (\dd_{3l}X+\dd_{3l-2}X-2^{1/\beta}\dd_{3l-1}X)_{l=1}^{[n/3]},
\nonumber
\end{align}
but at
\begin{align}
& (\dd_{j}X-\dd_{j-1}X)_{j=2}^{n},\quad 
(\dd_{j}X+\dd_{j-2}X-2\dd_{j-1}X)_{j=3}^{n}, \nn\\
& (\dd_{j}X+\dd_{j-2}X-2^{1/\beta}\dd_{j-1}X)_{j=3}^{n}.
\nonumber
\end{align}
But then, since the random variable are no longer independent even conditional on $\sig$, 
the forms of the asymptotic covariance matrices in the methods of moments discussed 
in Sections \ref{hm:sec_sym.slp_me} and \ref{hm:sec_skewslp_me} take different forms in a similar manner to \cite{Tod13}. 
Further, and more importantly, 
we need to look at asymptotic effect of plugging in $\hat{\beta}_{n}$ in the transformed increments in the final step 
for estimation of $\gam$.

%%%%%
%%%%%

\subsection{Remark on locally stable L\'evy process}\label{hm:sec_lslp}

The great advantage of the stable {\lp es} is the inherent scaling property \eqref{hm:ex_st3}, 
which enables us to {\it exactly} reduce things to those concerning i.i.d. stable random variables. 
As we have seen in the previous subsections, we do not suffer from the annoying lack of finite moments so much, 
by making use of sample median and appropriate moment fittings together with convenient transforms of the increments.

The infinite-variance tail may be too heavy in several modeling purpose. 
In view of Lemma \ref{hm:lem_ltst}, a far-reaching extension of the non-Gaussian stable L\'evy process is immediate: 
we call $X$ a {\it locally stable {\lp}} if there exist 
a constant $\mu\in\mbbr$ and a non-random positive function $\sig(h)\to 0$ as $h\to 0$ such that the linear transform
\begin{equation}
\sig(h)^{-1}(X_{h}-\mu h)\cil F
\nn
\end{equation}
for a strictly $\beta$-stable distribution $F$; specifically, all the possible cases are 
$F=S_{\beta}(\sig,\rho,0)$ for $\beta\ne 1$, and $F=S_{1}(\sig,0,\gam)$. 
Recall that the scaling function $\sig(\cdot)$ is necessarily of regular variation with index $1/\beta$ 
where $\beta\in(0,2]$, most typically $\sig(h)=h^{1/\beta}$. 
We claim that the whole locally stable {\lp es} constitute 
an important subclass of general infinite-activity L\'evy processes, 
since they can exhibit not only approximate scaling property in small-time, 
but also a variety of tail behavior of the {\lm}. 
We should note, however, that convergence of moments of $\sig(h)^{-1}(X_{h}-\mu h)$ for $h\to 0$ is quite severe. 
As a matter of fact, the convergence in $L^{2}$ cannot hold regardless of the tail behavior of $\mcl(X_{h})$: 
assume, for example, that $h^{-1/\beta}X_{h}\cil S$ with $\mcl(S)=S_{\beta}(1)$ and $E(|X_{1}|^{\beta})<\infty$. 
Then we have 
$\sup_{h>0}E(|h^{-1/\beta}X_{h}|^{q'})\lesssim\sup_{h>0}h^{1-q'/\beta}\lesssim 1$ only for $q'\le\beta$ 
(see \cite{LusPag08}), so that 
\begin{equation}
E(|h^{-1/\beta}X_{h}|^{q})\to E(|S|^{q})
\label{hm:me_oconv}
\end{equation}
may hold only when $q<\beta$. 
This is in sharp contrast to the case of Wiener process, where $(h^{-1/2}X_{h})_{h>0}$ is $L^{q}$-bounded for any $q>0$. 

By the way, we have already encountered in this chapter several concrete examples of the locally stable {\lp}: 
the inverse-Gaussian subordinator is locally half-stable, and 
the Meixner and the normal inverse-Gaussian {\lp es} are locally Cauchy. 
One of the other prominent examples is the {\it (exponentially) tempered stable {\lp}} 
(see \cite{Ros07} and the references therein), 
which has several merits from both theoretical and numerical points of view; 
we refer to \cite{KawMas11ts} for 
a comparative study of numerical recipes for generating tempered-stable random numbers 
as well as a summary of basic facts concerning the tempered stable {\lp es}. 
A detailed study of the tempered stable model with a view toward application to finance can be found in \cite{KucTap13}. 
Yet another interesting example is the {\it normal tempered stable {\lp}} \cite{BarShe02}, 
which is defined as the normal variance-mean mixture of a tempered $\beta$-stable subordinator $\tau$:
\begin{equation}
X_{t}=t\mu+\beta\tau_{t}+w_{\tau_{t}},
\label{hm:lp_nvmm}
\end{equation}
where $w$ is a standard Wiener process independent of $\tau$.

\medskip

For a pure-jump {\lp} to have the local-stable property, 
it suffices to look at the behavior of the {\lm} $\nu(dz)$ near the origin. 
It is the case especially if $\nu(dz)=g(z)dz$ in a neighborhood $U$ of the origin with the {\ld} $g$ satisfying that
\begin{equation}
g(z)=\frac{c}{|z|^{1+\beta}}\{1+g^{\natural}(z)\}
\nonumber
\end{equation}
for constants $c>0$ and $\beta\in(0,2)$ and for a continuous function $g^{\natural}$ 
which is bounded in $U$ with $\lim_{|z|\to 0}g^{\natural}(z)=0$; 
see \cite[Lemma 4.4]{Mas10} and \cite{Tod13} as well as the references therein 
for details and more general criteria. 
Further, the following two points are worth mentioning.
\begin{itemize}
\item 
If $\tau^{+}$ and $\tau^{-}$ are mutually independent 
locally $\beta_{+}$-stable and locally $\beta_{-}$-stable subordinators with no drift, 
then $X_{t}:=\tau^{+}_{t}-\tau^{-}_{t}$ is a locally $\beta$-stable {\lp} with $\beta:=\beta_{+}\vee\beta_{-}$; 
in particular, if $\beta_{+}>\beta_{-}$ (resp. $\beta_{+}<\beta_{-}$), 
then the asymptotic distribution of $h^{-1/\beta}X_{h}$ is 
spectrally positive $\beta_{+}$-stable (resp. spectrally negative $\beta_{-}$-stable), 
that is to say, the more active part is dominant.

\item Given a locally $\beta$-stable subordinator $\tau$ with no drift, 
a {\lp} $X$ of the form \eqref{hm:lp_nvmm} defines a locally $2\beta$-stable {\lp} on $\mbbr$; 
indeed, it is easy to see that
\begin{equation}
h^{-1/(2\beta)}(X_{h}-\mu h)\cil Y:=(S^{+}_{\beta})^{1/2}\eta
\nonumber
\end{equation}
for independent random variables $S^{+}_{\beta}$ and $\eta$ 
where $\mcl(S^{+}_{\beta})$ is positive strictly $\beta$-stable and $\mcl(\eta)$ is standard normal. 
The distribution $\mcl(Y)$ is symmetric $2\beta$-stable; see Sato \cite[Theorem 30.1]{Sat99} for general details.

\end{itemize}
The asymptotic singularity in joint estimation of the index $\beta$ and a scale-parameter 
(recall Theorem \ref{hm:sslp_th1}) would also emerge for locally stable {\lp es}. 
This is expected from the form of the likelihood function of the totally skewed tempered stable distribution, 
whose probability density takes the exponential-tilting form $x\mapsto ce^{-\lam x}p_{\beta}(x)$ 
with a totally skewed $\beta$-stable probability density $p_{\beta}$, 
and any general tempered stable density is a convolution of them; see \cite[Proposition 1]{BaeMee10} for details.

%%%%%
%%%%%

\section{Uniform tail-probability estimate of statistical random fields}\label{hm:sec_pldi}

In practice we may resort to some tractable $M$- or $Z$-estimation procedure 
other than likelihood based ones, in compensation for possible efficiency loss (e.g. \cite[Chapter 5]{vdV98}). 
In this section we will prove a uniform tail-probability estimate of statistical random fields, 
applying the general polynomial type large deviation inequality developed in \cite{Yos11}.

We are assuming that the parameter space $\Theta\subset\mbbr^{p}$ is a bounded convex domain. 
Throughout this section we fix a $\theta_{0}\in\Theta$ to be estimated. 
An estimator $\hat{\theta}_{n}$ of $\theta_{0}$ is usually defined to be any
\begin{equation}
\hat{\theta}_{n}\in\argmax_{\theta\in\overline{\Theta}}\mbbm_{n}(\theta)
\label{hm:M_estimator_def}
\end{equation}
for some contrast function $\mbbm_{n}:\Theta\to\mbbr$. 
By means of the argmax continuous mapping argument \cite[Section 5.9]{vdV98}, 
we can derive an asymptotic distribution of $\hat{\theta}_{n}$ by verifying 
the weak convergence of the statistical random field associated with $\mbbm_{n}$ 
(also referred to as the local criterion function) on compact sets, 
the identifiability condition on the weak limit, 
and the tightness of the suitably scaled estimator, say $A_{n}(\theta_{0})^{-1}(\hat{\theta}_{n}-\theta_{0})$, 
where the rate matrix satisfies that $A_{n}(\theta_{0})>0$ and $|A_{n}(\theta_{0})|\to 0$. 
Possible form of $\mbbm_{n}$ is strongly model-dependent and may be several things, 
and wide applicability (simplicity) and large loss of asymptotic efficiency may often occur simultaneously. 
Let us recall that we can specify an asymptotically optimal phenomenon 
if we have the asymptotic normality of the form 
$A_{n}(\theta_{0})^{-1}(\hat{\theta}_{n}-\theta_{0})\cil N_{p}\left(0,\Sig(\theta_{0})\right)$ 
for a regular estimator and if the LAN is in force (cf. Section \ref{hm:sec_la}): 
the LAN tells us which $A_{n}(\theta_{0})$ and $\Sig(\theta_{0})$ are the best possible.

We here consider $\hat{\theta}_{n}$ of \eqref{hm:M_estimator_def} with $\mbbm_{n}$ taking the form
\begin{equation}
\mbbm_{n}(\theta):=-|\mbbg_{n}(\theta)|^{2}=-\sum_{k=1}^{p}\mbbg_{k,n}(\theta)^{2}
\label{hm:MC_H_def}
\end{equation}
for a continuous random function $\mbbg_{n}=(\mbbg_{k,n})_{k=1}^{p}: \Theta\to\mbbr^{p}$, 
each $\mbbg_{k,n}(\theta)$ being $\sig(X_{t^{n}_{j}}; j\le n)$-measurable.

The estimate $\hat{\theta}_{n}$ can be any root of $\mbbg_{n}(\theta)=0$ if exists. 
For brevity, we here suppose that there exists a $\theta\in\overline{\Theta}$ such that $\mbbg_{n}(\theta)=0$ 
from the beginning. The merit of the form \eqref{hm:MC_H_def} is that 
it provides us with a unified way to deal with $Z$-estimation such as the method of moments, 
as well as $M$-estimation such as minimum-distance and quasi-likelihood type contrast functions.

We will prove an extension of the argument \cite[Theorem 3.5(a)]{Mas13as} to the two-scaling case, 
from which directly follows the $L^{q}$-boundedness of the scaled $M$-estimator; 
in particular, we can deduce the convergence of moments of the scaled $M$-estimator.

%%%%%

\subsection{Polynomial type large deviation inequality}\label{hm:ssec_pldi}

To handle a contrast function of the form \eqref{hm:MC_H_def} possibly having more than one scaling rate, 
we will prove a general result on the polynomial type large deviation estimate. 
For this purpose, in this section we proceed with an {\it auxiliary} setting, and will return to our main context 
in Section \ref{hm:sec_2step_pldi}.

Suppose that we are given the random function $\mbbh_{n}$ of the form
\begin{equation}
\mbbh_{n}(\zeta,\tau)=-\frac{1}{b_{n}}|\mbbs_{n}(\zeta,\tau)|^{2},
\label{hm:def_constastH}
\end{equation}
where $\theta:=(\zeta,\tau)\in\Theta_{\zeta}\times\Theta_{\tau}=:\Theta$ with 
$\Theta_{\zeta}\subset\mbbr^{p_{\zeta}}$ and $\Theta_{\tau}\subset\mbbr^{p_{\tau}}$ being bounded convex domains, 
where $(b_{n})$ is a sequence of positive constants such that $b_{n}\to\infty$, 
and where $\mbbs_{n}=(\mbbs_{k,n})_{k=1}^{p}:\Theta_{\zeta}\times\Theta_{\tau}\to\mbbr^{p}$ is a continuous random function. 
We fix a true parameter value $\theta_{0}=(\zeta_{0},\tau_{0})\in\Theta$ and let $a_{n}(\theta_{0})=a_{n}:=b_{n}^{-1/2}$ 
($b_{n}$ may depend on $\theta_{0}$); 
in the sequel, we will largely omit the dependence on the fixed argument $\theta_{0}$ from notation. 
Informally speaking, the first element ``$\zeta$'' can be estimated more quickly than the remaining ``$\tau$'', 
the latter being regarded as a nuisance parameter at first stage; 
in the single-scaling case we may ignore $\tau$ from the very beginning. 
In case where there are two different scalings for $\zeta$ and $\tau$ with $\mbbh_{n}$ being the log-likelihood 
continuously differentiable in $\theta$, 
we may think of the score function $\mbbs_{n}(\theta)=(\p_{\zeta}\mbbh_{n}(\theta),~\p_{\tau}\mbbh_{n}(\theta))$; 
in this case, the squared-norm form \eqref{hm:def_constastH} is redundant 
and we may set $\mbbh_{n}$ to be the log-likelihood itself. 
Nevertheless, as mentioned before the form \eqref{hm:def_constastH} may be more beneficial 
since it can subsume the $Z$-estimation setting.

\medskip

We now introduce the statistical random field
\begin{equation}
\mbbz_{n}(u;\tau):=\exp\left\{\mbbh_{n}(\zeta_{0}+a_{n}u,\tau)-\mbbh_{n}(\zeta_{0},\tau)\right\}
\label{hm:MC_Z_def}
\end{equation}
for $u\in\{v\in\mbbr^{p_{\zeta}};~\zeta_{0}+a_{n}v\in\Theta_{\zeta}\}$. 
Following \cite{Yos11}, we will provide a set of sufficient conditions under which 
the {\it polynomial type large deviation inequality (PLDI)} holds: 
given a constant $M>0$, there exists a constant $C_{M}>0$ such that
\begin{equation}
\sup_{n\in\mbbn}P_{0}\left(\sup_{|u|>r}~\sup_{\tau\in\Theta_{\tau}}
\mbbz_{n}(u,\tau)\ge e^{-r}\right)\le\frac{C_{M}}{r^{M}},\quad r>0,
\label{hm:PLDI_def2}
\end{equation}
where $P_{0}:=P_{\theta_{0}}$. We define $\hat{\theta}_{n}=(\hat{\zeta}_{n},\hat{\tau}_{n})$ 
to be any $\hat{\theta}_{n}\in\argmax_{\theta\in\overline{\Theta}}\mbbh_{n}(\theta)$. Let
\begin{equation}
\hat{u}_{n}:=a_{n}^{-1}(\hat{\zeta}_{n}-\zeta_{0}),
\nonumber
\end{equation}
which is to have a non-trivial asymptotic distribution, namely, $a_{n}$ is 
the right norming for estimating $\zeta_{0}$ by $\mbbh_{n}$. 
Since $\sup_{\tau\in\Theta_{\tau}}\mbbz_{n}(\hat{u}_{n},\tau)\ge 1$ by the definition of $\hat{\theta}_{n}$, 
the PLDI (\ref{hm:PLDI_def2}) gives
\begin{equation}
\sup_{n\in\mbbn}P_{0}(|\hat{u}_{n}|>r)
\le\sup_{n\in\mbbn}P_{0}\left(\sup_{|u|>r}~\sup_{\tau\in\Theta_{\tau}}
\mbbz_{n}(u,\tau)\ge 1\right)\le\frac{C_{M}}{r^{M}},\quad r>0,
\nonumber
\end{equation}
entailing the $L^{q}(P_{0})$-boundedness $\sup_{n}E_{0}(|\hat{u}_{n}|^{q})<\infty$ for $q\in(0,M)$ 
as well as the tightness of $(\hat{u}_{n})_{n}$. 
Therefore, if in particular $\hat{u}_{n}\cil\hat{u}_{0}$ for some random variable $\hat{u}_{0}$, 
then, supposing for brevity that $\hat{u}_{0}$ is defined on the original probability space, 
we immediately get the convergence of moments
\begin{equation}
E_{0}\{f(\hat{u}_{n})\}\to E\{f(\hat{u}_{0})\}
\nonumber
\end{equation}
for any measurable function $f:\mbbr^{p}\to\mbbr$ satisfying that $\lim_{|u|\to\infty}|u|^{-q}|f(u)|<\infty$. 
This greatly improves the mode of convergence of $\hat{u}_{n}$. 

\medskip

It is convenient first to state a general theorem without specific form of $\mbbs_{n}$.

\begin{ass}[Smoothness]
The random function $\mbbs_{n}(\cdot,\tau)$ for each $\tau$ 
is of class $\mcc^{3}(\Theta_{\zeta})$, $P_{0}$-a.s, and moreover, 
$\p_{\zeta}^{k}\mbbs_{n}(\cdot)$ for $k\in\{0,1,2,3\}$ 
can be continuously extended to the boundary of $\Theta$; 
we denote the extended versions by the same notations.
\label{hm:MC_A1}
\end{ass}

\begin{ass}[Bounded moments]
For every $K>0$, we have
\begin{equation}
\sup_{n\in\mbbn}E_{0}\left(\sup_{\tau\in\Theta_{\tau}}
\left|a_{n}\mbbs_{n}(\zeta_{0},\tau)\right|^{K}\right)
+\max_{0\le l\le 3}
\sup_{n\in\mbbn}E_{0}\left(\sup_{\theta\in\Theta}
\left|\frac{1}{b_{n}}\p_{\zeta}^{l}\mbbs_{n}(\theta)\right|^{K}\right)<\infty.
\nonumber
\end{equation}
\label{hm:MC_A2}
\end{ass}

\begin{ass}[Limits]
\begin{description}
\item[(a)] There exist a non-random function $\mbbs_{0}: \Theta\to\mbbr^{p}$ and 
positive constants $\chi=\chi(\theta_{0})$ and $\ep_{0}$ such that: 
$\mbbs_{0}(\zeta_{0},\tau)=0$ for every $\tau$; 
$\sup_{\theta}|\mbbs_{0}(\theta)|<\infty$; 
$|\mbbs_{0}(\theta)|^{2}\ge\chi|\zeta-\zeta_{0}|^{2}$ for every $\theta\in\Theta$; and
\begin{equation}
\sup_{n\in\mbbn}E_{0}\left\{\sup_{\theta\in\Theta}\left|b_{n}^{\ep_{0}}
\left(\frac{1}{b_{n}}\mbbs_{n}(\theta)-\mbbs_{0}(\theta)\right)\right|^{K}\right\}<\infty
\nn%\label{hm:MC_A3_m1}
\end{equation}
for every $K>0$.

\item[(b)] There exist non-random functions 
$\mbbs_{1,\infty}'(\zeta_{0},\cdot),\dots,\mbbs_{p,\infty}'(\zeta_{0},\cdot): \Theta_{\tau}\to\mbbr^{p_{\zeta}}$ 
and a positive constant $\ep_{1}$ such that: 
$\max_{1\le k\le p}\sup_{\tau}|\mbbs'_{k,0}(\zeta_{0},\tau)|<\infty$; 
the minimum eigenvalue of the matrix
\begin{equation}
\Gam_{0}(\tau):=2\sum_{k=1}^{p}\left\{\mbbs_{k,0}'(\zeta_{0},\tau)\right\}^{\otimes 2}
\nonumber
\end{equation}
is bounded away from zero uniformly in $\tau\in\Theta_{\tau}^{-}$; and
\begin{equation}
\sup_{n\in\mbbn}E_{0}\left\{\sup_{\tau\in\Theta_{\tau}}
\left|b_{n}^{\ep_{1}}\left(\frac{1}{b_{n}}\p_{\zeta}\mbbs_{k,n}(\zeta_{0},\tau)
-\mbbs_{k,0}'(\zeta_{0},\tau)\right)\right|^{K}\right\}<\infty,\qquad k=1,\dots,p.
\nn%\label{hm:MC_A3_m2}
\end{equation}

\end{description}
\label{hm:MC_A3}
\end{ass}

Now we can state our basic tool:

\begin{thm}
Under Assumptions \ref{hm:MC_A1} to \ref{hm:MC_A3}, 
the PLDI (\ref{hm:PLDI_def2}) holds for any $M>0$.
\label{hm:thm_PLDI_G}
\end{thm}

It is worth mentioning that Assumptions \ref{hm:MC_A1} to \ref{hm:MC_A3} 
do not refer to any concrete structure of the underlying model. 
We also remark that it is possible to give weaker conditions 
if we want to prove the PLDI for not every but only some specific value of $M$, 
although the resulting conditions are then somewhat more complex to write down.

Theorem \ref{hm:thm_PLDI_G} is due to \cite[Theorems 1 and 3(c)]{Yos11}, its proof being elementary but artful. 
For convenience and completeness, we give a self-contained proof.

\begin{proof}[Theorem \ref{hm:thm_PLDI_G}] 
Taylor's formula applied to (\ref{hm:MC_Z_def}) gives
\begin{equation}
\log\mbbz_{n}(u;\tau)=\Delta_{n}(\tau)[u]-\frac{1}{2}\Gam_{0}(\tau)[u,u]+r_{n}(u;\tau),
\nn
\end{equation}
where
\begin{align}
\Delta_{n}(\tau)&:=a_{n}\p_{\zeta}\mbbh_{n}(\zeta_{0},\tau),
\nn\\
r_{n}(u;\tau)&:=\frac{1}{2}\{\Gam_{0}(\tau)-\Gam_{n}(\theta)\}[u,u]
\nn\\
&{}\qquad
-\int_{0}^{1}(1-s)\int\p_{\zeta}\Gam_{n}(\zeta_{0}+sta_{n}u,\tau)
[sa_{n}u,u^{\otimes 2}]dtds,
\nn
\end{align}
with $\Gam_{n}(\theta):=-b_{n}^{-1}\p_{\zeta}^{2}\mbbh_{n}(\theta)$. 
Without loss of generality, we may and do suppose that
\begin{equation}
\ep_{0}\vee\ep_{1}<\frac{1}{2}
\nonumber
\end{equation}
for the constants $\ep_{0}$ and $\ep_{1}$ given in Assumption \ref{hm:MC_A3}. 
Fix any $M>0$ and $\al\in(0,\ep_{0})$ in what follows. 
Instead of the target region $\{u\in\mbbr^{p_{\zeta}}: |u|\ge r\}$, we will look at the following two separately:
\begin{equation}
U_{n}^{0}(r):=\left\{u: |u|\ge b_{n}^{(1-\al)/2}\right\}, \qquad 
U^{1}_{n}(r):=\left\{u: r\le|u|\le b_{n}^{(1-\al)/2}\right\}.
\nonumber
\end{equation}
To complete the proof, obviously it suffices to focus on $r>0$ and $n$ large enough. 
We will proceed with:
\begin{itemize}
\item Making use of the global identifiability condition on $U_{n}^{0}(r)$; 
\item Direct estimate of the remainder $r_{n}(u;\tau)$ on $U_{n}^{1}(r)$.
\end{itemize}
(The newly introduced threshold ``$b_{n}^{(1-\al)/2}$'' will turn out to work effectively.) 
We will denote by $C$ a generic positive constant possibly varying from line to line.

\medskip

First we look at $\sup_{u\in U_{n}^{0}(r)}\mbbz_{n}(u;\tau)$. 
According to the boundedness of $\Theta_{\zeta}$, the variable $\rho:=a_{n}u$ is bounded: $|\rho|\le C$. 
Then $|u|>b_{n}^{(1-\al)/2}$ implies that
\begin{equation}
|\rho|\ge b_{n}^{-\al/2},
\label{hm:pldi_proof_1}
\end{equation}
and also $|u|\ge r$ does
\begin{equation}
r\le Cb_{n}^{1/2}.
\label{hm:pldi_proof_2}
\end{equation}
Put $\mbby_{n}(\theta)=b_{n}^{-1}\{\mbbh_{n}(\zeta,\tau)-\mbbh_{n}(\zeta_{0},\tau)\}$ 
and $\mbby_{0}(\theta)=-|\mbbs_{0}(\zeta,\tau)|^{2}$. 
Fix any constant $\kappa_{0}\in(1-2\ep_{0},1-2\al)$, 
and observe that by using \eqref{hm:pldi_proof_1} and \eqref{hm:pldi_proof_2} we have
\begin{align}
& P_{0}\left(\sup_{u\in U_{n}^{0}(r)}~\sup_{\tau\in\Theta_{\tau}}
\mbbz_{n}(u;\tau)\ge\exp(-r^{1+\kappa_{0}})\right) \nn\\
&\le P_{0}\left(
\sup_{\rho:~b_{n}^{-\al/2}\le|\rho|\le C}~\sup_{\tau\in\Theta_{\tau}}
\mbby_{n}(\zeta_{0}+\rho,\tau)\ge-r^{1+\kappa_{0}}b_{n}^{-1}\right) \nn\\
&\le P_{0}\left(
\sup_{\theta\in\Theta}\left|b_{n}^{\ep_{0}}\{\mbby_{n}(\theta)-\mbby_{0}(\theta)\}\right|
\ge r^{1+\kappa_{0}}b_{n}^{\ep_{0}-1}\right)
\nn\\
&{}\qquad +P_{0}\left(
\sup_{\rho:~b_{n}^{-\al/2}\le|\rho|\le C}~\sup_{\tau\in\Theta_{\tau}}
\mbby_{0}(\zeta_{0}+\rho,\tau)\ge -2r^{1+\kappa_{0}}b_{n}^{-1}\right)
\nn\\
&\le P_{0}\left(
\sup_{\theta\in\Theta}\left|b_{n}^{\ep_{0}}\{\mbby_{n}(\theta)-\mbby_{0}(\theta)\}\right|
\gtrsim b_{n}^{\ep_{0}-(1-\kappa_{0})/2}\right)
\nn\\
&{}\qquad +P_{0}\left(
\inf_{\rho:~b_{n}^{-\al/2}\le|\rho|\le C}~\inf_{\tau\in\Theta_{\tau}}
(-\mbby_{0}(\zeta_{0}+\rho,\tau))\lesssim b_{n}^{-(1-\kappa_{0})/2}\right).
\label{hm:pldi_proof_5}
\end{align}
Using the estimate
\begin{align}
& \left|b_{n}^{\ep_{0}}(\mbby_{n}(\theta)-\mbby_{0}(\theta))\right| \nn\\
&\le b_{n}^{\ep_{0}-1}\left|a_{n}\mbbs_{n}(\zeta_{0},\tau)\right|^{2}
+\left(\left|\frac{1}{b_{n}}\mbbs_{n}(\theta)\right|+|\mbbs_{0}(\theta)|\right)
\left|b_{n}^{\ep_{0}}\left(\frac{1}{b_{n}}\mbbs_{n}(\theta)-\mbbs_{0}(\theta)\right)\right|,
\nonumber
\end{align}
it is straightforward under the assumptions to deduce
\begin{equation}
\sup_{n\in\mbbn}E_{0}\left(\sup_{\theta\in\Theta}
\left|b_{n}^{\ep_{0}}\{\mbby_{n}(\theta)-\mbby_{0}(\theta)\}\right|^{K}\right)<\infty.
\label{hm:pldi_proof_3}
\end{equation}
Further, under \eqref{hm:pldi_proof_1} we have
\begin{equation}
\inf_{\rho:~b_{n}^{-\al/2}\le|\rho|\le C}~\inf_{\tau\in\Theta_{\tau}}
(-\mbby_{0}(\zeta_{0}+\rho,\tau))\gtrsim\inf_{\rho:~b_{n}^{-\al/2}\le|\rho|\le C}|\rho|^{2}\ge b_{n}^{-\al},
\nn
\end{equation}
from which combined with the present choice of $\kappa_{0}$ 
it follows that the second term on the right-hand side of \eqref{hm:pldi_proof_5} becomes zero for every $n$ large enough. 
Let $M_{0}:=(M/2)\{\ep_{0}-(1-\kappa_{0})/2\}^{-1}$ and note that $b_{n}^{-1/2}\lesssim r^{-1}$. 
Substituting this together with \eqref{hm:pldi_proof_3} into \eqref{hm:pldi_proof_5}, we have
\begin{equation}
P_{0}\left(\sup_{u\in U_{n}^{0}(r)}~\sup_{\tau\in\Theta_{\tau}}
\mbbz_{n}(u;\tau)\ge e^{-r}\right)\lesssim b_{n}^{-M_{0}\{\ep_{0}-(1-\kappa_{0})/2\}}\lesssim r^{-M}
\nonumber
\end{equation}
for every large $n$ and $r$, achieving the desired bound.

\medskip

Now we turn to prove the bound
\begin{equation}
P_{0}\left(\sup_{u\in U_{n}^{1}(r)}~\sup_{\tau\in\Theta_{\tau}}
\mbbz_{n}(u;\tau)\ge e^{-r}\right)\lesssim r^{-M}
\label{hm:pldi_proof_9}
\end{equation}
for every large $n$ and $r$. 
Recalling the definition \eqref{hm:def_constastH}, we have 
$|\p_{\zeta}\Gam_{n}(\theta)|\lesssim
|b_{n}^{-1}\mbbs_{n}(\theta)|\cdot|b_{n}^{-1}\p_{\zeta}^{3}\mbbs_{n}(\theta)|
+|b_{n}^{-1}\p_{\zeta}\mbbs_{n}(\theta)|\cdot|b_{n}^{-1}\p_{\zeta}^{2}\mbbs_{n}(\theta)|$, 
from which
\begin{equation}
\sup_{n\in\mbbn}E_{0}
\left(\sup_{\tau\in\Theta_{\tau}}\left|\p_{\zeta}\Gam_{n}(\zeta_{0},\tau)\right|^{K}
\right)<\infty
\label{hm:pldi_proof_6}
\end{equation}
for every $K>0$. Moreover,
\begin{align}
& \left|b_{n}^{\ep_{1}}\left(\Gam_{n}(\zeta_{0},\tau)-\Gam_{0}(\tau)\right)\right|
\nn\\
&=\bigg|b_{n}^{\ep_{1}}
\left\{2\sum_{k=1}^{p}\left(\frac{1}{b_{n}}\p_{\zeta}\mbbs_{k,n}(\zeta_{0},\tau)\right)^{\otimes 2}-\Gam_{0}(\tau)
\right\}
\nn\\
&{}\qquad+2b_{n}^{\ep_{1}-1/2}\sum_{k=1}^{p}\left\{a_{n}\mbbs_{k,n}(\zeta_{0},\tau)\cdot
\frac{1}{b_{n}}\p_{\zeta}^{2}\mbbs_{k,n}(\zeta_{0},\tau)\right\}\bigg|
\nn\\
&\lesssim
\left(\left|\frac{1}{b_{n}}\p_{\zeta}\mbbs_{n}(\zeta_{0},\tau)\right|+
\sum_{k=1}^{p}\left|\mbbs'_{k,0}(\zeta_{0},\tau)\right|\right)
\!\cdot\!
\sum_{k=1}^{p}
\left|b_{n}^{\ep_{1}}
\left(\frac{1}{b_{n}}\p_{\zeta}\mbbs_{k,n}(\zeta_{0},\tau)-\mbbs'_{k,0}(\zeta_{0},\tau)
\right)\right|
\nn\\
&{}\qquad+
b_{n}^{\ep_{1}-1/2}\left|a_{n}\mbbs_{n}(\zeta_{0},\tau)\right|
\left|\frac{1}{b_{n}}\p_{\zeta}^{2}\mbbs_{n}(\zeta_{0},\tau)\right|.
\nn
\end{align}
This leads to
\begin{equation}
\sup_{n\in\mbbn}E_{0}\left(
\sup_{\tau\in\Theta_{\tau}}
\left|b_{n}^{\ep_{1}}\left(\Gam_{n}(\zeta_{0},\tau)-\Gam_{0}(\tau)\right)\right|^{K}
\right)<\infty.
\label{hm:pldi_proof_7}
\end{equation}
Put $\del=\{\al\wedge(2\ep_{1})\}/(1-\al)$. Then, by the inequality 
$|r_{n}(u;\tau)|\lesssim |\Gam_{n}(\zeta_{0},\tau)-\Gam_{0}(\tau)||u|^{2}
+a_{n}|u|^{3}\sup_{\tau}|\p_{\zeta}\Gam_{n}(\zeta_{0},\tau)|$, 
the following estimate holds whenever $r\le|u|\le b_{n}^{(1-\al)/2}$:
\begin{align}
r^{\del}\frac{|r_{n}(u;\tau)|}{1+|u|^{2}}
&\lesssim(b_{n}^{-\ep_{1}}r^{\del})b_{n}^{\ep_{1}}|\Gam_{n}(\zeta_{0},\tau)-\Gam_{0}(\tau)|
+(a_{n}|u|r^{\del})\sup_{\tau\in\Theta_{\tau}}|\p_{\zeta}\Gam_{n}(\zeta_{0},\tau)|
\nn\\
&\lesssim
b_{n}^{\ep_{1}}|\Gam_{n}(\zeta_{0},\tau)-\Gam_{0}(\tau)|
+\sup_{\tau\in\Theta_{\tau}}|\p_{\zeta}\Gam_{n}(\zeta_{0},\tau)|.
\label{hm:pldi_proof_10}
\end{align}
Pick any $\kappa_{1}\in(1-\del,1)$. 
Markov's inequality for the exponent $M_{1}:=M\{\del-(1-\kappa_{1})\}^{-1}$ 
together with the estimates \eqref{hm:pldi_proof_6}, \eqref{hm:pldi_proof_7} and \eqref{hm:pldi_proof_10} 
leads to
\begin{equation}
P_{0}\left(\sup_{u\in U_{n}^{1}(r)}~\sup_{\tau\in\Theta_{\tau}}
\frac{|r_{n}(u;\tau)|}{1+|u|^{2}}\ge r^{-(1-\kappa_{1})}\right)
\lesssim r^{-M_{1}\{\del-(1-\kappa_{1})\}}=r^{-M}.
\label{hm:pldi_proof_11}
\end{equation}
Moreover, for every $K>0$ H\"older's inequality gives
\begin{equation}
E_{0}\left(\sup_{\tau\in\Theta_{\tau}}|\Del_{n}(\tau)|^{K}\right)
\lesssim
E_{0}\left(\sup_{\tau\in\Theta_{\tau}}
\left|\frac{1}{b_{n}}\p_{\zeta}\mbbs_{n}(\zeta_{0},\tau)\right|^{K}
\left|a_{n}\mbbs_{n}(\zeta_{0},\tau)\right|^{K}\right)\lesssim 1.
\label{hm:pldi_proof_8}
\end{equation}
By \eqref{hm:pldi_proof_11} and \eqref{hm:pldi_proof_8}, for every large $r>0$ 
the left-hand side of \eqref{hm:pldi_proof_9} can be bounded by
\begin{align}
& P_{0}\left\{\sup_{u\in U_{n}^{1}(r)}~\sup_{\tau\in\Theta_{\tau}}
\mbbz_{n}(u;\tau)\ge\exp\left(-\frac{1}{2}r^{1+\kappa_{1}}\right)\right\}
\nn\\
&\lesssim r^{-M}+
P_{0}\Bigg\{\sup_{u\in U_{n}^{1}(r)}~\sup_{\tau\in\Theta_{\tau}}\left(
|\Del_{n}(\tau)||u|-\frac{1}{2}\Gam_{0}(\tau)[u,u]\right.
\nn\\
&{}\qquad\left.
+r^{-(1-\kappa_{1})}(1+|u|^{2})\right)\ge-\frac{r^{1+\kappa_{1}}}{2}\Bigg\}
\nn\\
&\lesssim r^{-M}+
P_{0}\left\{\sup_{u\in U_{n}^{1}(r)}~\sup_{\tau\in\Theta_{\tau}}\left(
|\Del_{n}(\tau)||u|-r^{-(1-\kappa_{1})}|u|^{2}\right)
\ge-\frac{r^{1+\kappa_{1}}}{2}(1+2r^{-2})\right\}
\nn\\
&\le r^{-M}+P_{0}\left(\sup_{\tau\in\Theta_{\tau}}|\Del_{n}(\tau)|\ge 2r^{\kappa_{1}}\right)
\nn\\
&{}\qquad+P_{0}\left\{\left(\sup_{\tau\in\Theta_{\tau}}|\Del_{n}(\tau)|\right)r
-r^{-(1-\kappa_{1})+2}\ge-\frac{r^{1+\kappa_{1}}}{2}(1+2r^{-2})\right\}
\nn\\
&\lesssim r^{-M}+P_{0}\left(\sup_{\tau\in\Theta_{\tau}}|\Del_{n}(\tau)|\gtrsim r^{\kappa_{1}}\right)
\lesssim r^{-M}.
\end{align}
Hence \eqref{hm:pldi_proof_9} follows and we are done.
\qed\end{proof}

\begin{rem}{\rm 
The differentiability of $\theta\mapsto\mbbh_{n}(\theta)$ is not essential for the PLDI. 
For example, we could derive the PLDI for the least-absolute deviation type contrast function 
$\gam\mapsto-\sumj|\dd_{j}X-h_{n}\gam|$ 
for estimating the location parameter $\gam$ of the stable {\lp} $X$ such that 
$\mcl(X_{t})=S_{\beta}(t^{1/\beta}\sig)\ast\del_{\gam t}$ 
based on a high-frequency sampling, which we discussed in Section \ref{hm:sec_sym.slp_me}. 
The maximum point of the contrast function equals the sample median $\hat{\gam}_{n}$ defined by \eqref{hm:sslp_smed}. 
In this case, under appropriate conditions we could follow exactly 
the same line of the proof of \cite[Theorem 2.2]{Mas10}, which made use of \cite[Theorem 3(a)]{Yos11}, 
to conclude that
\begin{equation}
\sup_{n\in\mbbn}E_{0}\left\{\left|\sqrt{n}h_{n}^{1-1/\beta_{0}}(\hat{\gam}_{n}-\gam_{0})\right|^{q}\right\}<\infty
\nonumber
\end{equation}
for every $q>0$.
}\qed\end{rem}

%%%%%

\subsection{Description of a two-step procedure}\label{hm:sec_2step_pldi}

The concrete form of the partition $\theta=(\zeta,\tau)$ is of course model dependent; 
recall that the argument $\tau$ is unnecessary if we have only single rate. 
In this section, returning to \eqref{hm:MC_H_def} we observe how Theorem \ref{hm:thm_PLDI_G} works 
for establishing the $L^{q}(P_{0})$-boundedness of the rescaled $\hat{\theta}_{n}$ under multi-scaling. 
The subsequent argument is essentially due to \cite[Section 5, Proposition 2]{Yos11}.

We focus on the case of two-different rates:
\begin{equation}
\theta=(\theta_{1},\theta_{2})\mapsto\mbbm_{n}(\theta_{1},\theta_{2})
=-|\mbbg_{n}(\theta_{1},\theta_{2})|^{2},
\label{hm:2step_pldi_eq+1}
\end{equation}
where $\theta_{i}\in\Theta_{i}\subset\mbbr^{p_{i}}$, $i=1,2$ ($p_{1}+p_{2}=p$). 
This contrast function is maximized at 
$\hat{\theta}_{n}=(\hat{\theta}_{1,n},\hat{\theta}_{2,n})\in\overline{\Theta}$. 
We set the rate matrix to be
\begin{equation}
A_{n}=\begin{pmatrix}
a_{1n}I_{p_{1}} & 0 \\
0 & a_{2n}I_{p_{2}} \\
\end{pmatrix}
\nonumber
\end{equation}
for some sequences $a_{1n}$ and $a_{2n}$ satisfying that as $n\to\infty$
\begin{equation}
a_{1n}\vee a_{2n}\to 0,\quad \frac{a_{1n}}{a_{2n}}\to 0.
\label{hm:a1-a2}
\end{equation}
The latter condition implies that $\theta_{1,0}$ is estimated more quickly than $\theta_{2,0}$, 
where $\theta_{0}=(\theta_{1,0},\theta_{2,0})$ denotes the true value of $\theta$. 
This setting is in particular relevant when the scaled estimator
\begin{equation}
\left(a_{1n}^{-1}(\hat{\theta}_{1,n}-\theta_{1,0}),~
a_{2n}^{-1}(\hat{\theta}_{2,n}-\theta_{2,0})\right)
\nn%\label{hm:MC_2wc}
\end{equation}
is asymptotically normally distributed, so one may keep this in mind in the rest of this section. 
Furthermore, we split the estimating function as $\mbbg_{n}=(\mbbg^{(1)}_{n},\mbbg^{(2)}_{n})$, 
where $\mbbg_{n}^{(i)}$ is $\mbbr^{p_{i}}$-valued.

With the setup described above, given a specific $\mbbg_{n}$ 
we call for the {\it two-step} application of Theorem \ref{hm:thm_PLDI_G}.

\begin{itemize}
\item In the {\it first} step, we apply Theorem \ref{hm:thm_PLDI_G} 
with setting $\zeta=\theta_{1}$, $\tau=\theta_{2}$, $a_{n}=a_{1n}$ ($b_{1n}:=a_{1n}^{-2}$), 
and $\mbbs_{n}(\zeta,\tau)=\mbbg_{n}^{(1)}(\theta_{1},\theta_{2})$, so that
\begin{equation}
\mbbh_{n}(\zeta,\tau)=-\frac{1}{b_{1n}}|\mbbg_{n}^{(1)}(\theta_{1},\theta_{2})|^{2}.
\nonumber
\end{equation}
Under appropriate conditions we deduce the $L^{q}(P_{0})$-boundedness of
\begin{equation}
\hat{u}_{n}:=a_{1n}^{-1}(\hat{\theta}_{1,n}-\theta_{1,0}).
\label{hm:hat-u}
\end{equation}
This step regards the second component $\theta_{2}$ as a nuisance parameter.

\item In the {\it second} step, having the $L^{q}(P_{0})$-boundedness of $\hat{u}_{n}$, 
we apply Theorem \ref{hm:thm_PLDI_G} with setting $\zeta=\theta_{2}$ and $a_{n}=a_{2n}$ ($b_{2n}:=a_{2n}^{-2}$), 
and $\mbbs_{n}(\zeta,\tau)=\mbbs_{n}(\zeta)=\mbbg_{n}^{(2)}(\hat{\theta}_{1,n},\theta_{2})$, 
\begin{equation}
\mbbh_{n}(\zeta,\tau)=-\frac{1}{b_{2n}}|\mbbg_{n}^{(2)}(\hat{\theta}_{1,n},\theta_{2})|^{2}.
\nonumber
\end{equation}
This is maximized at $\hat{\theta}_{2,n}$ as a function of $\theta_{2}$. 
As before, under appropriate conditions we deduce the $L^{q}(P_{0})$-boundedness of
\begin{equation}
\hat{v}_{n}:=a_{2n}^{-1}(\hat{\theta}_{2,n}-\theta_{2,0}).
\label{hm:hat-v}
\end{equation}
Note that in this step we do not have a nuisance argument $\tau$, 
hence the supremum taken over $\tau$ can be removed from the conditions. 
For checking the moment boundedness, it is convenient to partly utilize the expansion
\begin{align}
\frac{1}{b_{2n}}\mbbg_{n}^{(2)}(\hat{\theta}_{1,n},\theta_{2})
&=a_{2n}\left\{a_{2n}\mbbg_{n}^{(2)}(\theta_{0})\right\}
\nn\\
&{}+a_{1n}\left\{
\int_{0}^{1}\frac{1}{b_{2n}}\p_{\theta_{1}}
\mbbg_{n}^{(2)}\binom{\theta_{1,0}+s(\hat{\theta}_{1,n}-\theta_{1,0})}{\theta_{2,0}+s(\theta_{2}-\theta_{2,0})}ds
\right\}[\hat{u}_{n}]
\nn\\
&{}+\left\{
\int_{0}^{1}\frac{1}{b_{2n}}\p_{\theta_{2}}
\mbbg_{n}^{(2)}\binom{\theta_{1,0}+s(\hat{\theta}_{1,n}-\theta_{1,0})}{\theta_{2,0}+s(\theta_{2}-\theta_{2,0})}ds
\right\}[\theta_{2}-\theta_{2,0}],
\nn
\end{align}
together with the previously obtained $L^{q}(P_{0})$-boundedness of $(\hat{u}_{n})$; 
the three $\{\cdots\}$ terms in the right-hand side should be $O_{p}(1)$.
\end{itemize}
Building on the two-step argument, 
the $L^{q}$-boundedness of $A_{n}^{-1}(\hat{\theta}_{n}-\theta_{0})=(\hat{u}_{n},\hat{v}_{n})$ 
follows from \eqref{hm:hat-u} and \eqref{hm:hat-v}. 
It is straightforward to extend the above procedure to the case where we have more than two different rates. 

The interested reader can refer to \cite[Section 6]{Yos11} for a detailed exposition of 
the two-step argument in estimating a multi-dimensional nonlinear ergodic diffusion observed at high frequency.

\begin{rem}{\rm 
The uniform tail estimate $P_{0}(|A_{n}^{-1}(\hat{\theta}_{n}-\theta_{0})|>r)\lesssim r^{-M}$ 
entails the consistency of $\hat{\theta}_{n}$. 
Concerning our contrast function $\mbbm_{n}$ of the form \eqref{hm:MC_H_def}, 
in order to deduce the asymptotic normality under the conditions of Theorem \ref{hm:thm_PLDI_G} 
it just remains to prove a central limit theorem for 
$\{a_{1n}\mbbg_{n}^{(1)}(\theta_{0}),~a_{2n}\mbbg_{n}^{(2)}(\theta_{0})\}$, 
together with some ``separation'' condition when we have more than or equal to two rates; 
we refer to \cite[Theorem 3.5(b)]{Mas13as} for details in the single-scaling case.
}\qed\end{rem}

%%%%%

\bigskip

Most often, $\mbbg_{n}$ is a sum of independent random functions:
\begin{equation}
\mbbg_{n}(\theta)=\mbbg_{n}(\theta_{1},\theta_{2})=\sumj g_{n}(\dd_{j}X;\theta)=:\sumj g_{nj}(\theta),
\label{hm:pldi_G_form}
\end{equation}
for some measurable function $g_{n}=(g_{k,n})_{k=1}^{p}:\mbbr\times\Theta=\mbbr\times(\Theta_{1}\times\Theta_{2})\to\mbbr^{p}$. 
Under suitable regularity conditions on $\mbbg_{n}$ and the identifiability condition, 
it is more or less routine to verify the assumptions of Theorem \ref{hm:thm_PLDI_G}.

Still, a remark on the uniform moment estimate in Assumptions \ref{hm:MC_A2} and \ref{hm:MC_A3} is in order. 
Suppose that $u$ is continuously differentiable with $u$ and $\p_{\theta}u$ having a continuous extension to 
the compact set $\overline{\Theta}$. Then, it follows from the boundedness and convexity of $\Theta$ that 
$\sup_{\theta\in\overline{\Theta}}|u(\theta)|^{q}
\lesssim\int_{\overline{\Theta}}\{|u(\theta)|^{q}+|\p_{\theta}u(\theta)|^{q}\}d\theta$. 
This is a version of Sobolev-imbedding type integral inequalities (see \cite[Section 1.4]{Ada73}), 
based on which we have for random $u$
\begin{equation}
E_{0}\left(\sup_{\theta\in\overline{\Theta}}|u(\theta)|^{q}\right)
\lesssim \sup_{\theta\in\overline{\Theta}}E_{0}\left(|u(\theta)|^{q}\right)
+\sup_{\theta\in\overline{\Theta}}E_{0}\left(|\p_{\theta}u(\theta)|^{q}\right),
\nonumber
\end{equation}
the upper bound being much easier to handle.

%%%%%
%%%%%

\subsection{Examples}

Let us briefly illustrate application of Theorem \ref{hm:thm_PLDI_G} in situations where 
an asymptotic normality of the form
\begin{equation}
A_{n}^{-1}(\hat{\theta}_{n}-\theta_{0})\cil N_{p}\left(0,\Sig(\theta_{0})\right)
\nonumber
\end{equation}
holds; then, once the PLDI \eqref{hm:PLDI_def2} is derived for a given $M>0$, we have
\begin{equation}
E_{0}\left\{f\left(A_{n}^{-1}(\hat{\theta}_{n}-\theta_{0})\right)\right\}
\to\int f(y)\phi\left(y;0,\Sig(\theta_{0})\right)dy
\nonumber
\end{equation}
for every continuous function $f: \mbbr^{p}\to\mbbr$ such that 
$\limsup_{|u|\to\infty}|u|^{-q}|f(u)|<\infty$ for some $q<M$.

%%%

\subsubsection{Maximum-likelihood estimation}

Theorem \ref{hm:thm_PLDI_G} is applicable to 
likelihood-ratio random fields with non-degenerate asymptotic Fisher information. 
Let us consider the inverse-Gaussian subordinator $X$ such that $\mcl(X_{t})=IG(\del t,\gamma)$ 
treated in Section \ref{hm:sec_ig_subo_uan}. 
Although we know that the MLE is explicitly given by \eqref{hm:ig1}, 
it is not a trivial matter to verify convergence of its moments. 
Put $a_{1n}=1/\sqrt{n}$ and $a_{2n}=1/\sqrt{T_{n}}$. 
Recalling the log-likelihood function $\ell_{n}(\theta)$ of \eqref{hm:ll-ig}, we set 
$\mbbg_{n}(\theta)=(\mbbg_{n}^{(1)}(\theta), \mbbg_{n}^{(2)}(\theta))$ with
\begin{align}
\mbbg_{n}^{(1)}(\theta):=\p_{\del}\ell_{n}(\theta)&=\sumj\left(
\frac{1}{\del}+\gam h_{n}-\frac{\del h_{n}^{2}}{\dd_{j}X}
\right),\nn\\
\mbbg_{n}^{(2)}(\theta):=\p_{\gam}\ell_{n}(\theta)&=\sumj\left(\del h_{n}-\gam \dd_{j}X\right).
\nonumber
\end{align}
We have to verify the moment estimates concerning:
\begin{itemize}
\item $\del\mapsto\mbbg_{n}^{(1)}(\del,\gam)$ and its partial derivatives uniformly in $\gam$ 
in the {\it first} step;
\item $\gam\mapsto\mbbg_{n}^{(2)}(\hat{\del}_{n},\gam)
=\mbbg_{n}^{(2)}(\theta_{0})+X_{T_{n}}(\gam_{0}-\gam)+\hat{u}_{n}\sqrt{nh_{n}^{2}}$ 
with $\hat{u}_{n}:=\sqrt{n}(\hat{\del}_{n}-\del_{0})$ in the {\it second} step.
\end{itemize}
From the expression of $\mbbg_{n}^{(2)}(\hat{\del}_{n},\gam)$, we obviously need $nh_{n}^{2}\lesssim 1$, 
which is also necessary for verifying the moment boundedness 
$\sup_{n}E_{0}\{\sup_{\gam}|n^{-1/2}\mbbg_{n}^{(1)}(\del_{0},\gam)|^{K}\}<\infty$ in the first step 
since the (non-random) second term in the right-hand side of
\begin{equation}
\frac{1}{\sqrt{n}}\mbbg_{n}^{(1)}(\del_{0},\gam)
=\frac{1}{\sqrt{n}}\mbbg_{n}^{(1)}(\theta_{0})+\del_{0}(\gam-\gam_{0})\sqrt{nh_{n}^{2}}
\nonumber
\end{equation}
has to stay bounded uniformly in $\gam$ 
(we can apply Burkholder's inequality for the first term $n^{-1/2}\mbbg_{n}^{(1)}(\theta_{0})$). 
Through the use of the moment bound \eqref{hm:ig_inv_moments} as well as 
the explicit expressions of the moments $E_{\theta}(X_{h}^{k})$ for $k\in\{-2,-1,1,2\}$ 
mentioned in Section \ref{hm:sec_ig_subo_uan}, 
it is straightforward to verify all the conditions in Theorem \ref{hm:thm_PLDI_G}. 
Thus we can follow the two-step PLDI argument described in Section \ref{hm:sec_2step_pldi}.

%%%

\subsubsection{Gaussian quasi-likelihood estimation}

Here is an example where we do have an original smooth contrast function $\mbbm_{n}(\theta_{1},\theta_{2})$ 
but direct setting $\mbbg_{n}=(\p_{\theta_{1}}\mbbm_{n},\p_{\theta_{2}}\mbbm_{n})$ does not work properly.

Let $X$ be given by
\begin{equation}
X_{t}=bt+\sig Z_{t},
\nonumber
\end{equation}
where $Z$ is a non-degenerate {\lp} with jumps satisfying that 
$E(Z_{1})=0$, $E(Z_{1}^{2})=1$ and $E(|Z_{1}|^{q})<\infty$ for every $q>0$, 
and where the parameter of interest is $\theta=(b,\sig^{2})\in\Theta$ with compact $\overline{\Theta}\subset\mbbr\times(0,\infty)$. 
The process $Z$ may have both continuous and jump parts. 
The Gaussian quasi-maximum likelihood estimator is defined to be a maximizer 
$\hat{\theta}_{n}=(\hat{b}_{n}, \hat{\sig}^{2}_{n})$ of
\begin{equation}
\mbbm_{n}(b,\sig^{2}):=-\sumj\left\{\log\sig^{2}+\frac{1}{\sig^{2}h_{n}}(\dd_{j}X-bh_{n})^{2}\right\},
\nonumber
\end{equation}
which stems from the small-time Gaussian approximation
\begin{equation}
\mcl(\dd_{j}X)\approx N_{1}(bh_{n},\sig^{2}h_{n}).
\nn
\end{equation}
Although the approximation is wrong in the presence of jumps, 
the resulting estimator is consistent and asymptotically normal at rate $\sqrt{T_{n}}$, say
\begin{equation}
\hat{u}_{n}:=\sqrt{T_{n}}(\hat{\theta}_{n}-\theta_{0})
=\sqrt{T_{n}}\left(\hat{b}_{n}-b_{0},~\hat{\sig}^{2}_{n}-\sig^{2}_{0}\right)
\cil N_{2}\left(0,
\begin{pmatrix}
V_{bb} & V_{b\sig} \\
V_{b\sig} & V_{\sig\sig}
\end{pmatrix}
\right),
\nonumber
\end{equation}
where $V_{b\sig}\ne 0$ if and only if $\int z^{3}\nu(dz)\ne 0$, where $\nu$ denotes the {\lm} of $X$. 
See \cite[Theorem 2.9]{Mas13as} for details; although the main objective of \cite{Mas13as} 
is a possibly multivariate ergodic diffusion with jumps, 
it is trivial that its main results remain valid even for {\lp es}.

In the present setting, $Z$ is a sum of a constant (possibly zero) multiple of a standard Wiener process 
and a pure-jump {\lp} $J$ with mean zero. By means of \cite[Lemma 3.1]{AsmRos01} we have
\begin{equation}
\lim_{h\to 0}\frac{1}{h}E_{\theta_{0}}(|X_{h}|^{q})
=\left\{
\begin{array}{ll}
\ds{c+\int |z|^{2}\nu(dz)}, &\quad q=2, \\
\ds{\int |z|^{q}\nu(dz)}, &\quad q>2,
\end{array}\right.
\label{hm:pldi_mm_1}
\end{equation}
where $c\ge 0$ denotes the Gaussian variance of $X$; 
in particular, $\sup_{h>0}h^{-1}E(|J_{h}|^{q})<\infty$ for $q\ge 2$. 
We also note that
\begin{equation}
\lim_{h\to 0}\frac{1}{h}E_{\theta_{0}}(X_{h}^{k})=\int z^{k}\nu(dz)
\nonumber
\end{equation}
for $k\ge 3$ being an integer, which is easier to derive 
(through the differentiation of the characteristic function of $\mcl(X_{h})$). 
Then, by setting
\begin{equation}
\mbbg_{n}=\left(\p_{b}\mbbm_{n},~h_{n}\p_{\sig^{2}}\mbbm_{n}\right),
\label{hm:pldi_gqmle_eq1}
\end{equation}
and by making use of \eqref{hm:pldi_mm_1} repeatedly for the moment estimates, 
we can apply Theorem \ref{hm:thm_PLDI_G} to derive the PLDI concerning $\hat{u}_{n}$. 
The situation here is entirely different from the case of Gaussian $X$, 
where $(\sqrt{n}(\hat{\sig}^{2}_{n}-\sig^{2}_{0}),~\sqrt{T_{n}}(\hat{b}_{n}-b_{0}))$ is asymptotically normal. 
It is the property \eqref{hm:pldi_mm_1} that slows down the speed of estimating $\sig$; 
\eqref{hm:pldi_mm_1} implies that $(h_{n}^{-1/2}X_{h_{n}})$ is no longer $L^{q}(P_{0})$ bounded for $q>2$ 
(See Section \ref{hm:sec_lslp} for a related remark). 
This is reflected by the factor ``$h_{n}$'' in front of $\p_{\sig^{2}}\mbbm_{n}$ in \eqref{hm:pldi_gqmle_eq1}. 

The Gaussian quasi-maximum likelihood estimator is too naive to estimate the Gaussian variance and {\lm} separately. 
Nevertheless, the estimator is easy-to-use and may exhibit unexpectedly good finite-sample performance 
if $\mcl(Z_{1})$ is ``distributionally'' close to the normal; see the simulations in \cite{Mas13as}.

%%%

\subsubsection{Method of moments at single rate $\sqrt{T_{n}}$}

Typically, the method of moments \cite[Chapter 4]{vdV98} based on the law of large numbers
\begin{equation}
\frac{1}{T_{n}}\sumj g(\dd_{j}X)\cip m(\theta_{0};g)
\label{hm:pldi_mm_0}
\end{equation}
leads to an asymptotically normally distributed estimator at rate $\sqrt{T_{n}}$ for all components, 
where the non-random function $g:\mbbr\to\mbbr^{p}$ 
is to be chosen so as to make the limit $m(\theta;g)$ non-trivial function of $\theta$; 
note that the method of moments for the stable {\lp es} treated in Section \ref{hm:sec_sym.slp_me} 
does not fit \eqref{hm:pldi_mm_0}. 
The convergence \eqref{hm:pldi_mm_0} suggests the estimating equation
\begin{equation}
\mbbg_{n}(\theta):=\sumj g(\dd_{j}X)-T_{n}m(\theta;g)=0.
\label{hm:pldi_mm_add2}
\end{equation}
As seen in the gamma-subordinator case (Remark \ref{hm:rem1_gamma_uan}), 
naive moment fitting may entail information loss with slower rate of convergence. 
Nevertheless, this procedure would still deserve to be considered 
if $m(\theta;g)$ can be an explicit or numerically tractable.

In the present setting, the moment bounds involving the non-random derivatives 
$\p_{\theta}^{k}\mbbg_{n}(\theta)=-\p_{\theta}^{k}m(\theta;g)$, $k\ge 1$, are easy to verify. 
Let us make a few comments on the verification of 
$\sup_{n}E_{0}\{|T_{n}^{-1/2}\mbbg_{n}(\theta_{0})|^{K}\}<\infty$ required in Assumption \ref{hm:MC_A2}. 
We have
\begin{align}
\frac{1}{\sqrt{T_{n}}}\mbbg_{n}(\theta_{0})
&=\frac{1}{\sqrt{T_{n}}}\sumj\left(g(\dd_{j}X)-E_{0}\{g(X_{h_{n}})\}\right)
\nn\\
&{}\qquad
+\sqrt{T_{n}}\left(\frac{1}{h_{n}}E_{0}\{g(X_{h_{n}})\}-m(\theta_{0};g)\right),
\nonumber
\end{align}
which is to be $L^{q}(P_{0})$-bounded for any $q>0$. 
Obviously, under appropriate integrability conditions 
the first term in the right-hand side is $L^{q}(P_{0})$-bounded for every $q>0$. 
We need to be a little more careful in verifying
\begin{equation}
\sup_{n\in\mbbn}\left|
\sqrt{T_{n}}\left(\frac{1}{h_{n}}E_{0}\{g(X_{h_{n}})\}-m(\theta_{0};g)\right)
\right|<\infty.
\label{hm:pldi_mm_2}
\end{equation}
This may not be quite obvious for general $g$, 
but it is possible to provide a simple sufficient condition when $g$ is smooth enough.

To be specific, we suppose that
\begin{equation}
X_{t}=bt+\sqrt{c}w_{t}+J_{t}
\nonumber
\end{equation}
for a standard Wiener process $w$ and a pure-jump {\lp} $J$ 
such that $E(J_{1})=0$ and $E(|J_{1}|^{q})<\infty$ for every $q>0$. 
We may and do set $g(0)=0$, and suppose that $g$ is smooth. 
The extended infinitesimal generator of $X$ (under $P_{\theta_{0}}$) takes the form
\begin{equation}
\mca_{\theta_{0}}g(x)=b\p g(x)+\frac{1}{2}c\p^{2}g(x)
+\int\left\{g(x+z)-g(x)-\p g(x)z\right\}\nu(dz),
\label{hm:pldi_mm_4}
\end{equation}
where we implicitly suppose that the integral exist for each $x$. 
Applying It\^o's formula twice, we get the expression of the form
\begin{equation}
g(X_{h_{n}})=h_{n}\mca_{\theta_{0}}g(0)
+\int_{0}^{h_{n}}\!\!\int_{0}^{s}\mca_{\theta_{0}}^{2}g(X_{u})duds+M_{\theta_{0},h_{n}}.
\label{hm:pldi_mm_add1}
\end{equation}
Here we can deduce that $E_{0}(M_{\theta_{0},h_{n}})=0$ and 
$\sup_{t\le 1}E_{0}(|\mca_{\theta_{0}}^{3}g(X_{t})|)<\infty$ under appropriate integrability conditions, 
and in that case it follows from \eqref{hm:pldi_mm_add1} that
\begin{equation}
\left|\frac{1}{h_{n}}E_{0}\{g(X_{h_{n}})\}-\mca_{\theta_{0}}g(0)\right|\lesssim h_{n}.
\label{hm:pldi_mm_3}
\end{equation}
Having \eqref{hm:pldi_mm_3} in hand, the condition \eqref{hm:pldi_mm_2} holds 
with
\begin{equation}
m(\theta_{0};g)=\mca_{\theta_{0}}g(0)
\nonumber
\end{equation}
if $nh_{n}^{3}\lesssim 1$. 
For example, we have $m(\theta_{0};g)=\int z^{k}\nu(dz)$ for the choice $g(x)=x^{k}$, $k\ge 3$, 
which is explicit in case of the tempered stable {\lp es}; see the references cited in Section \ref{hm:sec_lslp}. 

\medskip

It is obvious that instead of \eqref{hm:pldi_mm_add2} we could more generally consider
\begin{equation}
\mbbg_{n}^{[k]}(\theta):=\sumj g(\dd_{j}X)
-T_{n}\sum_{l=1}^{k}\frac{h_{n}^{l-1}}{l!}\mca_{\theta}^{l}g(0),
\nonumber
\end{equation}
that is, we could utilize the higher-order It\^o-Taylor expansion to make the estimating function 
more closer to the genuine martingale estimating function:
\begin{equation}
\sumj\left(g(\dd_{j}X)-E_{\theta}\{g(X_{h_{n}})\}\right)
=\mbbg_{n}^{[k]}(\theta)+O_{p}(nh_{n}^{k+1}).
\nonumber
\end{equation}
Then, since
\begin{equation}
\left|\frac{1}{h_{n}}E_{0}\{g(X_{h_{n}})\}-\sum_{l=1}^{k}\frac{h_{n}^{l-1}}{l!}\mca^{l}_{\theta_{0}}g(0)\right|\lesssim h_{n}^{k}
\nonumber
\end{equation}
we can put a weaker condition on the decreasing rate of $h_{n}$, 
in compensation for a more complicated form of the estimating function.

\begin{rem}{\rm 
We refer to \cite[Section 3]{Jac07} for a detailed asymptotics for power-variation statistics, 
where, in particular, laws of large numbers for $T_{n}^{-1}\sumj g(\dd_{j}X)$ and 
central limit theorems for $\sqrt{T_{n}}\{T_{n}^{-1}\sumj g(\dd_{j}X)-h_{n}^{-1}E_{\theta_{0}}\{g(X_{h_{n}})\}\}$ 
have been derived for certain classes of $g$. When making use of them for our estimation problem, 
we do not need to verify \eqref{hm:pldi_mm_2} if $E_{\theta}\{g(X_{h_{n}})\}$ is explicit as a function of $\theta$, 
while unfortunately this is not often the case. Hence we have resorted to the approximation procedure, 
presupposing that the quantities $\mca_{\theta}^{k}g(0)$ are explicit.
}\qed\end{rem}

The paper \cite{Fig11} studied a non-parametric estimation problem of the functional parameter
\begin{equation}
\beta(\vp):=\int\vp(z)\nu(dz)
\nonumber
\end{equation}
under $T_{n}\to\infty$ for a random time change of a {\lp} $Y$ with {\lm} $\nu$ 
and for a measurable function $\vp$ such that the integral $\beta(\vp)$ is well-defined. 
There, the author suggests using the natural statistics
\begin{equation}
\hat{\beta}_{n}(\vp):=\frac{1}{T_{n}}\sumj\vp(\dd_{j}Y),
\nonumber
\end{equation}
and provides sets of conditions under which the estimator is asymptotically normal at rate $\sqrt{T_{n}}$. 
To deduce it, as in \eqref{hm:pldi_mm_2} we need information about the rate of convergence of 
$h_{n}^{-1}E_{0}\{\vp(X_{h_{n}})\}$ to its limit; see \cite[Assumption 3]{Fig11}. 
If in particular $Y$ is a {\lp} with no drift and no Gaussian component 
and if $\vp$ is smooth enough with $\vp(0)=\p\vp(0)=0$, 
then it follows from \eqref{hm:pldi_mm_4} and \eqref{hm:pldi_mm_3} that 
\begin{equation}
\left|\frac{1}{h_{n}}E\{\vp(X_{h_{n}})\}-\beta(\vp)\right|\lesssim h_{n}.
\nonumber
\end{equation}
See also \cite{Fig08-2} for a related discussion.

%%%%%
%%%%%

\section{Concluding remarks}\label{hm:sec_cr}

%% How to cite the other chapters:
%Mark clearly within your file where the references to the other chapters should be and describe it in the following way: 
%``Please see section <sectionnumber> in chapter <chaptertitle> by <chapterauthor> in this book. 
%Or simply please see chapter <chaptertitle> by >chapterauthor> in this book''. 
%Springer will then insert the correct DOI of the reference for the reference linking.

In this chapter we have mainly discussed parametric estimation of jump-type {\lp es} observed at high-frequency. 
Our primary interest is in explicit case studies, with special attention on the stable {\lp es}. 
That said, the specific {\lp es} treated here seems to have shown that 
possible asymptotic phenomena in small time are of wide variety. 

We close this chapter with mentioning a few related topics, which we did not touch in this chapter.

%%%%%

\runinhead{Fourier-transform based methods.} 
Estimation methodologies based on the empirical characteristic function or Laplace transform 
are quite popular in non-parametric estimation of {\lp es}. 
It is certainly relevant for parametric situation too. 
For the state-of-the-art of this research area, 
the interested reader can refer to the other chapters in this book: 
see \tcm{the chapter by Comte and Genon-Catalot} for non-parametric adaptive estimation method 
with regularization under high-frequency and long-term sampling, 
and \tcm{the chapter by Belomestny and {Rei\ss}} for the Fourier method under low-frequency sampling. 
See also \cite{TodTau12} for a related issue about the realized Laplace transform.

%%%%%

\runinhead{On model building}

One may want to do some statistical test about presence of the Gaussian and jump parts: 
is the underlying process continuous, pure-jump, or both? 
Many kinds of test statistics for this do exist, most of which are based on the multipower variation 
with or without threshold (jump-detection filter) carving up the increments of the underlying process.  
The recent development of this research area has its root in financial econometrics. 
We refer the interested reader to \cite{AitJac12} for a nice overview of many recent results in this direction; 
see also \cite{ConMan11}. 
Further, we refer to \cite{JinKonLiu13} for testing presence of the Gaussian component, 
namely, a test to support pure-jump models. 
These analyses are only utilizing small-time structure of the model, 
and applicable to a broad class of general It\^o processes. 

%%%%%

\runinhead{Threshold estimation}

Obviously, the coexistence of the Gaussian part and the jump part makes the parametric estimation problem 
much more difficult and cumbersome when, for example, 
trying to estimate the both parts separately via likelihood-based method. 
Researches in this direction basically build on threshold estimation to judge whether or not a jump occurred in 
each small-time interval $(t^{n}_{j-1},t^{n}_{j}]$; 
in small-time scale, big-size (resp. small-size) increments should come from a big jump 
(resp. Gaussian fluctuation and/or small jumps). 
We refer to \cite{Man04}, \cite{OgiYos11}, and \cite{ShiYos06} for theoretical results 
concerning diffusion processes with finite intensity of jumps. 

We should note that asymptotic theory normally does not tell us how to select a threshold in finite sample. 
Indeed, the selection is in general a difficult practical problem. 
As was exemplified in \cite{Shi08} through simulations, 
a naive choice of the threshold may severely deteriorate estimation performance. 

Simultaneous estimation of all the elements of the generating triplet was studied by \cite{GegSta10}, 
which may be seen as a refinement of the classical result \cite{RubTuc59}. 
Although their result are not asymptotically optimal, 
they looked at both finite- and infinite-activity cases and also discussed data-driven choice of the threshold.

%%%%%
%%%%%

%\bigskip

\begin{acknowledgement}
I extend my thanks to Professor Jean Jacod and the anonymous referee for their detailed suggestions and comments, 
which not only brought some errors in the first draft to my attention 
but also led to substantial improvement in the exposition of this chapter. 
I am grateful to Professor Claudia Kl\"uppelberg for her encouragement. 
My thanks also go to Sangji Kim, Yuma Uehara, Shoichi Eguchi, and Yusuke Shimizu for proofreading. 
Needless to say, all remaining errors are of my own. 
Some materials of this chapter are based on the joint papers with Dr. Reiichiro Kawai, 
to whom I thank for fruitful discussions during the works. 

This work was partly supported by JSPS KAKENHI Grant Numbers 23740082, 26400204.
\end{acknowledgement}

%%%%%% Begin of the Appendix%%%%%%%%

%\begin{appendix}
%
%\section{...}
%
%\end{appendix}

%%%%%%%%%%%%%%%%%%%%References%%%%%%%%%%%%%%%%%%%%%%%%%%%%%%%%%%%%%

\addcontentsline{toc}{chapter}{References}
\bibliography{Hiroki_bibs.bib}
\bibliographystyle{BibStyleLM}

%%%%%%%%%%%%%%%%%%%%Index%%%%%%%%%%%%%%%%%%%%%%%%%%%%%%%%%%%%%%%%%%%

\addcontentsline{toc}{chapter}{Index}
%\printindex

\endgroup

\end{document}